\documentclass{amsart}
\pdfoutput=1
\usepackage{float}
\floatstyle{boxed}
\restylefloat{figure}
\usepackage{mathtools}
\usepackage{amssymb}
\usepackage{mathrsfs}
\usepackage[all]{xy}
\UseComputerModernTips
\CompileMatrices
\usepackage{graphicx} 
\usepackage{hyperref}
\usepackage{ifthen}
\usepackage{latexsym}
\allowdisplaybreaks
\numberwithin{equation}{section}
\newtheorem{theorem}[equation]{Theorem}
\newtheorem{lemma}[equation]{Lemma}
\newtheorem{proposition}[equation]{Proposition}
\newtheorem{corollary}[equation]{Corollary}
\theoremstyle{definition}
\newtheorem{definition}[equation]{Definition}

\theoremstyle{remark}
\newtheorem{remark}[equation]{Remark}

\renewcommand{\phi}{\varphi}
\DeclareMathSymbol{\boxprod}{\mathbin}{AMSa}{"03} 
\DeclareMathSymbol{\mixprod}{\mathbin}{AMSa}{"4F} 

\newcommand{\convto}{\Rightarrow}
\newcommand{\dirsum}{\oplus}
\newcommand{\Dirsum}{\bigoplus}
\newcommand{\disjunion}{\sqcup}
\newcommand{\Disjunion}{\coprod}

\newcommand{\hmtpc}{\simeq}

\newcommand{\includesin}{\hookrightarrow}
\newcommand{\intersect}{\cap}
\newcommand{\iso}{\cong}

\newcommand{\Mackey}[1]{\overline{#1}\vphantom{#1}}

\newcommand{\smsh}{\wedge}

\newcommand{\susp}{\Sigma}
\newcommand{\tensor}{\otimes}

\newcommand{\Union}{\bigcup}

\newcommand{\C}{{\mathbb C}}
\newcommand{\D}{{\mathscr D}}

\newcommand{\LL}{{\Lambda}}			

\newcommand{\MM}{{\mathrm M}}					

\newcommand{\R}{{\mathbb R}}

\newcommand{\Trep}{{\mathbb T}}
\newcommand{\U}{{\mathscr U}}
\newcommand{\Urep}{{\mathbb U}}
\newcommand{\V}{{\mathscr V}}
\newcommand{\Vrep}{\mathbb{V}}

\newcommand{\Wrep}{\mathbb{W}}

\newcommand{\Xrep}{\mathbb{X}}

\newcommand{\Yrep}{\mathbb{Y}}

\newcommand{\Z}{\mathbb{Z}}
\newcommand{\Vinit}{\V^i}			
\newcommand{\Vlow}{\V^\ell}			
\newcommand{\vRO}{v\mathscr{R}^iO}	
\newcommand{\CP}{\C P}
\newcommand{\tE}{\tilde E}
\newcommand{\ev}{\text{ev}}
\newcommand{\Ab}{\text{\textit{Ab}}}
\newcommand{\orb}[1]{{\mathscr{O}_{#1}}}
\newcommand{\sorb}[1]{{\widehat{\mathscr{O}}_{#1}}}
\newcommand{\conc}[1]{\langle #1 \rangle}
\DeclareMathOperator*{\colim}{colim}
\DeclareMathOperator{\Hom}{Hom}
\DeclareMathOperator{\Ext}{Ext}
\DeclareMathOperator{\Pic}{Pic}
\newcommand{\tensorS}{\tensor_{\Mackey H_G^\bullet(S^0)}}
\begin{document}
\title{The $\Z/p$ ordinary cohomology of $B_GU(1)$}

\author{Steven R. Costenoble}
\address{Department of Mathematics\\Hofstra University\\
   Hempstead, NY 11549}
\email{Steven.R.Costenoble@Hofstra.edu}


\date{\today}

\abstract
With $G = \Z/p$, $p$ prime, we calculate the ordinary $G$-cohomology (with Burnside ring coefficients) of 
$\CP_G^\infty = B_GU(1)$,
the complex projective space, a model for the classifying space for $G$-equivariant
complex line bundles. The $RO(G)$-graded ordinary cohomology was
calculated by Gaunce Lewis, but here we extend to a larger grading
in order to capture a more natural set of generators, including the Euler class
of the canonical bundle, as well as a significantly simpler set of relations.
\endabstract
\maketitle
\tableofcontents


\section*{Introduction}

It is hard to overstate the importance of characteristic classes of bundles, 
and the related characteristic numbers
of smooth manifolds, in nonequivariant algebraic topology.
Equivariantly, however, a good theory of characteristic classes has been lacking,
hence so have the techniques and calculations they would permit.
One reason for this problem has been the lack of a good ordinary cohomology theory
in which characteristic classes could live.
Nonequivariantly, the Euler class of an $n$-dimensional vector
bundle lies in the $n$th cohomology group. Equivariantly, what should we call the
dimension of a vector bundle? The integer dimension is inadequate
to get results such as a Thom isomorphism theorem in Bredon cohomology.
In the highly restricted case where all the fibers of the bundle are modeled on a single
representation of the ambient group, that representation can be used as the dimension
if we extend to $RO(G)$-graded cohomology.
But the vector bundles we encounter are rarely that simple.

It was with this in mind that Stefan Waner and I wrote \cite{CW:ordinaryhomology},
defining and exploring an ordinary cohomology theory with a grading expanded
beyond $RO(G)$, in which there are natural gradings to use as the dimensions
of arbitrary vector bundles. This theory does possess a Thom isomorphism theorem
for any vector bundle. In particular, it allows us to define Euler clases of
arbitrary vector bundles, suggesting the possibility of developing a good theory
of characteristic classes.

Of course, characteristic classes are best viewed as elements in the cohomology
of classifying spaces of vector bundles.Various
attempts have been made to calculate the cohomology rings of some equivariant
classifying spaces: 
See, for example, \cite{Le:projectivespaces}, \cite{Kro:SerreSS}, and \cite{Dug:Grassmannians}.
In particular, in \cite{Le:projectivespaces}, Guance Lewis gave calculations of the
$RO(G)$-graded cohomologies of complex projective spaces, including the classifying
space $B_GU(1)$, for $G = \Z/p$ with $p$ any prime.
But, because the canonical line bundle over $B_GU(1)$ has fibers modeled on all
the possible one-dimensional complex representations of $G$, there is no obvious candidate
for the Euler class, or the first Chern class.
Moreover, as we shall see, the calculation, particularly for odd $p$, is
much more complicated than it needs to be because of the restriction to $RO(G)$ grading.
In effect, these calculations saw only a small slice of a structure that
can be better described if we allow ourselves to use the larger grading
developed in \cite{CW:ordinaryhomology}.

The goal of this paper, then, is to calculate the equivariant ordinary cohomology
groups of $B_GU(1)$ for $G = \Z/p$, $p$ prime, in the expanded grading
defined in  \cite{CW:ordinaryhomology}. The first part of the paper reviews
some necessary background material and describes the ordinary cohomology theory
we use. The equivariant cohomology of a $G$-space is a module over the
$RO(G)$-graded equivariant cohomology of a point, which is highly nontrivial
away from the integer-graded part. The equivariant cohomology of a point was first
calculated by Stong, in an unpublished manuscript, and first published by
Lewis in \cite{Le:projectivespaces}. We summarize the calculation
in \S4, where we also give some other calculations we need, the cohomologies
of $EG$ and $\tE G$; part of our calculation of the cohomology of $B_GU(1)$
is based on the cofibration sequence $EG_+\to S^0 \to \tE G$.

Part 2 of the paper gives the calculation of the cohomology of $B_GU(1)$
with coefficients in the Burnside ring Mackey functor.
The main results are Theorem~\ref{thm:oddadditivestructure},
which shows that the cohomology of $B_GU(1)$ is a free module over the cohomology
of a point, and gives an explicit basis,
and Theorem~\ref{thm:multstructure}, which gives generators and relations
for the cohomology of $B_GU(1)$ as an algebra over the cohomology a a point.
Among the generators is the Euler class of the canonical bundle, as expected,
but there are a good many other interesting classes as well.
We give a brief comparison of our results with Lewis's
and also describe the result when using other coefficient systems,
including constant $\Z$ coefficients.

Part 3 returns to some general results about ordinary cohomology,
starting with a discussion of the ``decatorification'' process necessary
to consider equivariant cohomology as graded on a group, rather than as a functor
on a category of representations. This is necessary to make sure we get signs right.
The bulk of Part~3 gives another calculation of the cohomology of a point.
This might seem superfluous, given \cite{Le:projectivespaces},
but we need calculations of the cohomologies of $EG$ and $\tE G$, which do not
appear in the form we need in the literature, and once those are calculated it is a relatively simple
step to use them to calculate the cohomology of a point, so it seemed worthwhile to include
the complete calculation here. We also need to know how the three
cohomologies are related for our main calculation.
This calculation is similar to the ``Tate approach'' described by
John Greenlees in \cite{Green:fourapproaches}.
Although he pleads for consistency in notation in describing $RO(G)$-graded
theories in the case $G = \Z/2$, his suggested notation
would not work well for us here.

This work owes a large debt to Gaunce Lewis, of course.
The calculations he did and the methods he developed were the basis for a lot
of what appears here. I would like to think that he would have enjoyed this
paper, if he hadn't written it himself, first.
This paper is also founded on joint work with Peter May on equivariant orientation theory,
and on my long and continuing collaboration with Stefan Waner that produced
\cite{CW:ordinaryhomology} among many other results.

\part{Equivariant ordinary cohomology}\label{part:cohomology}

\section{The representation ring and the Burnside ring}

Throughout this paper $G$ will denote a cyclic group $\Z/p$ of prime order,
with generator $t$.
We begin by introducing notations for some representations of $G$.

\begin{definition}\label{def:irrRepresentations}
\hspace{2em}
\begin{enumerate}
\item
If $G = \Z/2$, let $\LL$ denote its nontrivial irreducible real representation,
that is, $\R$ with $G$ acting by multiplication by $-1$.
We fix once and for all a nonequivariant identification $\R = \LL$.
We also write $\MM_1 = \LL^2$ and hence fix a nonequivariant identification of
$\MM_1$ with $\R^2$, to be consistent with the next definition.

\item
If $G = \Z/p$ with $p$ odd, let $\MM_k$ denote the real representation of $G$
with underlying space $\R^2 = \C$, on which the generator $t$ of $G$ acts by
multiplication by $e^{2\pi ik/p}$.
(In parallel with the use of $\Lambda$ above, you should think of $\MM$ as
the capital Greek letter ``mu,'' rather than the Latin letter ``em.'')
The nontrivial irreducible real representations of $G$ are given by
$\MM_k$, $1\leq k \leq (p-1)/2$.
We fix once and for all nonequivariant identifications $\R^2 = \C = \MM_k$.

\item
If $G = \Z/p$ with $p$ odd, we can consider $\MM_k$ for any integer $k$,
defined as above. With this definition, 
$\MM_0 = \R^2$ has trivial action by $G$, and
$\MM_k \iso \MM_{p-k}$ as real representations.
Although this gives us no new representations of $G$, 
the identifications $\R^2 = \MM_k$ give $\MM_k$ and $\MM_{p-k}$
opposite nonequivariant orientations, which will be a useful
distinction to be able to make.

\item
If $G = \Z/p$ for any prime $p$, let $\C_k$ denote the complex representation of $G$
with underlying space $\C$ on which $t$ acts by multiplication by
$e^{2\pi ik/p}$.
The nontrivial irreducible complex representations of $G$ are given by
$\C_k$, $1\leq k \leq p-1$.
By definition, there are fixed nonequivariant identifications of the $\C_k$ with $\C$.

\end{enumerate}
\end{definition}

We fix an identification $\C = \R^2$, which fixes identifications
$\C_k = \MM_k$ as real representations.

Recall that the {\em (real) representation ring} $RO(G)$ of  $G$ is the Grothendieck group
on the monoid of isomorphism classes of finite-dimensional real representations of $G$
under direct sum,
with multiplication given by tensor products.
It is the (additive) group structure on $RO(G)$ that matters most to us here.
For any group, $RO(G)$ is the free abelian group with one generator for each irreducible real
representation of $G$. In particular, $RO(\Z/2)$ is free abelian on two generators, which
we call $1$ and $\LL$. For odd $p$, $RO(\Z/p)$ is free abelian on $(p+1)/2$ generators,
$1$, $\MM_1$, \dots, $\MM_{(p-1)/2}$.

\begin{remark}
In \cite{Le:projectivespaces}, when $p$ is odd, Lewis grades his cohomology on $RSO(G)$, the
ring of oriented real representations, in order to simplify his calculations somewhat.
In $RSO(G)$, $\MM_k$ and $\MM_{p-k}$ give distinct generators.
We will do all our calculations using $RO(G)$ grading, but will have
to pay careful attention to signs.
\end{remark}

\begin{definition}\label{def:ro0G}
If $\alpha\in RO(G)$, let $|\alpha|\in \Z$ denote the dimension of $\alpha$
and let $\alpha^G\in\Z$ denote the dimension of its fixed set.
Let
\begin{align*}
 I^\ev(G) &= \{ \alpha\in RO(G) \mid |\alpha| = 0 \text{ and $\alpha^G$ even} \} \\
 RO_0(G) &= \{ \alpha\in RO(G) \mid |\alpha| = \alpha^G = 0 \} \\
 RO_+(G) &= \Big\{ \sum_k n_k(\MM_k-2) \in RO(G) \mid n_k \geq 0 \ \forall k \Big\}.
\end{align*}
\end{definition}

Note that $I^\ev(G)$ is the subgroup generated by the $\MM_k-2$.
When $p$ is odd, $|\alpha| = 0$ implies that $\alpha^G$ is even, so
$I^\ev(G) = I(G)$ is the usual augmentation ideal, but when $p=2$,
$I^\ev(G)$ is a strict subgroup of $I(G)$.
The subgroup $RO_0(G)$ was used in
\cite{tDP:geomodules} and \cite{Le:projectivespaces}.
Note that $RO_0(G) = 0$ for $p=2$ or $p=3$, but for $p>3$, $RO_0(G)$ is a
free abelian group on the elements $\MM_k - \MM_1$ for $2\leq k \leq (p-1)/2$.

Also recall that $A(G)$, the {\em Burnside ring} of $G$, 
is the Grothendieck ring on the monoid
of isomorphism classes of finite $G$-spaces and disjoint union, with multiplication given by Cartesian product.
Segal \cite{Seg:equivariantstable} showed that, for finite $G$, this definition is equivalent to saying that $A(G)$ is
the ring of stable $G$-maps from $S^0$ to itself.
When $G=\Z/p$, $A(G)$ is the free abelian group on two generators, $1 = [G/G]$ and $g = [G/e]$,
with multiplication given by $g^2 = pg$. 
We let $\kappa = p - g$, so $\kappa^2 = p\kappa$ and
the elements $1$ and $\kappa$ form another basis of $A(G)$.
We have the {\em augmentation map} $\epsilon\colon A(G) \to \Z$, 
the ring map given by forgetting the $G$-action
and counting the (signed) number of points. On elements, we have $\epsilon(1) = 1$, $\epsilon(g) = p$,
and $\epsilon(\kappa) = 0$.

For $G=\Z/p$, elements of $A(G)$ are characterized by the augmentation map and the fixed-point map,
that is, by their signed number of points and their signed number of fixed points.
Viewing elements of $A(G)$ as stable maps $f\colon S^V\to S^V$, these correspond
to the nonequivariant degree of $f$ and the nonequivariant degree of $f^G$, respectively,
which allows us to explicitly identify maps of spheres with elements of $A(G)$.
For example, when $p=2$, negation on $\LL$ defines a map $f\colon S^\LL\to S^\LL$ of
nonequivariant degree $-1$ with $f^G$ of degree 1 (being the identity map on $S^0$).
Thus, $f$ must represent the element $1-g\in A(G)$, because $\epsilon(1-g) = -1$ and
$(1-g)^G = 1$. Notice that $(1-g)^2 = 1$, so $1-g$ is a unit in $A(\Z/2)$.

\section{Mackey functors}\label{sec:MackeyFunctors}

We will view equivariant cohomology as Mackey functor--valued, so we review some basic
facts about such functors.

\subsection*{Definition and examples}

\begin{definition}
Let $\orb{G}$ denote the orbit category of $G$ and let $\sorb{G}$ denote the stable orbit category,
i.e., the category of orbits of $G$ and stable $G$-maps between them.
\end{definition}

In the case of $\Z/p$, the orbit and stable orbit categories each have two objects, $G/G$ and $G/e$.
We picture the maps as follows:
\[
 \xymatrix{
  G/G \\
  G/e \ar[u]^\rho \ar@(dl,dr)[]_t \\
  \orb{G}
 }
\qquad\qquad
 \xymatrix{
  G/G \ar@(ur,ul)[]_{A(G)} \ar@/^/[d]^{\tau} \\
  G/e \ar@/^/[u]^{\rho} \ar@(dl,dr)[]_{\Z[G]} \\
  \sorb{G}
 }
\]
That is, in the stable orbit category, the ring of self maps of $G/G$ is
$A(G)$ while the ring of self maps of $G/e$ is isomorphic to the group ring $\Z[G] \iso \Z[t]/\langle t^p \rangle$.
The group of maps $G/e\to G/G$ is free abelian on the projection $\rho$ while the group
of maps $G/G\to G/e$ is free abelian on the transfer map $\tau$.
We have $\rho \circ \tau = g$
and $\tau \circ \rho = N = \sum_{k=0}^{p-1} t^k$.
Finally, $\rho t = \rho$, $t\tau = \tau$, $g\rho = p\rho$, and $\tau g = p\tau$.

\begin{definition}
A {\em Mackey functor} is a contravariant additive functor from the stable orbit category to the
category of abelian groups.
\end{definition}

If $\Mackey T$ is a Mackey functor, we will generally picture $\Mackey T$ using a diagram of the following form:
\[
 \xymatrix{
  \Mackey T(G/G) \ar@/_/[d]_{\rho} \\
  \Mackey T(G/e) \ar@/_/[u]_{\tau} \ar@(dl,dr)[]_{t^*}
 }
\]
Here, $\rho$ and $\tau$ are the maps induced by the maps of the same name in $\sorb G$.
$\Mackey T(G/G)$ should be a module over the Burnside ring; the action is specified by this diagram
because the action of $g$ is given by $\tau\circ\rho$.

We now review and give names to the Mackey functors that will appear in our calculations, beginning
with the following two:
\[
 \Mackey A_{G/G} = \sorb{G}(-,G/G)\colon \xymatrix{
		A(G) \ar@/_/[d]_{\epsilon} \\
		\Z \ar@/_/[u]_{\cdot g} \ar@(dl,dr)[]_{1}
	}
\qquad\qquad
 \Mackey A_{G/e} = \sorb{G}(-,G/e)\colon \xymatrix{
		\Z \ar@/_/[d]_{\cdot N} \\
		\Z[G] \ar@/_/[u]_{\epsilon} \ar@(dl,dr)[]_{\cdot t}
	  }
\]
We call $\Mackey A_{G/G}$ the {\em Burnside ring Mackey functor}.
In $\Mackey A_{G/e}$, $N = \sum_{k=0}^{p-1}t^k$ as above, and
$\epsilon\colon\Z[G]\to\Z$ is
\[
 \epsilon\Big(\sum_{k=0}^{p-1} a_k t^k\Big) = \sum_{k=0}^{p-1} a_k.
\]
$\Mackey A_{G/G}$ and $\Mackey A_{G/e}$, being represented functors,
are both projective, with
\[
 \Hom(\Mackey A_{G/G}, \Mackey T) \iso \Mackey T(G/G) \qquad\text{and}\qquad
 \Hom(\Mackey A_{G/e}, \Mackey T) \iso \Mackey T(G/e).
\]
Related to these are the following functors, where $d\in\Z$:
\[
 \Mackey A[d]\colon
 \xymatrix{
	\Z\dirsum\Z \ar@/_/[d]_{\left(\begin{smallmatrix} d & p \end{smallmatrix}\right)} \\
	\Z \ar@/_/[u]_{\left(\begin{smallmatrix} 0 \\ 1 \end{smallmatrix}\right)} \ar@(dl,dr)[]_{1}
 }
\]
Here, we are thinking of elements of $\Z\dirsum\Z$ as column vectors, so that the descending vertical map
takes $\left(\begin{smallmatrix} a \\ b \end{smallmatrix}\right) \mapsto da + pb$.
$\Mackey A[0]$ is a direct sum of two simpler functors while
$\Mackey A[1] \iso \Mackey A_{G/G}$. Also, $\Mackey A[d_1] \iso \Mackey A[d_2]$ if and only if
$d_1 \equiv \pm d_2 \pmod p$.
Thus, when $p=2$, there are, up to isomorphism, two such functors, and when
$p$ is odd there are $(p+1)/2$ (but, see Proposition~\ref{prop:extensions}).
Note that these functors are of interest really only when $p > 3$.

The next two are examples of a general construction $\conc C$
for any abelian group $C$, but these are the two cases that will occur in our calculations:
\[
 \conc\Z \colon \xymatrix{
		\Z \ar@/_/[d] \\
		0 \ar@/_/[u] \ar@(dl,dr)[]_{}
	   }
\qquad\qquad
 \conc{\Z/p}\colon \xymatrix{
			\Z/p \ar@/_/[d] \\
			0 \ar@/_/[u] \ar@(dl,dr)[]
	     }
\]

For the last group of Mackey functors, consider the forgetful functor from Mackey functors to
$\Z[G]$-modules that takes $\Mackey T \mapsto \Mackey T(G/e)$.
This functor has both a left and a right adjoint. The left adjoint $\Mackey L$ is defined by
\[
 \Mackey L U\colon
 \xymatrix{
		\Z\tensor_{\Z[G]} U \ar@/_/[d]_{N\tensor 1} \\
		\Z[G]\tensor_{\Z[G]}U = U. \ar@/_/[u]_{\epsilon\tensor 1} \ar@(dl,dr)[]_{t\tensor 1}
 }
\]
The right adjoint $\Mackey R$ is defined by
\[
 \Mackey R U\colon
 \xymatrix{
		\Hom_{\Z[G]}(\Z, U) \ar@/_/[d]_{\epsilon^*} \\
		\Hom_{\Z[G]}(\Z[G], U) = U. \ar@/_/[u]_{N^*} \ar@(dl,dr)[]_{t^*}
 }
\]
The particular cases that will occur in our calculations use $U = \Z$ with trivial
$\Z[G]$ action, or, in the case $p=2$, $U=\Z_-$, on which $t\in\Z[G]$ acts as $-1$.
These give us the following four Mackey functors.
\[
\begin{array}{rcrc}
 \Mackey L\Z\colon & \xymatrix{
		\Z \ar@/_/[d]_{p} \\
		\Z \ar@/_/[u]_{1} \ar@(dl,dr)[]_{1}
	}
&\qquad\qquad
 \Mackey L\Z_{-}\colon & \xymatrix{
		\Z/2 \ar@/_/[d]_{0} \\
		\Z \ar@/_/[u]_{\pi} \ar@(dl,dr)[]_{-1}
	 }
\\ \\
 \Mackey R\Z\colon & \xymatrix{
		\Z \ar@/_/[d]_{1} \\
		\Z \ar@/_/[u]_{p} \ar@(dl,dr)[]_{1}
	}
&\qquad\qquad
 \Mackey R\Z_{-}\colon & \xymatrix{
		0 \ar@/_/[d] \\
		\Z \ar@/_/[u] \ar@(dl,dr)[]_{-1}
	 }
\end{array}
\]

\subsection*{Multiplicative structures}

Mackey functors can have multiplicative pairings, based on the box product $\boxprod$.
The box product itself is described in \cite{Le:projectivespaces};
for us it suffices to know that a map $\Mackey S\boxprod \Mackey T \to \Mackey U$
is equivalent to a pair of maps
\begin{align*}
 \Mackey S(G/G)\tensor \Mackey T(G/G) &\to \Mackey U(G/G) \quad\text{and} \\
 \Mackey S(G/e)\tensor \Mackey T(G/e) &\to \Mackey U(G/e)
\end{align*}
satisfying the following conditions, where we write
$xy$ for the image of $x\tensor y$ under the appropriate one of these maps:
\begin{align*}
 t(xy) &= (tx)(ty), \\
 \rho(xy) &= \rho(x)\rho(y), \\
 \tau(x\rho(y)) &= \tau(x)y, \quad\text{and} \\
 \tau(\rho(x)y) &= x\tau(y).
\end{align*}
The last two conditions are called the {\em Frobenius relations.}
By convention, if $x\in \Mackey S(G/G)$ and $y\in\Mackey T(G/e)$,
we will often write $xy$ for $\rho(x)y \in \Mackey U(G/e)$.

The functor $\Mackey A_{G/G}$ has a self-pairing 
$\Mackey A_{G/G}\boxprod \Mackey A_{G/G}\to \Mackey A_{G/G}$ using the usual
ring structures on $A(G)$ and $\Z$.
A {\em unital ring} is a Mackey functor $\Mackey T$ with an associative and unital
pairing $\Mackey T\boxprod \Mackey T\to \Mackey T$,
where the unit is given by a map $\Mackey A_{G/G}\to \Mackey T$.
(Here, we use that $\Mackey A_{G/G}$ is the unit for $\boxprod$,
meaning that $\Mackey A_{G/G}\boxprod\Mackey T \iso \Mackey T$ for
any Mackey functor $\Mackey T$.)
The conditions above say that this is equivalent to $\Mackey T(G/G)$ being
a unital ring, $\Mackey T(G/e)$ being a unital ring (with the action of $t$ being a ring map), 
$\rho\colon \Mackey T(G/G)\to \Mackey T(G/e)$ being a ring map,
and $\tau\colon\Mackey T(G/e)\to\Mackey T(G/G)$ being a left and right
$\Mackey T(G/G)$-module map.
Clearly, $\Mackey A_{G/G}$ is itself a unital ring.
Every Mackey functor is a {\em module} over $\Mackey A_{G/G}$ in the obvious sense.

$\Mackey R\Z$ is a unital ring with the usual ring structure on $\Z$.
A module over $\Mackey R\Z$ is precisely a Mackey functor such that $\tau\circ\rho$
is multiplication by $p$.
$\Mackey A_{G/e}$, $\Mackey L\Z$, and $\conc{\Z/p}$ are modules over $\Mackey R\Z$ and, if $p=2$, so are
$\Mackey R\Z_-$ and $\Mackey L\Z_-$.
$\Mackey R\Z$ and $\Mackey A_{G/e}$ are projective $\Mackey R\Z$-modules.

\subsection*{Generators and relations}

We will want to describe the results of our calculations in terms of generators and relations.
When doing so, we identify elements of $\Mackey T(G/G)$ with maps $\Mackey A_{G/G}\to \Mackey T$
and elements of $\Mackey T(G/e)$ with maps $\Mackey A_{G/e}\to \Mackey T$.
For example, we can say that $\Mackey R\Z$ is generated by an element $\xi$ at level $G/G$
subject to the relation $\kappa\xi = 0$. By this we mean that 
the following sequence is exact:
\[
 \Mackey A_{G/G} \xrightarrow{\kappa} \Mackey A_{G/G} \xrightarrow{\epsilon} \Mackey R\Z \to 0.
\]
Here, $\kappa$ is the map 
corresponding to $\kappa\in \Mackey A_{G/G}(G/G)$ and
$\epsilon$ is the map
corresponding to $1\in \Mackey R\Z(G/G)$, which we are also calling $\xi$.

Generators may occur at either level $G/G$ or level $G/e$, and similarly for relations.
Here are descriptions of the other examples in terms of generators and relations:

\begin{itemize}
\item $\conc\Z$: Generated by an element $e$ at level $G/G$ subject to $\rho(e) = 0$.
That is, there is an exact sequence
\[
 \Mackey A_{G/e} \to \Mackey A_{G/G} \to \conc\Z \to 0
\]
where the first map is specified at level $G/e$ by $\epsilon\colon\Z[G]\to\Z$.

\item $\conc{\Z/p}$: Generated by an element $e$ at level $G/G$ subject to
$\rho(e) = 0$ and $pe = 0$.

\item $\Mackey L\Z$: Generated by an element $\iota$ at level $G/e$
such that $t\iota = \iota$.

\item $\Mackey L\Z_-$ ($p=2$): Generated by an element $\iota$ at level $G/e$
such that $t\iota = -\iota$.

\item $\Mackey R\Z_-$ ($p=2$): Generated by an element $\iota$ at level $G/e$
such that $\tau\iota = 0$.

\end{itemize}

We noted that several of these Mackey functors are modules over the ring $\Mackey R\Z$.
We can describe modules over $\Mackey R\Z$ in terms of generators and relations as well,
where a generator at level $G/G$ gives a copy of $\Mackey R\Z$ while a generator
at level $G/e$ gives a copy of $\Mackey A_{G/e}$.
The modules $\Mackey L\Z$, $\Mackey L\Z_-$, and $\Mackey R\Z_-$ are described
by the same generators and relations as above.
However, we can simplify the description of $\conc{\Z/p}$:
As an $\Mackey R\Z$-module, it is generated by an element $e$ at level $G/G$
such that $\rho(e) = 0$.

The functor $\Mackey A[d]$ is an interesting, if annoying, case.
It can be described as generated (as a Mackey functor)
by two elements, $\mu$ at level $G/G$ and $\iota$ at level $G/e$,
such that $\rho(\mu) = d\iota$ and $t\iota = \iota$.
(Note that we can, of course, simplify this description if $d\equiv 1 \pmod p$.)
$\mu$ and $\tau(\iota)$ then form a $\Z$-basis for $\Mackey A[d](G/G)$.
This description depends on the particular integer $d$.
Suppose that $d' \equiv d \pmod p$ and let
\[
 \mu' = \mu + \frac{d'-d}{p}\tau(\iota),
\]
then $\mu'$ and $\iota$ generate $\Mackey A[d]$ with
$\rho(\mu') = d'\iota$.
Similarly, $-\mu$ and $\iota$ generate, with $\rho(-\mu) = -d\iota$.
So, the generator $\mu$ is determined only up to sign and a multiple of $\tau(\iota)$,
although it is determined by a particular choice of $d\in\Z$ in the appropriate congruence
classes modulo $p$, if we insist that $\rho(\mu) = d\iota$.

There is another basis for $\Mackey A[d]$ that is useful in our calculations.
Let $d^{-1}\in\Z$ denote an integer such that $d\cdot d^{-1} \equiv 1 \pmod p$
and let $q = (1 - dd^{-1})/p$, so that
$dd^{-1} + qp = 1.$ Then the matrix
\[
 \left(\begin{matrix}d^{-1}&p\\q&-d\end{matrix}\right)
\]
is invertible over $\Z$.
Note that
\[
 \kappa\mu = p\mu - d\tau(\iota).
\]
Moreover, this element is independent of the choice of $d$ in its congruence class modulo $p$
because $\kappa\tau(\iota) = 0$: If $\mu' = \mu + \frac{1}{p}(d'-d)\tau(\iota)$, then
\[
 \kappa\mu' = \kappa\mu + \frac{d'-d}{p}\kappa\tau(\iota) = \kappa\mu.
\]
So if we let 
\[
 \lambda = d^{-1}\mu + q\tau(\iota),
\]
then $\{\lambda, \kappa\mu\}$ is another basis of $\Mackey A[d](G/G)$,
with $\rho(\lambda) = \iota$ and $\rho(\kappa\mu) = 0$.
It follows that $\lambda$ and $\kappa\mu$ generate $\Mackey A[d]$, subject to the relations
$\rho(\kappa\mu) = 0$ and $\kappa\lambda = d^{-1}\kappa\mu$.
We can write our original basis as
\begin{align*}
 \mu &= d\lambda + \frac{1-dd^{-1}}{p}\kappa\mu \quad\text{and} \\
 \tau(\iota) &= p\lambda - d^{-1}\kappa\mu.
\end{align*}
Using the basis $\{\lambda,\kappa\mu\}$, the functor $\Mackey A[d]$ can be displayed as:
\[
 \Mackey A[d]\colon
 \xymatrix{
	\Z\dirsum\Z \ar@/_/[d]_{\left(\begin{smallmatrix} 1 & 0 \end{smallmatrix}\right)} \\
	\Z \ar@/_/[u]_{\left(\begin{smallmatrix} p \\ -d^{-1} \end{smallmatrix}\right)} \ar@(dl,dr)[]_{1}
 }
\]

\subsection*{Extensions}

Finally, a word about extension problems that will show up
in Part~\ref{part:point} when we are calculating the cohomology of a point.

\begin{proposition}\label{prop:extensions}
There are $p$ distinct extensions of $\Mackey R\Z$ by $\conc\Z$
(up to the equivalence used to define $\Ext^1$), all having the form
\[
 0 \to \conc\Z \to \Mackey A[d] \to \Mackey R\Z \to 0,
\]
where the map $\conc\Z\to\Mackey A[d]$ takes $1$ to $\kappa\mu$ and the
map $\Mackey A[d]\to\Mackey R\Z$ takes $\lambda$ to $1$ and $\kappa\mu$ to $0$.
Further, if we have a commutative diagram of the following form with exact rows:
\[
 \xymatrix{
  0 \ar[r] & \conc\Z \ar[r] \ar@{=}[d] & \Mackey A[d] \ar[r] \ar[d]
    & \Mackey R\Z \ar[r] \ar[d]^{\cdot m} & 0 \\
  0 \ar[r] & \conc\Z \ar[r] & \Mackey T \ar[r]
    & \Mackey R\Z \ar[r] & 0 
 }
\]
then $\Mackey T\iso \Mackey A[md]$.
\end{proposition}

\begin{proof}
We begin by computing $\Ext^1(\Mackey R\Z, \conc\Z)$ to determine the number of possible extensions.
We have the short exact sequence
\[
 0 \to \conc\Z \to \Mackey A_{G/G} \to \Mackey R\Z \to 0.
\]
Here, the map $\conc\Z\to\Mackey A_{G/G}$ takes $1 \mapsto \kappa$ while
$\Mackey A_{G/G}\to \Mackey R\Z$ is $\epsilon$ at level $G/G$.
Because $\Mackey A_{G/G}$ is projective, we can calculate the Ext group as the cokernel in
\[
 \xymatrix@R-1.5ex{
 \Hom(\Mackey A_{G/G},\conc\Z) \ar[r] \ar@{=}[d]
   & \Hom(\conc\Z, \conc\Z) \ar[r] \ar@{=}[d]
   & \Ext^1(\Mackey R\Z, \conc\Z) \ar[r] & 0 \\
 \Z \ar[r]^p & \Z
 }
\]
Hence, $\Ext^1(\Mackey R\Z, \conc\Z) \iso \Z/p$ and there are exactly $p$ distinct extensions.
(As usual, this means distinct up to isomorphisms that are the identity on $\conc\Z$ and
$\Mackey R\Z$.)

For every $d\in\Z$, we have a short exact sequence
\[
 0 \to \conc\Z \to \Mackey A[d] \to \Mackey R\Z \to 0
\]
given by the following diagram, where we use the basis $\{\lambda,\kappa\mu\}$ for $\Mackey A[d]$:
\[
 \xymatrix@C+1.5em{
   0 \ar[r]
    & \Z \ar@/_/[d] \ar[r]^-{\left(\begin{smallmatrix}0\\1\end{smallmatrix}\right)}
     & \Z\dirsum\Z \ar@/_/[d]_{(\begin{smallmatrix}1&0\end{smallmatrix})}
           \ar[r]^-{(\begin{smallmatrix}1&0\end{smallmatrix})}
      & \Z \ar@/_/[d]_{1} \ar[r]
       & 0
\\
   0 \ar[r]
    & 0 \ar@/_/[u] \ar[r]
     & \Z \ar@/_/[u]_{\left(\begin{smallmatrix}p\\-d^{-1}\end{smallmatrix}\right)} \ar[r]_1
      & \Z \ar@/_/[u]_{p} \ar[r]
       & 0
 }
\]
If we have a diagram of the form
\[
 \xymatrix{
  0 \ar[r] & \conc\Z \ar[r] \ar@{=}[d] & \Mackey A[d_1] \ar[r] \ar[d]
    & \Mackey R\Z \ar[r] \ar@{=}[d] & 0 \\
  0 \ar[r] & \conc\Z \ar[r] & \Mackey A[d_2] \ar[r]
    & \Mackey R\Z \ar[r] & 0 
 }
\]
then the map $\Mackey A[d_1](G/e)\to \Mackey A[d_2](G/e)$ may be taken to be the identity and
it is straightforward to show that the map $\Mackey A[d_1](G/G)\to \Mackey A[d_2](G/G)$
must then be given by a matrix of the form $\left(\begin{smallmatrix}1&0\\q&1\end{smallmatrix}\right)$
where $d_1^{-1} = d_2^{-1} + qp$, hence $d_1\equiv d_2 \pmod p$.
Therefore, there are $p$ distinct such extensions, one for each congruence class modulo $p$,
and they account for all the possible extensions.

Finally, if we have a diagram of the form
\[
 \xymatrix{
  0 \ar[r] & \conc\Z \ar[r] \ar@{=}[d] & \Mackey A[d] \ar[r] \ar[d]
    & \Mackey R\Z \ar[r] \ar[d]^{\cdot m} & 0 \\
  0 \ar[r] & \conc\Z \ar[r] & \Mackey T \ar[r]
    & \Mackey R\Z \ar[r] & 0 
 }
\]
then we know that $\Mackey T\iso \Mackey A[d_2]$ for some $d_2\in\Z$.
A similar argument to the one above shows that
the map $\Mackey A[d]\to \Mackey A[d_2]$ must be given by a matrix of the form
$\left(\begin{smallmatrix}m&0\\q&1\end{smallmatrix}\right)$ where
$d^{-1} = md_2^{-1} + qp$, hence $d_2 \equiv md \pmod p$.
\end{proof}

\begin{proposition}\label{prop:extensions2}
All extensions of $\conc{\Z/p}$ by $\conc\Z$ have the form
\[
 0 \to \conc\Z \xrightarrow{i_2} \conc{\Z/p}\dirsum\conc\Z \xrightarrow{\pi_1}
  \conc{\Z/p} \to 0
\]
or
\[
 0 \to \conc\Z \xrightarrow{p} \conc\Z \xrightarrow{d} \conc{\Z/p} \to 0
\]
where $d\not\equiv 0 \pmod p$.
Further, if we have a commutative diagram of the form
\[
 \xymatrix{
  0 \ar[r] & \conc\Z \ar[r] \ar@{=}[d] & \Mackey A[d] \ar[r] \ar[d]
    & \Mackey R\Z \ar[r] \ar[d]^{\pi} & 0 \\
  0 \ar[r] & \conc\Z \ar[r] & \Mackey T \ar[r]
    & \conc{\Z/p} \ar[r] & 0 
 }
\]
where $\pi$ is the natural projection, then $\Mackey T\iso \conc{\Z/p}\dirsum\conc\Z$
if $d \equiv 0 \pmod p$, otherwise the diagram must take the form
\[
 \xymatrix{
  0 \ar[r] & \conc\Z \ar[r] \ar@{=}[d] & \Mackey A[d] \ar[r] \ar[d]^\eta
    & \Mackey R\Z \ar[r] \ar[d]^{\pi} & 0 \\
  0 \ar[r] & \conc\Z \ar[r]_p & \conc\Z \ar[r]_d
    & \conc{\Z/p} \ar[r] & 0 
 }
\]
where $\eta(\mu) = 1$ and $\eta(\tau(\iota)) = 0$.
\end{proposition}

\begin{proof}
The statement about the possible extensions of $\conc{\Z/p}$ by $\conc\Z$ follows
from the similar statement about the extensions of the group
$\Z/p$ by $\Z$ (where $\Ext(\Z/p,\Z) \iso \Z/p$).

If we have a diagram of the form
\[
 \xymatrix{
  0 \ar[r] & \conc\Z \ar[r] \ar@{=}[d] & \Mackey A[d] \ar[r] \ar[d]
    & \Mackey R\Z \ar[r] \ar[d]^{\pi} & 0 \\
  0 \ar[r] & \conc\Z \ar[r] & \Mackey T \ar[r]
    & \conc{\Z/p} \ar[r] & 0 
 }
\]
then $\Mackey A[d]$ must be the pullback of the right-hand square. Running through the $p$
possibilities for the bottom row shows that the pullback establishes an isomorphism
$\Ext^1(\conc{\Z/p},\conc\Z) \iso \Ext^1(\Mackey R\Z,\conc\Z)$
given explicitly as in the statement of the proposition.
\end{proof}

\section{Equivariant ordinary cohomology}\label{sec:genCohomology}

In \cite{CW:ordinaryhomology}, Stefan Waner and the author gave a detailed exposition of equivariant
ordinary cohomology graded on ``representations of the fundamental groupoid.''
In this section we review some of the basic definitions and properties.
We assume that $G$ is a finite group throughout, though
\cite{CW:ordinaryhomology} is written in the more general context of compact Lie groups.

\subsection*{The equivariant fundamental groupoid and its representations}

When $X$ is a $G$-space, we have the following definition,
given originally by tom Dieck \cite{tD:transfgroups}
and used extensively in \cite{CMW:orientation} and \cite{CW:ordinaryhomology}.

\begin{definition}
The {\em equivariant fundamental groupoid of $X$}, denoted $\Pi_G X$, is the category
whose objects are the $G$-maps $x\colon G/H\to X$ for all the orbits $G/H$ of $G$, 
and whose maps from
$x$ to $y\colon G/K\to X$ are pairs $(\omega, \alpha)$, where
$\alpha\colon G/H\to G/K$ is a $G$-map and $\omega$ is a $G$-homotopy class of paths,
rel endpoints, from $x$ to $y\circ\alpha$.
Composition is induced by composition of maps of orbits and the usual composition
of path classes.
\end{definition}

$\Pi_G$ is a 2-functor, taking $G$-maps of spaces to functors
and homotopies of $G$-maps to natural isomorphisms.
There is an evident functor $\pi\colon \Pi_G X\to \orb G$, with
$\pi(x\colon G/H\to X) = G/H$ and $\pi(\omega,\alpha) = \alpha$.
This makes $\Pi_G X$ a {\em bundle of groupoids} over $\orb G$
in the language of \cite{CMW:orientation}.

\begin{definition}
For $n$ an integer, let $v\V_G(n)$ denote the category of 
{\em virtual $n$-dimensional orthogonal bundles over orbits.}
Its objects are pairs $(G\times_H V, G\times_H W)$, where $V$ and $W$ are
representations of $H$ with $|V|-|W| = n$; we use the notation $G\times_H(V\ominus W)$
for such an object and think of it as a formal difference of bundles.
A morphism is a virtual map $G\times_H(V_1\ominus W_1) \to G\times_K(V_2\ominus W_2)$,
given by a $G$-map $\alpha\colon G/H\to G/K$ and an equivalence class of pairs of bundle maps
\begin{align*}
 \phi\colon G\times_H(V_1\dirsum Z_1) &\to G\times_K(V_2\dirsum Z_2) \\
 \psi\colon G\times_H(W_1\dirsum Z_1) &\to G\times_K(W_2\dirsum Z_2)
\end{align*}
over $\alpha$, where $Z_1$ is a representation of $H$ and $Z_2$ is a representation of $K$.
Two such pairs of maps are considered equivalent if they are stably $G$-homotopic
through orthogonal bundle maps over $\alpha$.
Here, stabilization is given by addition
to both $\phi$ and $\psi$ of the same bundle map $G\times_H U_1\to G\times_K U_2$
over $\alpha$, and we also allow replacement of $Z_1$ and $Z_2$ by isomorphic representations,
meaning that, if $f\colon Z_1\to Z'_1$ and $g\colon Z_2\to Z'_2$ are $H$- and $K$-isomorphisms,
respectively, then the pair $(\phi,\psi)$ is equivalent to $(\phi',\psi')$, where
the latter are determined by the fact that the following diagrams commute:
\[
 \xymatrix{
   G\times_H(V_1\dirsum Z_1) \ar[r]^\phi \ar[d]_f
     & G\times_K(V_2\dirsum Z_2) \ar[d]^g \\
  G\times_H(V_1\dirsum Z'_1) \ar[r]^{\phi'}
     & G\times_K(V_2\dirsum Z'_2)
  }
\]
and
\[
 \xymatrix{
   G\times_H(W_1\dirsum Z_1) \ar[r]^\psi \ar[d]_f
     & G\times_K(W_2\dirsum Z_2) \ar[d]^g \\
  G\times_H(V_1\dirsum Z'_1) \ar[r]^{\psi'}
     & G\times_K(W_2\dirsum Z'_2)
  }
\]
We write a typical map as $(\alpha,[\phi\ominus\psi])$.

We let $v\V_G$ be the (disjoint) union of the categories $v\V_G(n)$ for all $n$.
\end{definition}

It is not necessarily obvious from the definition above that the collection of morphisms
between two object is (only) a set, but
\cite[\S1.3]{CW:ordinaryhomology} gives an equivalent definition that makes it
clear.

We have a functor $\pi\colon v\V_G\to \orb G$, given by
$\pi(G\times_H(V\ominus W)) = G/H$ and
$\pi((\alpha,[\phi\ominus\psi])) = \alpha$.
This makes $v\V_G$ and each $v\V_G(n)$ into a bundle of groupoids over $\orb G$.

\begin{definition}
A {\em virtual $n$-dimensional orthogonal representation of $\Pi_G X$} is a functor
$\Pi_G X\to v\V_G(n)$ over $\orb G$.
These form the {\em category of virtual $n$-dimensional representations of $\Pi_G X$}
when we take morphisms to be natural isomorphisms.
\end{definition}

For each representation $V$ of $G$, there is a representation of $\Pi_G X$
we denote by $\Vrep$, given by taking each $x\colon G/H\to X$ to $G/H\times V$
and each $(\omega,\alpha)$ to $\alpha\times 1$.
This generalizes to virtual representations of $G$ to give
representations $\Vrep\ominus \Wrep$.

Less trivially, if $\xi\colon E\to X$ is a $G$-vector bundle over $X$, there is an associated
representation of $\Pi_G X$, denoted $\xi^*$, given by
$\xi^*(x\colon G/H\to X) = x^*(\xi)$.

\begin{definition}
The {\em orthogonal representation ring of $\Pi_G X$}, denoted $RO(\Pi_G X)$,
is the ring whose elements are the isomorphism classes of
the virtual orthogonal representations of $\Pi_G X$ of all dimensions.
Addition is given by direct sum of bundles and multiplication by tensor product.
\end{definition}

In particular, $RO(\Pi_G(*)) \iso RO(G)$ when $*$ denotes the one-point $G$-space.
In general, the map $X\to *$ induces a map $RO(G)\to RO(\Pi_G X)$, taking $V\ominus W$
to $\Vrep \ominus \Wrep$ as above.
We will not use the multiplication on $RO(\Pi_G X)$ in this paper, just the additive structure.

\subsection*{Ordinary cohomology}

Let $B$ be a $G$-space and let $(X,q,\sigma)$ be an ex-space over $B$, which is to say that
$q\colon X\to B$ is a $G$-map and $\sigma\colon B\to X$ is a section of $q$.
Suppose also given a virtual representation $\gamma$ of $\Pi_G B$ and a Mackey functor
$\Mackey T$. We can then define a group $\tilde H_G^\gamma(X;\Mackey T)$,
a contravariant functor of
$X$ and a covariant functor of $\gamma$ and of $\Mackey T$.

We will generally consider the reduced theory as written above. If we have
any parametrized space $q\colon X\to B$, rather than an ex-space, we can form the ex-space
$(X,q)_+$, which we will write $X_+$, given by
$(X\disjunion B, q\disjunion 1, \sigma)$, where
$\sigma\colon B\to X\disjunion B$ is the evident inclusion.
In particular, we may consider $\tilde H_G^\gamma(B_+;\Mackey T)$.

The collection of these groups as $\gamma$ varies gives what we call the
{\em $RO(\Pi_G B)$-graded ordinary cohomology} of ex-spaces over $B$.
(In \cite{CW:ordinaryhomology} we allow more general coefficient systems than just
Mackey functors, but in this paper we will stick to the simpler case.)
In particular, when we restrict the virtual representation $\gamma$ to
be of the form $\Vrep\ominus\Wrep$, the resulting groups are exactly the
$RO(G)$-graded ordinary cohomology of $X$
discussed in \cite{LMM:roghomology}, \cite{May:alaska}, and \cite{CW:ordinaryhomology},
which generalizes Bredon's integer-graded theory.
In particular, this theory obeys a dimension axiom that takes the following form:
For $x\colon G/H\to B$ and integers $n$, we have
\[
 \tilde H_G^n(G/H_+;\Mackey T) \iso
  \begin{cases}
    \Mackey T(G/H) & \text{if $n=0$} \\
    0 & \text{if $n\neq 0$,}
  \end{cases}
\]
naturally in $G/H$.

The $RO(\Pi_G B)$-graded theory has many nice properties, discussed in detail
in \cite{CW:ordinaryhomology}. One of the main reasons for introducing the enlarged
grading is to get general Thom isomorphism and Poincar\'e duality theorems.
In particular, the Thom isomorphism (\cite[3.11.3]{CW:ordinaryhomology}) takes the following form:
If $\xi\colon E\to X$ is a $G$-vector bundle, let
$T(\xi)$ denote the ex-space over $X$ given by taking the fiberwise one-point
compactification of each fiber, with the section given by the compactification points.
A {\em Thom class} for $\xi$ is a class $t\in \tilde H_G^{\xi^*}(T(\xi);\Mackey A_{G/G})$
such that, for every $x\colon G/K\to X$, the element
\begin{align*}
 x^*(t) &\in \tilde H_G^{\xi^*}(T(x^*\xi);\Mackey A_{G/G}) \\
  &\iso \tilde H_G^{x^*\xi^*}(G\times_K S^V;\Mackey A_{G/G}) \\
  &\iso \tilde H_K^{V}(S^V;\Mackey A_{K/K}) \\
  &\iso A(K)
\end{align*}
is a generator. Here, $V$ is the representation of $K$ such that $x^*\xi = G\times_K V$.

\begin{theorem}[Thom Isomorphism]
If $\xi\colon E\to X$ is a $G$-vector bundle, then there exists
a Thom class $t\in\tilde H_G^{\xi^*}(T(\xi);\Mackey A_{G/G})$.
For every Thom class $t$, the map
\[
 t\cup - \colon \tilde H_G^\gamma(X_+;\Mackey T) \to \tilde H_G^{\gamma+\xi^*}(T(\xi);\Mackey T)
\]
is an isomorphism for every representation $\gamma$ and Mackey functor $\Mackey T$.
\qed
\end{theorem}

As usual, given a Thom class $t(\xi)$ for a $G$-bundle $\xi$, we define
the corresponding {\em Euler class} $e(\xi) \in \tilde H_G^{\xi^*}(X_+;\Mackey A_{G/G})$
to be the restriction of $t(\xi)$ along the zero section $X_+\to T(\xi)$.

For computational purposes, it is useful to view cohomology not as group valued,
but as Mackey functor valued. We let
$\Mackey H_G^\gamma(X;\Mackey T)$ be the Mackey functor defined by
\[
 \Mackey H_G^\gamma(X;\Mackey T)(G/K) = \tilde H_G^\gamma(G/K_+\smsh_B X; \Mackey T).
\]
Because $\tilde H_G^*(-)$ is stable (it has suspension isomorphisms for all representations),
this is actually a functor on stable maps between orbits, so does define
a Mackey functor. 
On the other hand, the Wirthm\"uller isomorphism
(\cite[3.9.5]{CW:ordinaryhomology}) allows us to write this as
\[
 \Mackey H_G^\gamma(X;\Mackey T)(G/H) \iso \tilde H_K^{\gamma|K}(X;\Mackey T|K),
\]
where $\gamma|K$ and $\Mackey T|K$ are the restrictions from $G$ to $K$
defined in the most obvious ways.
Thus, treating ordinary cohomology as a Mackey functor amounts to considering
all the cohomologies of $X$ for all the subgroups of $G$ simultaneously,
along with the associated restriction and transfer maps.
As often happens, the more structure present, the more limited
the possibilities, hence the easier the computations.

Finally, a note about calling this $RO(\Pi_G B)$-graded cohomology.
For the purpose of computation, we want to pass from a functor on
representations $\gamma$ to a group graded on the ring $RO(\Pi_G B)$
of isomorphism classes of representations. As usual, this ``decategorification''
needs to be done with great care or we are sure to trip over
sign ambiguities. In Part~\ref{part:BU1} we will assume we can do this
so we can get on with the computations.
We will return to this issue and treat it properly in Part~\ref{part:point}.

\section{Summary of the cohomology of a point}\label{sec:pointSummary}

From this point on, all cohomology will be assumed to have coefficients in
$\Mackey A_{G/G}$ unless stated otherwise, and we write
$\Mackey H_G^\alpha(X)$ for $\Mackey H_G^\alpha(X;\Mackey A_{G/G})$.
Further, it will be useful to make a distinction later between cohomology graded
on $RO(G)$ and cohomology graded on $RO(\Pi_G B)$, so we now adopt the notation
that $\Mackey H_G^\bullet$ denotes an $RO(G)$-graded theory while
$\Mackey H_G^*$ denotes an $RO(\Pi_G B)$-graded theory (for a specified space $B$).

In Part~\ref{part:BU1} we shall describe the cohomology of $B_GU(1)$
as an algebra over the cohomology of a point, which is to say, the reduced
cohomology of $S^0$. The calculation will use
the cofibration sequence $EG_+\to S^0 \to \tE G$, where $EG$ is a nonequivariantly contractible 
free $G$ space. To that end, we need explicit descriptions of
the $RO(G)$-graded cohomologies of $EG_+$, $S^0$, and $\tE G$, and
the maps in the long exact sequence induced by the cofibration sequence.
We give the results we need here and the proofs in Part~\ref{part:point}.
The description of the cohomology of a point is essentially the same as that
given by Lewis in \cite{Le:projectivespaces},
but our notation is somewhat different.

We first introduce some elements in the cohomology of a point.

\begin{definition}\label{def:iota}
If $p=2$, let
\[
 \iota\in \tilde H_e^{\LL-1}(S^0) = \Mackey H_G^{\LL-1}(S^0)(G/e)
\]
be the image of the identity under the isomorphism
\[
 \tilde H_e^0(S^0) \iso \tilde H_e^1(S^1) \iso \tilde H_e^{\LL}(S^1)
  \iso \tilde H_e^{\LL-1}(S^0),
\]
using our chosen nonequivariant identification of $\LL$ with $\R$
from Definition~\ref{def:irrRepresentations}.
If $p$ is odd and $1\leq k \leq (p-1)/2$, let
\[
 \iota_k\in \tilde H_e^{\MM_k-2}(S^0) = \Mackey H_G^{\MM_k-2}(S^0)(G/e)
\]
be the image of the identity under the isomorphism
\[
 \tilde H_e^0(S^0) \iso \tilde H_e^2(S^2) \iso \tilde H_e^{\MM_k}(S^2)
  \iso \tilde H_e^{\MM_k-2}(S^0),
\]
using our chosen nonequivariant identifications of $\MM_k$ with $\R^2$.
\end{definition}

When $p=2$,
multiplication by $\iota$ is an isomorphism on
$\Mackey H_G^*(S^0)(G/e)$, so there is an element
$\iota^{-1}\in\Mackey H_G^{1-\LL}(S^0)(G/e)$ such that
$\iota\cdot \iota^{-1}$ is the identity in $\Mackey H_G^0(S^0)(G/e)$.
Similarly, when $p$ is odd, there are elements
$\iota_k^{-1} \in\Mackey H_G^{2-M_k}(S^0)$ such that
$\iota_k \cdot \iota_k^{-1}$ is the identity in $\Mackey H_G^0(S^0)(G/e)$.

We shall occasionally use the following convenient notation.
Recall Definition~\ref{def:ro0G}.

\begin{definition}
Suppose that $\alpha\in I^\ev(G)$, so 
$\alpha = \sum_k n_k(\MM_k-2)$ for $n_k\in\Z$. Define
\[
 \iota^\alpha = \prod_k \iota_k^{n_k}
  \in \Mackey H_G^\alpha(S^0)(G/e).
\]
\end{definition}

If we use this notation when $p=2$, understanding $\iota_1 = \iota^2$,
$\iota^\alpha$ will give only the even powers of $\iota$.

\begin{definition}\label{def:EulerClasses}
Let $V$ be a representation of $G$. 
The {\em Euler class} of $V$ is the
element 
\[
 e_V \in \tilde H_G^{V}(S^0) = \Mackey H_G^V(S^0)(G/G)
\]
that is the image of the identity
under the map
\[
 \tilde H_G^0(S^0) \iso \tilde H_G^V(S^V) \to \tilde H_G^V(S^0),
\]
where the last map is restriction along the inclusion.
\end{definition}

When $G = \Z/2$, we write $e = e_\LL$, using the notation introduced
in Definition~\ref{def:irrRepresentations}.
It will be convenient to write $e_1 = e_{\MM_1} = e^2$.
Similarly, when $G = \Z/p$ with $p$ odd, we write $e_k = e_{\MM_k}$
for $1\leq k\leq (p-1)/2$.

We will also be interested in the Euler classes of the complex representations $\C_k$
for all $1\leq k \leq p-1$. For $1\leq k \leq p/2$, this is just $e_k$.
For $p/2 < k \leq p-1$, $\C_k$ is isomorphic to $\MM_{p-k}$ as a
real representation, but with the opposite orientation, so it will be convenient
to let
\[
 e_k = -e_{p-k} \qquad (p+1)/2 \leq k \leq p-1.
\]
(We shall give a better justification of this in \S\ref{sec:pointprelim}.
The choice of sign really has to do with how we pass to grading on $RO(G)$
and is determined by the nonequivariant identifications with $\C$ that we fixed.)
Finally, we extend the notation to all integers by letting $e_0 = 0$ (the Euler class of 
the trivial representation $\R^2$ is 0) and saying that $e_k = e_{k'}$
if $k \equiv k' \pmod p$.
N.b., this means that
\[
 e_{-k} = \begin{cases}
            e_k & \text{if $p=2$, but} \\
            -e_k & \text{if $p$ is odd.}
          \end{cases}
\]

The following notation will be convenient.

\begin{definition}
If $p=2$ and $\alpha = n\Lambda\in RO(G)$ with $n\geq 0$, define
\[
 e^\alpha = e^n \in \Mackey H_G^\alpha(S^0)(G/G).
\]
If $p$ is odd and
$\alpha\in RO(G)$ has the form
$\alpha = \sum_{k=1}^{(p-1)/2} n_k\MM_k$ for $n_k \geq 0$, define
\[
 e^\alpha = \prod_{k=1}^{(p-1)/2} e_k^{n_k}
  \in \Mackey H_G^\alpha(S^0)(G/G).
\]
In contexts in which $e$ or each $e_k$ is invertible, we can extend this notation
to all $\alpha$ such that $\alpha^G = 0$.
\end{definition}

Lemma~\ref{lem:EGInvertibles} defines invertible elements
\[
 \xi_k \in \tilde H_G^{\MM_k-2}(EG_+)
\]
for $1\leq k \leq p/2$. Remark~\ref{rem:extendedXi} shows that the
natural extension of this definition is to let
\[
 \xi_k = \xi_{p-k} \qquad (p+1)/2 \leq k \leq p-1.
\]
The calculation of $\Mackey H_G^\bullet(S^0)$ shows that there are 
unique (noninvertible) elements
\[
 \xi_k \in \tilde H_G^{\MM_k-2}(S^0)
\]
mapping to $\xi_k\in \tilde H_G^{\MM_k-2}(EG_+)$ under the map induced by $EG_+\to S^0$.
When $p=2$ we often write $\xi$ for $\xi_1$.

We introduce the following notation.

\begin{definition}\label{def:xipower}
Suppose that $\alpha\in RO_+(G)$, so 
$\alpha = \sum_{1\leq k\leq p/2} n_k(\MM_k-2)$ for $n_k \geq 0$. Define
\[
 \xi^\alpha = \prod_{k=1}^{(p-1)/2} \xi_k^{n_k}
  \in \Mackey H_G^\alpha(S^0)(G/G).
\]
In contexts in which each $\xi_k$ is invertible, we can extend this notation
to all $\alpha\in I^\ev(G)$.
\end{definition}

\begin{definition}\label{def:nu}
If $p$ is odd, define a homomorphism
\[
 \nu\colon RO_0(G) \to (\Z/p)^\times
\]
(from the additive operation on the left to the multiplicative one on the right)
as follows:
Given $\alpha\in RO_0(G)$ write $\alpha = \sum_{k=2}^{(p-1)/2} n_k(\MM_k - \MM_1)$
and define
\[
 \nu(\alpha) = \Big[\prod_{k=2}^{(p-1)/2} k^{n_k}\Big].
\]
If $p=2$, $RO_0(G) = 0$ and we let $\nu(0) = [1] \in (\Z/2)^\times$.
\end{definition}

In our use of $\nu(\alpha)$, we shall usually think of it as a congruence class in $\Z$ and
write $a\in \nu(\alpha)$ to mean that $a$ is an integer in the congruence class $\nu(\alpha)$.
In Definition~\ref{def:muBetaD} and Corollary~\ref{cor:muBetaD} we will define elements
\[
 \mu^{\alpha,a} \in \Mackey H_G^{\alpha}(S^0)(G/G)
\]
for all $\alpha\in RO_0(G)$ and $a\in\nu(\alpha)$.
Because $RO_0(G) = 0$ when $p=2$ or $3$, these elements are primarily useful when $p > 3$.

We now state the calculations we need, with the proofs to be given
in Part~\ref{part:point}. The cases $p=2$ and $p$ odd are similar, but sufficiently different
that it makes sense to state them separately.

\begin{theorem}\label{thm:evenCohomPoint}
Let $p = 2$. Additively,
\[
 \Mackey H_G^\alpha(S^0) \iso
  \begin{cases}
   \Mackey A_{G/G} & \text{if $\alpha = 0$} \\
   \Mackey R\Z & \text{if $|\alpha| = 0$ and $\alpha^G < 0$ is even} \\
   \Mackey R\Z_- & \text{if $|\alpha| = 0$ and $\alpha^G \leq 1$ is odd} \\
   \Mackey L\Z & \text{if $|\alpha| = 0$ and $\alpha^G > 0$ is even} \\
   \Mackey L\Z_- & \text{if $|\alpha| = 0$ and $\alpha^G \geq 3$ is odd} \\
   \conc\Z & \text{if $|\alpha| \neq 0$ and $\alpha^G = 0$} \\
   \conc{\Z/2} & \text{if $|\alpha| > 0$ and $\alpha^G < 0$ is even} \\
   \conc{\Z/2} & \text{if $|\alpha| < 0$ and $\alpha^G \geq 3$ is odd} \\
   0 & \text{otherwise.}
  \end{cases}
\]
Multiplicatively, $\Mackey H_G^\bullet(S^0)$ is a strictly commutative $RO(G)$-graded ring,
generated by elements
\begin{align*}
 \iota &\in \Mackey H_G^{\LL-1}(S^0)(G/e) \\
 \iota^{-1} &\in \Mackey H_G^{1-\LL}(S^0)(G/e) \\
 \xi &\in \Mackey H_G^{2(\LL-1)}(S^0)(G/G) \\
 e &\in \Mackey H_G^{\LL}(S^0)(G/G) \\
 e^{-m}\kappa &\in \Mackey H_G^{-m\LL}(S^0)(G/G) & & m\geq 1 \\
 e^{-m}\delta\xi^{-n} &\in \Mackey H_G^{1 - m\LL - 2n(\LL-1)}(S^0)(G/G)
   & & m, n \geq 1.
\end{align*}
These generators satisfy the following {\em structural} relations:
\begin{align*}
 \tau(\iota^{-1}) &= 0 \\
 \tau(\iota^{-2n-1}) &= e^{-1}\delta\xi^{-n} & & \text{for $n\geq 1$} \\
 \kappa\xi &= 0 \\
 \rho(\xi) &= \iota^2 \\
 \rho(e) &= 0 \\
 \rho(e^{-m}\kappa) &= 0 & & \text{for $m\geq 1$}\\
 \rho(e^{-m}\delta\xi^{-n}) &= 0 & & \text{for $m\geq 2$ and $n\geq 1$}\\
 2e^{-m}\delta\xi^{-n} &= 0 & & \text{for $m\geq 2$ and $n\geq 1$}
\intertext{and the following {\em multiplicative} relations:}
  \iota\cdot \iota^{-1} &= \rho(1) \\
  e\cdot e^{-m}\kappa &= e^{-m+1}\kappa & &\text{for $m\geq 1$} \\
 \xi\cdot e^{-m}\kappa &= 0 & & \text{for $m\geq 1$} \\
 e^{-m}\kappa\cdot e^{-n}\kappa &= 2e^{-m-n}\kappa & & \text{for $m\geq 0$ and $n\geq 0$} \\
 e\cdot e^{-m}\delta\xi^{-n} &= e^{-m+1}\delta\xi^{-n} & &\text{for $m\geq 2$ and $n\geq 1$} \\
 \xi\cdot e^{-m}\delta\xi^{-n} &= e^{-m}\delta\xi^{-n+1} & & 
   \text{for $m\geq 1$ and $n\geq 2$} \\
 \xi\cdot e^{-m}\delta\xi^{-1} &= 0 & & \text{for $m\geq 2$} \\
 e^{-m}\kappa \cdot e^{-n}\delta\xi^{-k} &= 0 & & \text{if $m\geq 0$, $n\geq 1$, and $k\geq 1$}\\
 e^{-m}\delta\xi^{-k}\cdot e^{-n}\delta\xi^{-\ell} &= 0 & & 
   \text{if $m, n, k, \ell\geq 1$}
\end{align*}
The following relations are implied by the preceding ones:
\begin{align*}
 \kappa e &= 2e \\
 2e^m\xi^n &= 0 & & \text{if $m>0$ and $n>0$} \\
 t\iota^k &= (-1)^k\iota^k & & \text{for all $k$} \\
 \xi\cdot\tau(\iota^k) &= \tau(\iota^{k+2}) & &\text{for all $k$} \\
 e\cdot\tau(\iota^k) &= 0 & &\text{for all $k$} \\
 e^{-m}\kappa \cdot \tau(\iota^k) &= 0 & & \text{for all $m\geq 1$ and $k$} \\
 e^{-m}\delta\xi^{-n}\cdot \tau(\iota^k) &= 0 & & \text{for all $m, n\geq 1$ and $k$} \\
 \tau(\iota^k)\cdot\tau(\iota^\ell) &= 0 & &\text{if $k$ or $\ell$ is odd} \\
 \tau(\iota^{2k})\cdot\tau(\iota^{2\ell}) &= 2\tau(\iota^{2(k+\ell)}) & &\text{for all $k$ and $\ell$}\\
 \tau(\iota^{2k+1}) &= 0 & & \text{if $k\geq 0$} \\
 e\cdot e^{-1}\delta\xi^{-n} &= 0 & & \text{if $n\geq 1$} 
\end{align*}
\qed
\end{theorem}

The notation $e^{-m}\kappa$ comes from the fact that 
$e^m\cdot e^{-m}\kappa = \kappa \in \tilde H_G^0(S^0) = A(G)$.
The reason for the notation $e^{-m}\delta\xi^{-n}$ should become clearer shortly.

It helps to have a way to visualize these calculations, and the common
way of doing so is to plot the groups or Mackey functors on a grid.
Different authors, however, have used different axes.
Lewis, in \cite{Le:projectivespaces}, uses
$\alpha^G$ as the horizontal axis and $|\alpha|$ as the vertical axis;
because our results are stated in these terms we will do so as well.
Dugger \cite{Dug:AHSSforKR} and Kronholm \cite{Kro:SerreSS} use the
so-called motivic grading, plotting $|\alpha|$ vs $|\alpha| - \alpha^G$,
although Dugger uses $|\alpha|$ as the vertical axis while
Kronholm uses it as the horizontal axis.
Another set of axes for which one can make a case is
$\alpha^G$ for the horizontal axis and $|\alpha|-\alpha^G$
for the vertical axis.

So the top diagram in Figure~\ref{fig:EvenCohomPoint} shows the Mackey functors $\Mackey H_G^\alpha(S^0)$, with
$\alpha^G$ as the horizontal axis and $|\alpha|$ as the vertical axis,
with dots representing zero functors.
The bottom diagram gives the elements that generate the corresponding Mackey functors.
Generators shown in parentheses represent elements at level $G/e$.
\begin{figure}
\[\def\objectstyle{\scriptstyle}
 \xymatrix@!0@R=5ex@C=2.5em{
  & & & & & & & & & & \\
  & \conc{\Z/2} &\cdot& \conc{\Z/2} &\cdot& \conc\Z &\cdot& \cdot &\cdot& \cdot & \cdot \\
  & \conc{\Z/2} &\cdot& \conc{\Z/2} &\cdot& \conc\Z &\cdot& \cdot &\cdot& \cdot & \cdot \\
  & \conc{\Z/2} &\cdot& \conc{\Z/2} &\cdot& \conc\Z &\cdot& \cdot &\cdot& \cdot & \cdot \\
  & \conc{\Z/2} &\cdot& \conc{\Z/2} &\cdot& \conc\Z &\cdot& \cdot &\cdot& \cdot & \cdot \\
  \ar@{-}'[r]'[rr]'[rrr]'[rrrr]'[rrrrr]'[rrrrrr]'[rrrrrrr]'[rrrrrrrr]'[rrrrrrrrr]'[rrrrrrrrrr][rrrrrrrrrrr]
   & \Mackey R\Z & \Mackey R\Z_{\mathrlap{-}} & \Mackey R\Z & \Mackey R\Z_{\mathrlap{-}} & \Mackey A_{G/G} & \Mackey R\Z_{\mathrlap{-}} & \Mackey L\Z & \Mackey L\Z_{\mathrlap{-}} & \Mackey L\Z & \Mackey L\Z_{\mathrlap{-}} &\\
  & \cdot &\cdot& \cdot &\cdot& \conc\Z &\cdot& \cdot & \conc{\Z/2} & \cdot &  \conc{\Z/2}  \\
  & \cdot &\cdot& \cdot &\cdot& \conc\Z &\cdot& \cdot & \conc{\Z/2} & \cdot &  \conc{\Z/2}  \\
  & \cdot &\cdot& \cdot &\cdot& \conc\Z &\cdot& \cdot & \conc{\Z/2} & \cdot &  \conc{\Z/2}  \\
  & \cdot &\cdot& \cdot &\cdot& \conc\Z &\cdot& \cdot & \conc{\Z/2} & \cdot &  \conc{\Z/2}  \\
  & & & & & \ar@{-}'[u]'[uu]'[uuu]'[uuuu]'[uuuuu]'[uuuuuu]'[uuuuuuu]'[uuuuuuuu]'[uuuuuuuuu][uuuuuuuuuu]
 }
\]
\[\def\objectstyle{\scriptstyle}
 \xymatrix@!0@R=5ex@C=2.5em{
  & & & & & & & & & & \\
  & e^4\xi^2 &\cdot& e^4\xi &\cdot& e^4 &\cdot& \cdot &\cdot& \cdot & \cdot \\
  & e^3\xi^2 &\cdot& e^3\xi &\cdot& e^3 &\cdot& \cdot &\cdot& \cdot & \cdot \\
  & e^2\xi^2 &\cdot& e^2\xi &\cdot& e^2 &\cdot& \cdot &\cdot& \cdot & \cdot \\
  & e\xi^2 &\cdot& e\xi &\cdot& e &\cdot& \cdot &\cdot& \cdot & \cdot \\
  \ar@{-}'[r]'[rr]'[rrr]'[rrrr]'[rrrrr]'[rrrrrr]'[rrrrrrr]'[rrrrrrrr]'[rrrrrrrrr]'[rrrrrrrrrr][rrrrrrrrrrr]
   & \xi^2 & (\iota^3) & \xi & (\iota) & 1 & (\iota^{-1}) & (\iota^{-2}) & (\iota^{-3}) & (\iota^{-4}) & (\iota^{-5}) &\\
  & \cdot &\cdot& \cdot &\cdot& e^{-1}\kappa &\cdot& \cdot & e^{-2}\delta\xi^{-1} & \cdot &  e^{-2}\delta\xi^{-2}  \\
  & \cdot &\cdot& \cdot &\cdot& e^{-2}\kappa &\cdot& \cdot & e^{-3}\delta\xi^{-1} & \cdot &  e^{-3}\delta\xi^{-2}  \\
  & \cdot &\cdot& \cdot &\cdot& e^{-3}\kappa &\cdot& \cdot & e^{-4}\delta\xi^{-1} & \cdot &  e^{-4}\delta\xi^{-2}  \\
  & \cdot &\cdot& \cdot &\cdot& e^{-4}\kappa &\cdot& \cdot & e^{-5}\delta\xi^{-1} & \cdot &  e^{-5}\delta\xi^{-2}  \\
  & & & & & \ar@{-}'[u]'[uu]'[uuu]'[uuuu]'[uuuuu]'[uuuuuu]'[uuuuuuu]'[uuuuuuuu]'[uuuuuuuuu][uuuuuuuuuu]
 }
\]
\caption{$\protect\Mackey H_G^\bullet(S^0)$ and its generators, $p=2$}\label{fig:EvenCohomPoint}
\end{figure}

\begin{theorem}\label{thm:EvenEG}
Let $p=2$. Additively,
\[
 \Mackey H_G^{\alpha}(EG_+) \iso
  \begin{cases}
   \Mackey R\Z & \text{if $|\alpha| = 0$ and $\alpha^G$ is even} \\
   \Mackey R\Z_- & \text{if $|\alpha| = 0$ and $\alpha^G$ is odd} \\
   \conc{\Z/2} & \text{if $|\alpha|>0$ and $\alpha^G$ is even} \\
   0 & \text{otherwise.}
  \end{cases}
\]
Multiplicatively, $\Mackey H_G^\bullet(EG_+)$ is a strictly commutative $RO(G)$-graded
$\Mackey R\Z$-algebra generated by
\begin{align*}
 e &\in \Mackey H_G^{\LL}(EG_+)(G/G), \\
 \iota &\in \Mackey H_G^{\LL-1}(EG_+)(G/e), \\
 \iota^{-1} &\in \Mackey H_G^{1-\LL}(EG_+)(G/e), \\
 \xi &\in \Mackey H_G^{2(\LL-1)}(EG_+)(G/G), \quad\text{and} \\
 \xi^{-1} &\in \Mackey H_G^{2(1-\LL)}(EG_+)(G/G),
\end{align*}
subject to the relations
\begin{align*}
 \rho(e) &= 0 \\
 \tau(\iota) &= 0 \\
 \rho(\xi) &= \iota^2 \\
 \iota\cdot \iota^{-1} &= \rho(1) \quad\text{and} \\
 \xi\cdot\xi^{-1} &= 1.
\end{align*}
\qed
\end{theorem}

Figure~\ref{fig:EvenEG} shows $\Mackey H_G^\bullet(EG_+)$ and its generators when $p=2$.
We will not use this fact directly, but it follows from the calculation, and
the action of $\Mackey H_G^\bullet(S^0)$ implied by
Proposition~\ref{prop:EvenMaps} below, that
$
 \Mackey H_G^\bullet(EG_+) \iso \Mackey H_G^\bullet(S^0)[\xi^{-1}]
$.
\begin{figure}
\[\def\objectstyle{\scriptstyle}
 \xymatrix@!0@R=5ex@C=2.5em{
  & & & & & & & & & & \\
  & \conc{\Z/2} &\cdot& \conc{\Z/2} &\cdot& \conc{\Z/2} &\cdot& \conc{\Z/2} &\cdot& \conc{\Z/2} \\
  & \conc{\Z/2} &\cdot& \conc{\Z/2} &\cdot& \conc{\Z/2} &\cdot& \conc{\Z/2} &\cdot& \conc{\Z/2} \\
  & \conc{\Z/2} &\cdot& \conc{\Z/2} &\cdot& \conc{\Z/2} &\cdot& \conc{\Z/2} &\cdot& \conc{\Z/2} \\
  & \conc{\Z/2} &\cdot& \conc{\Z/2} &\cdot& \conc{\Z/2} &\cdot& \conc{\Z/2} &\cdot& \conc{\Z/2} \\
  & \conc{\Z/2} &\cdot& \conc{\Z/2} &\cdot& \conc{\Z/2} &\cdot& \conc{\Z/2} &\cdot& \conc{\Z/2} \\
  \ar@{-}'[r]'[rr]'[rrr]'[rrrr]'[rrrrr]'[rrrrrr]'[rrrrrrr]'[rrrrrrrr]'[rrrrrrrrr][rrrrrrrrrr]
   & \Mackey R\Z & \Mackey R\Z_{\mathrlap{-}} & \Mackey R\Z & \Mackey R\Z_{\mathrlap{-}} & \Mackey R\Z & \Mackey R\Z_{\mathrlap{-}} & \Mackey R\Z & \Mackey R\Z_{\mathrlap{-}} & \Mackey R\Z & \\
  & & & & & \ar@{-}'[u]'[uu]'[uuu]'[uuuu]'[uuuuu]'[uuuuuu][uuuuuuu]
 }
\]
\[\def\objectstyle{\scriptstyle}
 \xymatrix@!0@R=5ex@C=2.5em{
  & & & & & & & & & & \\
  & e^5\xi^2 &\cdot& e^5\xi &\cdot& e^5 &\cdot& e^5\xi^{-1} &\cdot& e^5\xi^{-2} \\
  & e^4\xi^2 &\cdot& e^4\xi &\cdot& e^4 &\cdot& e^4\xi^{-1} &\cdot& e^4\xi^{-2} \\
  & e^3\xi^2 &\cdot& e^3\xi &\cdot& e^3 &\cdot& e^3\xi^{-1} &\cdot& e^3\xi^{-2} \\
  & e^2\xi^2 &\cdot& e^2\xi &\cdot& e^2 &\cdot& e^2\xi^{-1} &\cdot& e^2\xi^{-2} \\
  & e\xi^2 &\cdot& e\xi &\cdot& e &\cdot& e\xi^{-1} &\cdot& e\xi^{-2} \\
  \ar@{-}'[r]'[rr]'[rrr]'[rrrr]'[rrrrr]'[rrrrrr]'[rrrrrrr]'[rrrrrrrr]'[rrrrrrrrr][rrrrrrrrrr]
   & \xi^2 & (\iota^3) & \xi & (\iota) & 1 & (\iota^{-1}) & \xi^{-1} & (\iota^{-3}) & \xi^{-2} & \\
  & & & & & \ar@{-}'[u]'[uu]'[uuu]'[uuuu]'[uuuuu]'[uuuuuu][uuuuuuu]
 }
\]
\caption{$\protect\Mackey H_G^\bullet(EG_+)$ and its generators, $p=2$}\label{fig:EvenEG}
\end{figure}

Turning to $\Mackey H_G^\bullet(\tE G)$, its ring structure is not that useful to us, partly because
it is a ring without a unit. We will, instead, describe $\Mackey H_G^\bullet(\tE G)$ as a module
over $\Mackey H_G^\bullet(S^0)$. In fact, it is a module over the following localization.

\begin{proposition}\label{prop:localizedPoint}
Let $p=2$. On inverting $e$ in $\Mackey H_G^\bullet(S^0)$ we get
\[
 \Mackey H_G^\bullet(S^0)[e^{-1}] \iso
  \conc{\Z}[e,e^{-1},\xi]/\langle 2\xi\rangle.
\]
\end{proposition}

\begin{proof}
Because $\rho(e) = 0$, inverting $e$ kills $\Mackey H_G^\bullet(S^0)(G/e)$,
hence the result is a module over $\conc{\Z}$.
In $\Mackey H_G^\bullet(S^0)$, every element is killed by a high enough power of $e$
except the terms $e^{-m}\kappa$ and $e^m\xi^n$ for $m\geq 0$ and $n\geq 0$.
We have $e^{m+1}\cdot e^{-m}\kappa = 2e$, so $e^{-m}\kappa = 2e^{-m}$
in $\Mackey H_G^\bullet(S^0)[e^{-1}]$.
We have $2e\xi = 0$, so $2\xi = 0$ in $\Mackey H_G^\bullet(S^0)[e^{-1}]$.
\end{proof}

\begin{theorem}\label{thm:EvenTEG}
Let $p=2$. Additively,
\[
 \Mackey H_G^\alpha(\tE G) \iso
  \begin{cases}
    \conc{\Z} & \text{if $\alpha^G = 0$} \\
    \conc{\Z/2} & \text{if $\alpha^G\geq 3$ is odd} \\
    0 & \text{otherwise.}
  \end{cases}
\]
Multiplication by $e\in\Mackey H_G^\bullet(S^0)$ is an isomorphism on $\Mackey H_G^\bullet(\tE G)$,
so $\Mackey H_G^\bullet(\tE G)$ is a module over $\Mackey H_G^\bullet(S^0)[e^{-1}]$.
As such, it is generated by elements
\begin{align*}
 \kappa &\in \Mackey H_G^0(\tE G)(G/G) \quad\text{and}\\
 \delta\xi^{-k} &\in \Mackey H_G^{1 - 2k(\LL-1)}(\tE G)(G/G) \quad k \geq 1
\end{align*}
such that 
\begin{align*}
 \xi\cdot\kappa &= 0 \\
 \xi\cdot\delta\xi^{-k} &= \delta\xi^{-(k-1)}\quad\text{$k>1$, and} \\
 \xi\cdot\delta\xi^{-1} &= 0.
\end{align*}
\qed
\end{theorem}

Figure~\ref{fig:EvenTEG} shows $\Mackey H_G^\bullet(\tE G)$ and its generators.
\begin{figure}
\[\def\objectstyle{\scriptstyle}
 \xymatrix@!0@R=5ex@C=2.5em{
  & \ar@{-}'[d]'[dd]'[ddd]'[dddd]'[ddddd]'[dddddd]'[ddddddd][dddddddd] \\
  & \conc\Z & \cdot & \cdot & \conc{\Z/2} & \cdot & \conc{\Z/2} & \cdot & \conc{\Z/2} \\
  & \conc\Z & \cdot & \cdot & \conc{\Z/2} & \cdot & \conc{\Z/2} & \cdot & \conc{\Z/2} \\
  & \conc\Z & \cdot & \cdot & \conc{\Z/2} & \cdot & \conc{\Z/2} & \cdot & \conc{\Z/2} \\
  \ar@{-}'[r]'[rr]'[rrr]'[rrrr]'[rrrrr]'[rrrrrr]'[rrrrrrr]'[rrrrrrrr][rrrrrrrrr]
   & \conc\Z & \cdot & \cdot & \conc{\Z/2} & \cdot & \conc{\Z/2} & \cdot & \conc{\Z/2} &  \\
  & \conc\Z & \cdot & \cdot & \conc{\Z/2} & \cdot & \conc{\Z/2} & \cdot & \conc{\Z/2} \\
  & \conc\Z & \cdot & \cdot & \conc{\Z/2} & \cdot & \conc{\Z/2} & \cdot & \conc{\Z/2} \\
  & \conc\Z & \cdot & \cdot & \conc{\Z/2} & \cdot & \conc{\Z/2} & \cdot & \conc{\Z/2} \\
  & &  
 }
\]
\[\def\objectstyle{\scriptstyle}
 \xymatrix@!0@R=5ex@C=2.5em{
  & \ar@{-}'[d]'[dd]'[ddd]'[dddd]'[ddddd]'[dddddd]'[ddddddd][dddddddd] \\
  & e^3\kappa & \cdot & \cdot & e^2\delta\xi^{-1} & \cdot & e^2\delta\xi^{-2} & \cdot & e^2\delta\xi^{-3} \\
  & e^2\kappa & \cdot & \cdot & e\delta\xi^{-1} & \cdot & e\delta\xi^{-2} & \cdot & e\delta\xi^{-3} \\
  & e\kappa & \cdot & \cdot & \delta\xi^{-1} & \cdot & \delta\xi^{-2} & \cdot & \delta\xi^{-3} \\
  \ar@{-}'[r]'[rr]'[rrr]'[rrrr]'[rrrrr]'[rrrrrr]'[rrrrrrr]'[rrrrrrrr][rrrrrrrrr]
   & \kappa & \cdot & \cdot & e^{-1}\delta\xi^{-1} & \cdot & e^{-1}\delta\xi^{-2} & \cdot & e^{-1}\delta\xi^{-3} &  \\
  & e^{-1}\kappa & \cdot & \cdot & e^{-2}\delta\xi^{-1} & \cdot & e^{-2}\delta\xi^{-2} & \cdot & e^{-2}\delta\xi^{-3} \\
  & e^{-2}\kappa & \cdot & \cdot & e^{-3}\delta\xi^{-1} & \cdot & e^{-3}\delta\xi^{-2} & \cdot & e^{-3}\delta\xi^{-3} \\
  & e^{-3}\kappa & \cdot & \cdot & e^{-4}\delta\xi^{-1} & \cdot & e^{-4}\delta\xi^{-2} & \cdot & e^{-4}\delta\xi^{-3} \\
  & &  
 }
\]
\caption{$\protect\Mackey H_G^\bullet(\tE G)$ and its generators, $p=2$}\label{fig:EvenTEG}
\end{figure}

To complete the picture for $p=2$, we need to describe the maps in the following
long exact sequence:
\[
 \cdots \to \Mackey H_G^{\bullet-1}(EG_+)
  \xrightarrow{\delta} \Mackey H_G^\bullet(\tE G)
  \xrightarrow{\psi} \Mackey H_G^\bullet(S^0)
  \xrightarrow{\phi} \Mackey H_G^\bullet(EG_+)
  \to \cdots
\]
It helps to look at Figures~\ref{fig:EvenCohomPoint}--\ref{fig:EvenTEG}
when reading the following result.

\begin{proposition}\label{prop:EvenMaps}
Let $p=2$.
$\delta\colon \Mackey H_G^{\bullet-1}(EG_+)\to \Mackey H_G^\bullet(\tE G)$ is given by
\begin{align*}
  \delta(\iota^k) &= 0 \text{ for all $k$ and} \\
  \delta(e^m\xi^n) &=
   \begin{cases}
    e^m\delta\xi^n & \text{if $n\leq -1$} \\
    0 & \text{otherwise.}
   \end{cases}
\end{align*}
$\psi\colon \Mackey H_G^\bullet(\tE G) \to \Mackey H_G^\bullet(S^0)$ is given by
\begin{align*}
  \psi(e^m\kappa) &= e^m\kappa \\
  \psi(e^m\delta\xi^{-n}) &= 
     \begin{cases}
      e^m\delta\xi^{-n} & \text{if $m\leq -1$} \\
      0 & \text{otherwise.}
     \end{cases}
\end{align*}
$\phi\colon \Mackey H_G^\bullet(S^0) \to \Mackey H_G^\bullet(EG_+)$ is given by
\begin{align*}
 \phi(\iota^k) &= \iota^k \\
 \phi(e^m\xi^n) &= e^m\xi^n \text{ for $m\geq 0$ and $n\geq 0$} \\
 \phi(e^{-m}\kappa) &= 0 \text{ for $m\geq 1$} \\
 \phi(e^{-m}\delta\xi^{-n}) &= 0.
\end{align*}
\qed
\end{proposition}

Note that $\phi\colon \Mackey H_G^{2(\LL-1)}(S^0)\to \Mackey H_G^{2(\LL-1)}(EG_+)$
is an isomorphism, hence $\xi\in \Mackey H_G^{2(\LL-1)}(S^0)$ is characterized
by the fact that $\phi(\xi) = \xi$, where $\xi = \xi_1$ is the Thom class
defined in Lemma~\ref{lem:EGInvertibles}.

We now give the results for odd $p$. As mentioned, the case $p=3$ is simpler than
for larger primes, but we will not separate it out in these statements.

\begin{theorem}\label{thm:oddCohomPoint}
Let $p$ be odd.
Additively,
\[
 \Mackey H_G^{\alpha}(S^0) \iso
  \begin{cases}
   \Mackey A[\nu(\alpha)] & \text{if $|\alpha| = \alpha^G = 0$} \\
   \Mackey R\Z & \text{if $|\alpha| = 0$ and $\alpha^G < 0$ is even} \\
   \Mackey L\Z & \text{if $|\alpha| = 0$ and $\alpha^G > 0$ is even} \\
   \conc\Z & \text{if $|\alpha| \neq 0$ and $\alpha^G = 0$} \\
   \conc{\Z/p} & \text{if $|\alpha| > 0$ and $\alpha^G < 0$ is even} \\
   \conc{\Z/p} & \text{if $|\alpha| < 0$ and $\alpha^G \geq 3$ is odd} \\
   0 & \text{otherwise.}
  \end{cases}
\]
Multiplicatively, $\Mackey H_G^\bullet(S^0)$ is a strictly commutative unital $RO(G)$-graded ring,
generated by elements
\begin{align*}
 \iota_k &\in \Mackey H_G^{\MM_k-2}(S^0)(G/e) & & 1\leq k \leq (p-1)/2 \\
 \iota_k^{-1} &\in \Mackey H_G^{2-\MM_k}(S^0)(G/e) & & 1\leq k \leq (p-1)/2 \\
 \xi_1 &\in \Mackey H_G^{\MM_1-2}(S^0)(G/G) \\
 e_1 &\in \Mackey H_G^{\MM_1}(S^0)(G/G) \\
 \mu^{\alpha,a} &\in \Mackey H_G^\alpha(S^0)(G/G) & & \alpha\in RO_0(G),\ a\in\nu(\alpha) \\
 e_1^{-m}\kappa &\in \Mackey H_G^{-m\MM_1}(S^0)(G/G) & & m\geq 1 \\
 e_1^{-m}\delta\xi_1^{-n} &\in \Mackey H_G^{1-m\MM_1 - n(\MM_1-2)}(S^0)(G/G)
  & & m, n \geq 1
\end{align*}
These generators satisfy the following {\em structural} relations:
\begin{align*}
 t\iota_k^{-1} &= \iota_k^{-1} & & \text{ for all $k$} \\
 \kappa\xi_1 &= 0 \\
 \rho(\xi_1) &= \iota_1 \\
 \rho(e_1) &= 0 \\
 \rho(\mu^{\alpha,a}) &= a \iota^\alpha 
     & & \text{ for all $\alpha\in RO_0(G)$ and $a\in\nu(\alpha)$} \\
 \rho(e_1^{-m}\kappa) &= 0 & & \text{ for $m\geq 1$} \\
 \rho(e_1^{-m}\delta\xi_1^{-n}) & = 0 & & \text{ for $m, n \geq 1$} \\
 p e_1^{-m}\delta\xi_1^{-n} &= 0 & & \text{ for $m, n \geq 1$,}
\intertext{the {\em redundancy} relations}
 \mu^{0,1} &= 1 \\
 \mu^{\alpha, a + p} &= \mu^{\alpha,a} + \tau(\iota^\alpha)
  & & \text{for all $\alpha\in RO_0(G)$ and $a\in\nu(\alpha)$,}
\intertext{and the following {\em multiplicative} relations:}
 \iota_k \cdot \iota_k^{-1} &= \rho(1) & & \text{for all $k$} \\
 \mu^{\alpha,a}\cdot \mu^{\beta,b} &= \mu^{\alpha+\beta,ab}
      & & \text{for $\alpha, \beta \in RO_0(G)$} \\
 e_1\cdot e_1^{-m}\kappa &= e_1^{-m+1}\kappa & & \text{for $m\geq 1$} \\
 \xi_1\cdot e_1^{-m}\kappa &= 0 & & \text{for $m\geq 1$} \\
 e_1 \cdot e_1^{-m}\delta\xi_1^{-n} &= e_1^{-m+1}\delta\xi_1^{-n} & & \text{for $m\geq 2$ and $n\geq 1$} \\
 e_1 \cdot e_1^{-1}\delta\xi_1^{-n} &= 0 & & \text{for $n\geq 1$} \\
 \xi_1\cdot e_1^{-m}\delta\xi_1^{-n} &= e_1^{-m}\delta\xi_1^{-n+1} 
     & & \text{for $m\geq 1$ and $n\geq 2$} \\
 \xi_1\cdot e_1^{-m}\delta\xi_1^{-1} &= 0 & & \text{for $m\geq 1$} \\
 e_1^{-m}\kappa \cdot e_1^{-n}\kappa &= p e_1^{-m-n}\kappa & & \text{for $m, n \geq 1$} \\
 e_1^{-\ell}\kappa\cdot e_1^{-m}\delta\xi_1^{-n} &= 0 & & \text{for $\ell, m \geq 1$} \\
 e_1^{-k}\delta\xi_1^{-\ell} \cdot e_1^{-m}\delta\xi_1^{-n} &= 0
     & & \text{for $k, \ell, m, n \geq 1$}
\end{align*}
For $\alpha\in RO_0(G)$ and $a\in \nu(\alpha)^{-1}$, let $a^{-1}\in\nu(\alpha)$ so
that $aa^{-1} \equiv 1 \pmod p$ and let
\[
 \lambda^{\alpha,a} = a\mu^{\alpha,a^{-1}} + \frac{1-aa^{-1}}{p}\tau(\iota^\alpha);
\]
this is independent of the choice of $a^{-1}$.
The following relations are implied by the preceding ones:
\begin{align*}
 \xi_1 \tau(\iota^\alpha) &= \tau(\iota_1\iota^\alpha)
  && \text{for $\alpha^G = 0$} \\
 \mu^{\alpha,a}\tau(\iota^\beta) &= a\tau(\iota^{\alpha+\beta}) 
  && \text{for $\alpha\in RO_0(G)$ and $\beta^G = 0$} \\
 \tau(\iota^\alpha)\tau(\iota^\beta) &= p\tau(\iota^{\alpha+\beta}) 
  && \text{for $\alpha^G = 0 = \beta^G$} \\
 e_1\tau(\iota^\alpha) &= 0 
  && \text{for $\alpha^G = 0$}\\
 \kappa e_1 &= pe_1 \\
 pe_1\xi_1 &= 0 \\
 \mu^{\alpha,a} &= \mathrlap{a\lambda^{\alpha,a^{-1}} + \frac{1-aa^{-1}}{p}\kappa\mu^{\alpha,a}} \\
  & && \text{for $\alpha\in RO_0(G)$, $a\in\nu(\alpha)$, $a^{-1}\in\nu(\alpha)^{-1}$} \\
 \lambda^{\alpha, a + p} &= \lambda^{\alpha,a} + \kappa\mu^{\alpha,a^{-1}} 
  && \text{for $\alpha\in RO_0(G)$, $a\in\nu(\alpha)^{-1}$, $a^{-1}\in\nu(\alpha)$} \\
 \lambda^{\alpha, a + p}\xi_1 &= \lambda^{\alpha,a}\xi_1 
  && \text{for $\alpha\in RO_0(G)$} \\
 e_1\lambda^{\alpha,a} &= a e_1\mu^{\alpha,a^{-1}}
  && \text{for $\alpha\in RO_0(G)$, $a\in\nu(\alpha)^{-1}$, $a^{-1}\in\nu(\alpha)$} \\
 \mu^{\alpha,a}\xi_1 &= a\lambda^{\alpha,a^{-1}}\xi_1
  && \text{for $\alpha\in RO_0(G)$, $a\in\nu(\alpha)$, $a^{-1}\in\nu(\alpha)^{-1}$} \\
 \rho(\lambda^{\alpha,a}) &= \iota^\alpha  
  && \text{for $\alpha\in RO_0(G)$} \\
 \lambda^{\alpha,a}  \lambda^{\beta,b} &= \lambda^{\alpha+\beta, ab} 
  && \text{for $\alpha, \beta \in RO_0(G)$}
\end{align*}
\qed
\end{theorem}

Figure~\ref{fig:OddCohomPoint} shows $\Mackey H_G^\bullet(S^0)$
and its generators.
For $\alpha\in RO_0(G)$,
this figure shows the groups in gradings that lie in the coset
$\alpha+\langle \R, \MM_1 \rangle \subset RO(G)$, with the horizontal grading being the fixed set dimension
and the vertical grading being the total dimension.
\begin{figure}
\[\def\objectstyle{\scriptstyle}
 \xymatrix@!0@R=5ex@C=2.5em{
  & & & & & & & & & & \\
  & \conc{\Z/p} & & \conc{\Z/p} & & \conc\Z & & \cdot & & \cdot &  \\
  &  &\cdot&  &\cdot&  &\cdot&  &\cdot&  & \cdot \\
  & \conc{\Z/p} & & \conc{\Z/p} & & \conc\Z & & \cdot & & \cdot &  \\
  &  &\cdot&  &\cdot&  &\cdot&  &\cdot&  & \cdot \\
  \ar@{-}'[r]'[rrr]'[rrrrr]'[rrrrrrr]'[rrrrrrrrr][rrrrrrrrrrr]
   & \Mackey R\Z &  & \Mackey R\Z &  & \Mackey A[\nu(\alpha)] &  & \Mackey L\Z &  & \Mackey L\Z &  &\\
  &  &\cdot&  &\cdot&  &\cdot&  & \conc{\Z/p} &  &  \conc{\Z/p}  \\
  & \cdot & & \cdot & & \conc\Z & & \cdot &  & \cdot &     \\
  &  &\cdot&  &\cdot&  &\cdot&  & \conc{\Z/p} &  &  \conc{\Z/p}  \\
  & \cdot & & \cdot & & \conc\Z & & \cdot &  & \cdot &     \\
  & & & & & \ar@{-}'[u]'[uuu]'[uuuuu]'[uuuuuuu]'[uuuuuuuuu][uuuuuuuuuu]
 }
\]
\[\def\objectstyle{\scriptstyle}
 \xymatrix@!0@R=5ex@C=2.5em{
  & & & & & & & & & & \\
  & \lambda^{\alpha,a^{-1}}e_1^2\xi_1^2 & & \lambda^{\alpha,a^{-1}}e_1^2\xi_1 & & \mu^{\alpha,a}e_1^2 & & \cdot & & \cdot &  \\
  &  &\cdot&  &\cdot&  &\cdot&  &\cdot&  & \cdot \\
  & \lambda^{\alpha,a^{-1}}e_1\xi_1^2 & & \lambda^{\alpha,a^{-1}}e_1\xi_1 & & \mu^{\alpha,a}e_1 & & \cdot & & \cdot &  \\
  &  &\cdot&  &\cdot&  &\cdot&  &\cdot&  & \cdot \\
  \ar@{-}'[r]'[rrr]'[rrrrr]'[rrrrrrr]'[rrrrrrrrr][rrrrrrrrrrr]
   & \lambda^{\alpha,a^{-1}}\xi_1^2 &  & \lambda^{\alpha,a^{-1}}\xi_1 &  & \{\mu^{\alpha,a},\iota^\alpha\} &  & (\iota^\alpha\iota_1^{-1}) &  & (\iota^\alpha\iota_1^{-2}) &  &\\
  &  &\cdot&  &\cdot&  &\cdot&  & \mu^{\alpha,a}e_1^{-1} \delta\xi_1^{-1} &  &  \mu^{\alpha,a}e_1^{-1} \delta\xi_1^{-2}  \\
  & \cdot & & \cdot & & \mu^{\alpha,a}e_1^{-1}\kappa & & \cdot &  & \cdot &     \\
  &  &\cdot&  &\cdot&  &\cdot&  & \mu^{\alpha,a}e_1^{-2} \delta\xi_1^{-1} &  &  \mu^{\alpha,a}e_1^{-2} \delta\xi_1^{-2}  \\
  & \cdot & & \cdot & & \mu^{\alpha,a}e_1^{-2}\kappa & & \cdot &  & \cdot &     \\
  & & & & & \ar@{-}'[u]'[uuu]'[uuuuu]'[uuuuuuu]'[uuuuuuuuu][uuuuuuuuuu]
 }
\]
\caption{$\protect\Mackey H_G^\bullet(S^0)$ and its generators, $p$ odd}\label{fig:OddCohomPoint}
\end{figure}

We mentioned earlier that this description is essentially the same as the one
given by Lewis in \cite{Le:projectivespaces}. The main difference is that
what Lewis did amounted to choosing one $a\in\nu(\alpha)$ for each $\alpha$ and using
that particular $\mu^{\alpha,a}$. In general there is no canonical choice and it makes
the relations easier if we allow all possible choices, at the expense, of course, of introducing
more generators.

The theorem mentions the Euler class $e_1$, but what about the other Euler classes $e_k$?
In Part~\ref{part:point}, we shall show the following.

\begin{proposition}\label{prop:EulerClasses}
With $e_k$ as in Definition~\ref{def:EulerClasses}, we have
\[
 e_k = \mu^{\MM_k-\MM_1,a} e_1, \quad 2\leq k \leq (p-1)/2,
\]
for any $a\in \nu(\MM_k-\MM_1) = [k]$. We also have
\[
 e_k = -\mu^{\MM_{p-k}-\MM_1,a} e_1, \quad (p+1)/2 \leq k \leq p-1,
\]
for $a\in [p-k] = -[k]$.
\qed
\end{proposition}

Note that, as discussed in \S\ref{sec:MackeyFunctors}, as a set of generators
for $\Mackey H_G^{\alpha}(S^0) \iso \Mackey A[\nu(\alpha)]$, we can also take
$\{\lambda^{\alpha,a^{-1}}, \kappa\mu^{\alpha,a}\}$, and 
$\kappa\mu^{\alpha,a}$ is independent of the choice of $a$.
With this and the preceding proposition in mind, we introduce some more notation.

\begin{definition}\label{def:ekappa}
For $\beta\in RO_0(G)$, let
\[
 \kappa^\beta = \kappa\mu^{\beta,b}
\]
for any $b\in\nu(\beta)$; as noted above, the result does not depend on the choice of $b$.
If $\alpha\in RO(G)$ has $\alpha^G = 0$, so $\alpha = \sum_{k=1}^{(p-1)/2} n_k\MM_k$
with $n_k\in\Z$, let
\[
 e^\alpha\kappa = e_1^{|\alpha|}\kappa^{\alpha-|\alpha|\MM_1}.
\]
Note that this is defined for $|\alpha|<0$ as well as $|\alpha|\geq 0$.
Combining the two notations, for $\alpha^G = 0$ and $\beta\in RO_0(G)$, let
\[
 e^\alpha\kappa^\beta = e^\alpha\kappa\mu^{\beta,b}
\]
for any $b\in \nu(\beta)$.
\end{definition}

\begin{theorem}\label{thm:OddEG}
Let $p$ be odd.
Additively,
\[
 \Mackey H_G^{\alpha}(EG_+) \iso
  \begin{cases}
   \Mackey R\Z & \text{if $|\alpha| = 0$} \\
   \conc{\Z/p} & \text{if $|\alpha| > 0$ and $\alpha^G$ is even} \\
   0 & \text{otherwise.}
  \end{cases}
\]
Multiplicatively, $\Mackey H_G^\bullet(EG_+)$ is a strictly commutative
$RO(G)$-graded $\Mackey R\Z$-algebra generated by
\begin{align*}
 e_1 &\in\Mackey H_G^{\MM_1}(EG_+)(G/G), \\
 \xi_k &\in \Mackey H_G^{\MM_k-2}(EG_+)(G/G), \quad 1 \leq k \leq (p-1)/2, \text{ and} \\
 \xi_k^{-1} &\in \Mackey H_G^{2-\MM_k}(EG_+)(G/G), \quad 1 \leq k \leq (p-1)/2,
\end{align*}
subject to the relation $\rho(e_1) = 0$.
\qed
\end{theorem}

\begin{figure}
\[\def\objectstyle{\scriptstyle}
 \xymatrix@!0@R=5ex@C=2.5em{
  & & & & & & & & & & \\
  & \conc{\Z/p} &  & \conc{\Z/p} &  & \conc{\Z/p} &  & \conc{\Z/p} &  & \conc{\Z/p}  \\
  & &\cdot&  &\cdot&  &\cdot&  &\cdot&  \\
  & \conc{\Z/p} &  & \conc{\Z/p} &  & \conc{\Z/p} &  & \conc{\Z/p} &  & \conc{\Z/p}  \\
  & &\cdot&  &\cdot&  &\cdot&  &\cdot&  \\
  & \conc{\Z/p} &  & \conc{\Z/p} &  & \conc{\Z/p} &  & \conc{\Z/p} &  & \conc{\Z/p}  \\
  & &\cdot&  &\cdot&  &\cdot&  &\cdot&  \\
  \ar@{-}'[r]'[rrr]'[rrrrr]'[rrrrrrr]'[rrrrrrrrr][rrrrrrrrrr]
    & \Mackey R\Z &  & \Mackey R\Z &  & \Mackey R\Z &  & \Mackey R\Z &  & \Mackey R\Z & \\
  & & & & & \ar@{-}'[u]'[uuu]'[uuuuu]'[uuuuuuu][uuuuuuuu]
 }
\]
\[\def\objectstyle{\scriptstyle}
 \xymatrix@!0@R=5ex@C=2.5em{
  & & & & & & & & & & \\
  & e_1^3\xi^\alpha\xi_1^2 &  & e_1^3\xi^\alpha\xi_1 &  & e_1^3\xi^\alpha &  & e_1^3\xi^{\alpha}\xi_1^{-1} &  & e_1^3\xi^{\alpha}\xi_1^{-2}  \\
  &  &\cdot&  &\cdot&  &\cdot&  &\cdot&  \\
  & e_1^2\xi^\alpha\xi_1^2 &  & e_1^2\xi^\alpha\xi_1 &  & e_1^2\xi^\alpha &  & e_1^2\xi^{\alpha}\xi_1^{-1} &  & e_1^2\xi^{\alpha}\xi_1^{-2}  \\
  &  &\cdot&  &\cdot&  &\cdot&  &\cdot&  \\
  & e_1\xi^\alpha\xi_1^2 &  & e_1\xi^\alpha\xi_1 &  & e_1\xi^\alpha &  & e_1\xi^{\alpha}\xi_1^{-1} &  & e_1\xi^{\alpha}\xi_1^{-2}  \\
  &  &\cdot&  &\cdot&  &\cdot&  &\cdot&  \\
  \ar@{-}'[r]'[rrr]'[rrrrr]'[rrrrrrr]'[rrrrrrrrr][rrrrrrrrrr]
   & \xi^\alpha\xi_1^2 &  & \xi^{\alpha}\xi_1 &  & \xi^\alpha &  & \xi^{\alpha}\xi_1^{-1} &  & \xi^{\alpha}\xi_1^{-2} & \\
  & & & & & \ar@{-}'[u]'[uuu]'[uuuuu]'[uuuuuuu][uuuuuuuu]
 }
\]
\caption{$\protect\Mackey H_G^\bullet(EG_+)$ and its generators, $p$ odd}\label{fig:OddEG}
\end{figure}
Figure~\ref{fig:OddEG} shows $\Mackey H_G^\bullet(EG_+)$ and its generators.
As in Figure~\ref{fig:OddCohomPoint}, 
for an $\alpha\in RO_0(G)$, the figure shows the groups in gradings that lie in the coset
$\alpha+\langle \R, \MM_1 \rangle \subset RO(G)$.

We need the following analogue of Proposition~\ref{prop:localizedPoint}.

\begin{proposition}\label{prop:oddLocalizedPoint}
Let $p$ be odd. On inverting the elements $e_k$ in $\Mackey H_G^\bullet(S^0)$,
$1\leq k \leq (p-1)/2$, we get
\[
 \Mackey H_G^\bullet(S^0)[e_k^{-1}] 
 \iso \conc\Z [e_k, e_k^{-1}, \xi_1 \mid 1\leq k \leq (p-1)/2]/\langle p\xi_1 \rangle.
\]
\end{proposition}

\begin{proof}
The proof is essentially the same as that of Proposition~\ref{prop:localizedPoint}.
We just need the additional observation that, if 
$\alpha = \sum_{k=2}^{(p-1)/2} n_k(\MM_k - \MM_1)$, then
\begin{align*}
 \mu^{\alpha,a} &= \prod_{k=2}^{(p-1)/2} (e_k e_1^{-1})^{n_k} = e^\alpha \quad\text{and} \\
 \lambda^{\alpha,b} &= b\prod_{k=2}^{(p-1)/2} (e_k e_1^{-1})^{n_k} = be^\alpha
\end{align*}
after inverting the $e_k$.
\end{proof}

\begin{theorem}\label{thm:OddTEG}
Let $p$ be odd.
Additively,
\[
 \Mackey H_G^\alpha(\tE G) \iso
  \begin{cases}
    \conc{\Z} & \text{if $\alpha^G = 0$} \\
    \conc{\Z/p} & \text{if $\alpha^G\geq 3$ is odd} \\
    0 & \text{otherwise.}
  \end{cases}
\]
Multiplication by each $e_k\in \Mackey H_G^\bullet(S^0)$ is an isomorphism
on $\Mackey H_G^\bullet(\tE G)$, so $\Mackey H_G^\bullet(\tE G)$ is a module
over $\Mackey H_G^\bullet(S^0)[e_k^{-1}]$. As such, it is generated by elements
\begin{align*}
 \kappa &\in \Mackey H_G^0(\tE G)(G/G) \quad\text{and}\\
 \delta\xi_1^{-k} &\in \Mackey H_G^{1-k(\MM_1-2)}(\tE G)(G/G) \quad k\geq 1
\end{align*}
such that
\begin{align*}
 \xi_1\cdot \kappa &= 0 \\
 \xi_1\cdot \delta\xi_1^{-k} &= \delta\xi^{-(k-1)} \quad k>1, \text{ and} \\
 \xi_1\cdot \delta\xi_1^{-1} &= 0.
\end{align*}
\qed
\end{theorem}

Figure~\ref{fig:OddTEG} shows $\Mackey H_G^\bullet(\tE G)$ and its generators.
\begin{figure}
\[\def\objectstyle{\scriptstyle}
 \xymatrix@!0@R=5ex@C=2.5em{
  & \ar@{-}'[d]'[ddd]'[ddddd]'[ddddddd]'[ddddddddd][dddddddddd] \\
  & \conc\Z & &  \cdot  & &  \cdot  & &  \cdot  \\
  & & \cdot & & \conc{\Z/p} & & \conc{\Z/p} & & \conc{\Z/p} \\
  & \conc\Z & &  \cdot  & &  \cdot  & &  \cdot  \\
  & & \cdot & & \conc{\Z/p} & & \conc{\Z/p} & & \conc{\Z/p} \\
  \ar@{-}'[r]'[rrr]'[rrrrr]'[rrrrrrr][rrrrrrrrr]
   & \conc\Z & &  \cdot  & &  \cdot  & &  \cdot  & & \\
  & & \cdot & & \conc{\Z/p} & & \conc{\Z/p} & & \conc{\Z/p} \\
  & \conc\Z & &  \cdot  & &  \cdot  & &  \cdot  \\
  & & \cdot & & \conc{\Z/p} & & \conc{\Z/p} & & \conc{\Z/p} \\
  & \conc\Z & &  \cdot  & &  \cdot  & &  \cdot  \\
  & &
 }
\]
\[\def\objectstyle{\scriptstyle}
 \xymatrix@!0@R=5ex@C=2.5em{
  & \ar@{-}'[d]'[ddd]'[ddddd]'[ddddddd]'[ddddddddd][dddddddddd] \\
  & e^\alpha e_1^2\kappa & &  \cdot  & &  \cdot  & &  \cdot  \\
  & & \cdot & & e^\alpha e_1\delta\xi_1^{-1} & & e^\alpha e_1\delta\xi_1^{-2} & & e^\alpha e_1\delta\xi_1^{-3} \\
  & e^\alpha e_1\kappa & &  \cdot  & &  \cdot  & &  \cdot  \\
  & & \cdot & & e^\alpha\delta\xi_1^{-1} & & e^\alpha\delta\xi_1^{-2} & & e^\alpha\delta\xi_1^{-3} \\
  \ar@{-}'[r]'[rrr]'[rrrrr]'[rrrrrrr][rrrrrrrrr]
   & e^\alpha\kappa & &  \cdot  & &  \cdot  & &  \cdot  & & \\
  & & \cdot & & e^\alpha e_1^{-1}\delta\xi_1^{-1} & & e^\alpha e_1^{-1}\delta\xi_1^{-2} & & e^\alpha e_1^{-1}\delta\xi_1^{-3} \\
  & e^\alpha e_1^{-1}\kappa & &  \cdot  & &  \cdot  & &  \cdot  \\
  & & \cdot & & e^\alpha e_1^{-2}\delta\xi_1^{-2} & & e^\alpha e_1^{-2}\delta\xi_1^{-2} & & e^\alpha e_1^{-2}\delta\xi_1^{-3} \\
  & e^\alpha e_1^{-2}\kappa & &  \cdot  & &  \cdot  & &  \cdot  \\
  & &
 }
\]
\caption{$\protect\Mackey H_G^\bullet(\tE G)$ and its generators, $p$ odd}\label{fig:OddTEG}
\end{figure}

Again, we need to describe the maps in the long exact sequence
\[
 \cdots \to \Mackey H_G^{\bullet-1}(EG_+)
  \xrightarrow{\delta} \Mackey H_G^\bullet(\tE G)
  \xrightarrow{\psi} \Mackey H_G^\bullet(S^0)
  \xrightarrow{\phi} \Mackey H_G^\bullet(EG_+)
  \to \cdots
\]

\begin{proposition}\label{prop:OddMaps}
Let $p$ be odd.
$\delta\colon \Mackey H_G^{\bullet-1}(EG_+)\to \Mackey H_G^\bullet(\tE G)$ is given by
\[
  \delta(e_1^m\xi^\alpha\xi_1^n) =
   \begin{cases}
    a e^\alpha e_1^m\delta\xi_1^n & \text{if $n\leq -1$} \\
    0 & \text{otherwise,}
   \end{cases}
\]
where $a\in \nu(\alpha)^{-1}$. (Note that the image, when nonzero, lies in
$\conc{\Z/p}$, so does not depend on the choice of $a$.)
$\psi\colon \Mackey H_G^\bullet(\tE G) \to \Mackey H_G^\bullet(S^0)$ is given by
\begin{align*}
  \psi(e^\alpha e_1^m\kappa) &= \mu^{\alpha,a}e_1^m\kappa \\
  \psi(e^\alpha e_1^m\delta\xi_1^{-n}) &= 
     \begin{cases}
      \mu^{\alpha,a}e_1^m\delta\xi_1^{-n} & \text{if $m\leq -1$} \\
      0 & \text{otherwise,}
     \end{cases}
\end{align*}
for $a\in\nu(\alpha)$. (Again, the choice of $a$ does not matter.)
$\phi\colon \Mackey H_G^\bullet(S^0) \to \Mackey H_G^\bullet(EG_+)$ is given by
\begin{align*}
 \phi(\iota^\alpha\iota_1^k) &= \iota^\alpha\iota_1^k = \rho(\xi^\alpha\xi_1^k) \\
 \phi(\lambda^{\alpha,a}e_1^m\xi_1^n) &= e_1^m\xi^\alpha \xi_1^n \text{ for $m\geq 0$ and $n\geq 1$} \\
 \phi(\mu^{\alpha,a} e_1^m) &= a\xi^\alpha e_1^m \text{ for $m\geq 0$} \\
 \phi(\mu^{\alpha,a}e_1^{-m}\kappa) &= 0 \text{ for all $m$} \\
 \phi(\mu^{\alpha,a}e_1^{-m}\delta\xi_1^{-n}) &= 0.
\end{align*}
\qed
\end{proposition}

Note that the calculation of $\phi$ implies in particular that
\[
 \phi(\lambda^{\MM_k-\MM_1,a}\xi_1) = \xi^{\MM_k-\MM_1}\xi_1 = \xi_k.
\]
Because $\phi$ is an isomorphism in grading $\MM_k-2$, it makes sense to define
\[
 \xi_k = \lambda^{\MM_k-\MM_1,a}\xi_1 \in \Mackey H_G^{\MM_k-2}(S^0),
\]
so that $\xi_k\in \Mackey H_G^\bullet(S^0)$ is characterized by
$\phi(\xi_k) = \xi_k$.

One last bit of notation. When $p=2$, $RO_0(G) = 0$, so there was no need
to introduce the elements $\mu^{\alpha,a}$ or $\lambda^{\alpha,a}$.
In order to treat both the even and odd cases at the same time, however, we
shall write
\begin{align*}
 \mu^{0,a} &= 1 + \frac{a-1}{2}g \in \Mackey H_{\Z/2}^0(S^0) \\
\intertext{and}
 \lambda^{0,a} &= 1 + \frac{a-1}{2}\kappa \in \Mackey H_{\Z/2}^0(S^0)
\end{align*}
for $p=2$ and $a$ odd. The reader can check that these definitions are consistent with the 
case when $p$ is odd.

\part{The cohomology of $B_GU(1)$}\label{part:BU1}

\section{The topology of $B_GU(1)$}

We introduce our model for the classifying space $B_GU(1)$ and discuss its topology
and its fundamental groupoid.

Recall from Definition~\ref{def:irrRepresentations}
that we write $\C_k$ for the irreducible complex
representation of $G$ given by letting the generator $t$ act by multiplication by $e^{2\pi i k/p}$
on $\C$.
This definition fixes nonequivariant identifications of each $\C_k$ with $\C$,
hence nonequivariant identifications of the $\C_k$ with one another.
The irreducible complex representations of $G$ are given by the $\C_k$
for $0\leq k \leq p-1$, and it is useful to think of the subscript as taken modulo $p$.
For example, we can write $\C_k\tensor_\C \C_\ell \iso \C_{k+\ell}$
as a statement about irreducible representations.

As a model for $B_GU(1)$ we take
\[
 B_GU(1) = \CP_G^\infty = \CP\Bigl(\Dirsum_{k=0}^{p-1} \C_k^\infty\Bigr),
\]
the Grassmannian of complex lines in $\Dirsum_{k} \C_k^\infty$.
Nonequivariantly, this is a copy of $\CP^\infty$.
Its fixed sets are
\[
 B_GU(1)^G = \Disjunion_{k=0}^{p-1} \CP(\C_k^\infty),
\]
the disjoint union of $p$ copies of $\CP^\infty$.
For notational simplicity we shall write $B = B_GU(1)$ and $B_k = \CP(\C_k^\infty)$,
so $B^G = \Disjunion_{k} B_k$.

Let $\omega_G$ denote the canonical complex line bundle over $B$.
Nonequivariantly, it is the usual canonical line bundle over the infinite complex projective space.
Its restriction to $B_k$ is the nonequivariant canonical bundle $\omega$
with $G$ acting on the fibers so that each is isomorphic to $\C_k$.

There is another action of $\Z/p$ on $B$ we will want to take into account in our calculations.
Write elements of $B = \CP(\Dirsum_{k} \C_k^\infty)$ as equivalence classes of $p$-tuples
$[z_0, z_1, \dots, z_{p-1}]$, with $z_k\in \C_k^\infty$ not all 0. We define a $G$-map $\chi\colon B\to B$ by
\[
 \chi[z_0, z_1, \dots, z_{p-1}] = [z_1, z_2, \dots, z_{p-1}, z_0],
\]
using our chosen nonequivariant identifications of the $\C_k$ with one another.
It is straightforward to check that $\chi$ is equivariant, and clearly $\chi^p$ is the identity.
Further, $\chi^G$ cyclicly permutes the spaces $B_k$, inducing homeomorphisms
$\chi^G\colon B_k\to B_{k-1}$, understanding the subscript modulo $p$.
Notice also that $\chi^*\omega_G \iso \omega_G\tensor_\C \C_{p-1}$.
More generally, if $f\colon X\to B$ classifies the line bundle $\theta$ over $X$,
then $\chi f$ classifies $\theta\tensor_\C \C_{p-1}$ and
$\chi^k f$ classifies $\theta\tensor_\C \C_{p-k}$.

The fundamental groupoid $\pi\colon \Pi_G B\to \orb G$
is relatively simple because $B$ and 
the components of its fixed set are all simply connected.
It is equivalent to a category over $\orb G$ having one object
$b$ over $G/e$ and $p$ objects, $b_0$, $b_1$, \dots, $b_{p-1}$, over $G/G$,
corresponding to the components of the fixed set.
The self-maps of $b$ map isomorphically to the self-maps of $G/G$;
there is one map $b\to b_k$ for each $k$, over the map $G/e\to G/G$;
and there are no non-identity self-maps of $b_k$.
We can picture the category as follows:
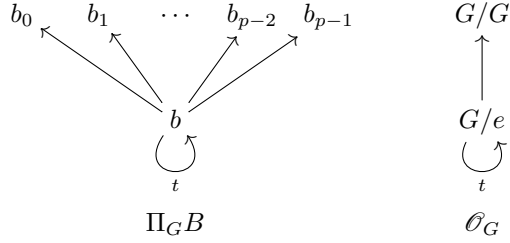
\[
 \xymatrix@!C=.5em{
  b_0 & b_1 & \cdots & b_{p-2} & b_{p-1} & \qquad\qquad & G/G \\
  &&b \ar[ull] \ar[ul] \ar[ur] \ar[urr] \ar@(dl,dr)[]_t &&&& G/e \ar[u] \ar@(dl,dr)[]_t \\
  &&\Pi_G B &&&& \orb{G}
 }
\]
There is also an action of $\chi$, fixing $b$ and cyclicly permuting the
objects $b_k$: $\chi_*(b_k) = b_{k-1}$.

We now want to compute the representation ring $RO(\Pi_G B)$.
We consider $p$ odd and $p=2$ separately. Start with the case where $p$ is odd, and
suppose that $\gamma$ is a real virtual representation of $\Pi_G B$.
We have $\gamma(b) = (G\times \R^i)\ominus (G\times \R^j)$ and, by an abuse of notation,
we shall write this as $\gamma(b) = G\times\R^{i-j} = G\times\R^n$ where
$n\in\Z$. (For this discussion, we care only about the isomorphism type of $\gamma$.)
The action of $t$ on $G\times\R^n$ potentially involves the homotopy class of a nontrivial
linear self-map of $\R^n$ but, because this homotopy class must have odd order $p$,
it must be trivial.
Each $\gamma(b_k)$ is a virtual $G$-representation 
$\R^{n_0}\dirsum\Dirsum_{i=1}^{(p-1)/2} \MM_i^{n_i}$ with
$n_0 + 2\sum_{i} n_i = n$, and no additional restrictions apply.
Thus, $\gamma$ is entirely determined, up to isomorphism, by $p$
virtual representations of $G$, the $\gamma(b_k)$, all of the same dimension.
So we can represent the elements of $RO(\Pi_G B)$
by $p$-tuples $(\alpha_0, \alpha_1, \dots, \alpha_{p-1})$ of virtual real representations of $G$,
all of the same dimension; the group structure is given by vector addition.
The representation ring $RO(G)$ sits inside $RO(\Pi_G B)$ as the set of
constant tuples $(\alpha,\alpha,\dots,\alpha)$, $\alpha\in RO(G)$.

There are a number of bases for this group, but the most useful to us here
is the following.

\begin{proposition}
If $p$ is odd, then 
\[
 RO(\Pi_G B) \iso 
   \{ (\alpha_0, \alpha_1, \dots, \alpha_{p-1}) \mid \alpha_k\in RO(G),\ |\alpha_0| = \cdots = |\alpha_{p-1}| \}
\]
is a free abelian group of rank
$p(p-1)/2 + 1$.
It has as a basis the elements $1 = (1, 1, \dots, 1)$ and 
$\Omega_{i,j}$, $0\leq i \leq p-1$, $1\leq j\leq (p-1)/2$, where
\[
 \Omega_{i,j} = (0, \dots, 0, \MM_j-2, 0, \dots, 0)
\]
with the nonzero element in the $i$th position.
\end{proposition}

\begin{proof}
$RO(G)$ is free abelian of rank $(p+1)/2$, the number of irreducible representations of $G$.
We already gave the argument that $RO(\Pi_G B)$ is the subgroup of
$RO(G)^p$ specified in the proposition. This subgroup is the kernel of the homomorphism
$RO(G)^p \to \Z^{p-1}$ whose $k$th component, $1\leq k\leq p-1$, is
$(\alpha_0, \dots, \alpha_{p-1}) \mapsto |\alpha_0| - |\alpha_k|$.
But, this homomorphism is easily seen to be a split epimorphism, hence 
the kernel is a free abelian group of rank
$p(p+1)/2 - (p-1) = p(p-1)/2 + 1$.

It is then straightforward to show that any element of $RO(\Pi_G B)$
can be written as a linear combination of the $p(p-1)/2+1$ elements
$1$ and $\Omega_{i,j}$, hence these elements must form a basis.
\end{proof}

If $\alpha = (\alpha_0, \dots, \alpha_{p-1}) \in RO(\Pi_G B)$, we write
$|\alpha|\in \Z$ for the common dimension $|\alpha_k|$.
Note that, if we write an element $\alpha$ in terms of the basis as
\[
 \alpha = n + \sum_{i,j} n_{i,j}\Omega_{i,j},
\]
then
\[
 \alpha_k = n + \sum_{j} n_{k,j} (\MM_j-2).
\]
Note also that
\[
 \MM_j - 2 = \sum_i \Omega_{i,j}
\]
in $RO(\Pi_G B)$.

The action of $\chi$ on $RO(\Pi_G B)$ is given by
\[
 \chi(\alpha_0, \alpha_1, \dots, \alpha_{p-1})
  = (\alpha_{p-1}, \alpha_0, \alpha_1 \dots, \alpha_{p-2})
\]
(writing $\chi$ again for the induced map on the representation ring),
so $\chi(1) = 1$ and
\[
 \chi\Omega_{i,j} = \Omega_{i+1,j}.
\]

The canonical complex line bundle $\omega_G$ induces a representation 
\[
 \omega_G^*  \in RO(\Pi_G B). 
\]
Examining the real representations that occur over each
component of the fixed set, we see that
\begin{align*}
  \omega_G^* &= (2, \MM_1, \MM_2, \dots, \MM_{(p-1)/2}, \MM_{(p-1)/2}, \dots, \MM_1) \\
    &= 2 + \sum_{i=1}^{p-1} \Omega_{i,i} 
\end{align*}
As here, we allow ourselves to write $\Omega_{i,j}$ for any integers $i$ and $j$,
with $j\not\equiv 0 \pmod p$, understanding
the subscript $i$ as taken modulo $p$ and the subscript $j$
also taken modulo $p$, but with $j$ identified with $p-j$ as well.
(Formally, index on $i\in\Z/p$ and $j\in(\Z/p)^\times/\{\pm 1\}$.)

Now consider the case $p=2$ and again let $\gamma$ be a real virtual representation
of $\Pi_G B$.
We can write $\gamma(b) = G\times\R^n$, but this time there are two possible
actions of $t$ on $\gamma(b)$: the map $t_+$ in which $t$ acts trivially on $\R^n$
and the map $t_-$ in which $t$ acts on $\R^n$ by an orientation-reversing linear map.
If $\gamma(b_0) = \R^{n_0} \dirsum \LL^{n_1}$ with $n_0 + n_1 = n$, when is there a $G$-map
$G\times\R^n\to \R^{n_0} \dirsum \LL^{n_1}$ invariant (up to homotopy) with respect to the action of $t$?
If the action of $t$ on $G\times\R^n$ is by $t_+$, then we must have $n_1$ even, whereas, if
the action of $t$ is by $t_-$, we must have $n_1$ odd. The same applies to $\gamma(b_1)$,
leading to the following.

\begin{proposition}
If $p = 2$, then
\[
 RO(\Pi_G B)
  \iso \{ (\alpha_0, \alpha_1) \mid \alpha_k\in RO(G), |\alpha_0| = |\alpha_1|,
      \text{ and }\alpha_0^G \equiv \alpha_1^G \pmod 2 \}
\]
is a free abelian group of rank 3. It has as a basis the elements
\begin{align*}
 1 &= (1, 1) \\
 \LL &= (\LL, \LL) \\
\intertext{and}
 \Omega &= (1-\LL, \LL-1).
\end{align*}
\end{proposition}

\begin{proof}
The argument that $RO(\Pi_GB)$ is the indicated subgroup of $RO(G)^2$
was given just before the statement of the proposition.
Note that, if $|\alpha_0| = |\alpha_1|$, then $\alpha_0^G \equiv \alpha_1^G \pmod 2$
if and only if $|\alpha_0 - \alpha_0^G| \equiv |\alpha_1 - \alpha_1^G| \pmod 2$.

Now, for any $(\alpha_0,\alpha_1) \in RO(\Pi_GB)$, let
$n = (\alpha_0^G-\alpha_1^G)/2$. Then
\[
 (\alpha_0, \alpha_1) = 
  (\alpha_0^G - n)\cdot 1 + (|\alpha_0|-\alpha_0^G + n)\cdot\LL
    + n\cdot\Omega.
\]
On the other hand, if $a\cdot 1 + b\cdot\LL + c\cdot\Omega = 0$, 
then, looking at components, it is easy to see that we must have
$a = b = c = 0$. Hence, $\{1, \LL, \Omega\}$ is a basis.
\end{proof}

It will be useful to use the following elements:
\begin{align*}
 \Omega_{0,1} &= (\MM_1 - 2, 0) = (2\LL - 2, 0) = -1 + \LL - \Omega \\
 \Omega_{1,1} &= (0, \MM_1 - 2) = (0, 2\LL - 2) = -1 + \LL + \Omega,
\end{align*}
named similarly to the elements used in the case $p$ odd.
Note, however, that $\{1, \Omega_{0,1}, \Omega_{1,1}\}$
is {\em not} a basis, as $\Omega$ is not in its span.
We could use either $\{1, \LL, \Omega_{0,1}\}$ or $\{1, \LL, \Omega_{1,1}\}$
as a basis, but the lack of symmetry makes neither choice particularly appealing.
As with $p$ odd, we will consider $\Omega_{i,j}$ as indexed
on $i\in\Z/2$ and $j\in (\Z/2)^\times = \{[1]\}$.

The involution $\chi$ acts on $RO(\Pi_GB)$ by 
$\chi(\alpha_0,\alpha_1) = (\alpha_1,\alpha_0)$, so we have
\begin{align*}
 \chi(1) &= 1 \\
 \chi(\LL) &= \LL \\
 \chi(\Omega) &= -\Omega \\
 \chi(\Omega_{0,1}) &= \Omega_{1,1} \\
 \chi(\Omega_{1,1}) &= \Omega_{0,1}.
\end{align*}

The canonical line bundle $\omega_G$ induces a representation
$\omega_G^*$, with
\[
 \omega_G^* = (2, 2\LL) = 2 + \Omega_{1,1} = 1 + \LL + \Omega.
\]

The following observation shows one of the situations in which we can
treat the cases $p=2$ and $p$ odd simultaneously.

\begin{lemma}\label{lem:kernelproj}
For any prime $p$ and any $0\leq k \leq p-1$, the kernel of the projection
$RO(\Pi_G B)\to RO(G)$ given by $(\alpha_i)_i \mapsto \alpha_k$
is the free abelian subgroup generated by the elements $\Omega_{i,j}$ with
$i\neq k$.
\end{lemma}

\begin{proof}
If $\alpha\in RO(\Pi_G B)$ and $\alpha_k = 0$, then $|\alpha_i| = 0$
for all $i$, because the dimensions are all equal, and $\alpha_i^G$ is even for all $i$,
because the fixed sets all have the same parity.
The set of elements $\beta \in RO(G)$ with $|\beta| = 0$ and $\beta^G$ even
is the subgroup generated by the $\MM_j-2$, regardless of whether $p=2$ or $p$ is odd.
Applying this to the components $\alpha_i$, $i\neq k$, we get the claim of the lemma.
\end{proof}

There are several subsets of $RO(\Pi_G B)$ we shall want notations for.
Recall Definition~\ref{def:ro0G}.

\begin{definition}\label{def:roPi0G}
If $\alpha\in RO(\Pi_G B)$, recall that we write $|\alpha|$ for the common dimension
$|\alpha_k|$ of each component. Let
\begin{align*}
 I^\ev(\Pi_G B) &= \{ \alpha\in RO(\Pi_G B) \mid |\alpha| = 0 \text{ and }
                               \alpha_k^G \text{ even } \forall k \} \\
 RO_0(\Pi_G B) &= \{ \alpha\in RO(\Pi_G B) \mid |\alpha| = 0 \text{ and }
                               \alpha_k^G = 0\ \forall k \} \\ 
 RO_+(\Pi_G B) &= \Big\{ \sum_{i,j} n_{i,j}\Omega_{i,j} \mid n_{i,j} \geq 0\ \forall i, j \Big\}.
\end{align*}
\end{definition}

It is an easy exercise to show that $I^\ev(\Pi_G B)$ is the free subgroup generated by
the $\Omega_{i,j}$, for all $i$ and $j$, and that
$RO_0(\Pi_G B)$ is the free subgroup generated by the differences $\Omega_{i,j}-\Omega_{i,1}$
for $j\neq 1$.
Note also that $\alpha\in I^\ev(\Pi_G B)$ if and only if $\alpha_k \in I^\ev(G)$ for all $k$,
and similarly for the other two subsets.

\section{The cohomology of $B_GU(1)^G$}

In this section we calculate the $RO(\Pi_GB)$-graded cohomologies
of $B_+^G$, $B_+^G\smsh_{B} EG_+$, and $B^G_+\smsh_{B}\tE G$ as ex-spaces over $B$.
To the extent possible, we will consider the cases $p=2$ and $p$ odd simultaneously,
as the results are essentially the same in both cases.

Recall that 
\[
 B^G = \Disjunion_{k=0}^{p-1} B_k,
\]
where each $B_k$ is a copy of $\CP^\infty$.
So, we begin by considering the equivariant cohomology of this nonequivariant space,
for which we need a couple of general results.

\begin{proposition}\label{prop:cohomologyTrivialAction}
Let $X$ be a based space with trivial $G$-action and let $\Mackey T$ be a Mackey functor.
Then, in integer grading,
\[
 \Mackey H_G^n(X;\Mackey T)(G/K) \iso \tilde H^n(X;\Mackey T(G/K)),
\]
naturally in $K$.
\end{proposition}

\begin{proof}
Both sides may be considered as integer-graded nonequivariant cohomology theories in $X$.
They obey the same dimension axiom, so, by the uniqueness of ordinary cohomology,
they must be isomorphic. Naturality in $K$ follows similarly. 
\end{proof}

\begin{proposition}
Let $C$ and $D$ be $G$-spaces,
let $X$ be an ex-space over $C$, and let $Y$ be an ex-space over $D$.
If $\alpha\in RO(\Pi_G C)$, and $\beta\in RO(\Pi_G D)$,
there is a spectral sequence
\[
 E_2^{p,q} = H_G^{\alpha+p}(X;\Mackey H_G^{\beta+q}(Y))
   \convto \Mackey H_G^{\alpha+\beta+p+q}(X\smsh_{C\times D} Y).
\]
\end{proposition}

\begin{proof}
We can consider $\Mackey H_G^*(X\smsh_{C\times D} Y)$ as a cohomology theory in $X$
with coefficient system $\Mackey H_G^*(Y)$.
The result then follows from the Atiyah-Hirzebruch spectral sequence 
from \cite{CW:ordinaryhomology}. 
\end{proof}

Recall our convention that $H_G^*$ means
$RO(\Pi_GB)$-graded cohomology and
$H_G^\bullet$ means $RO(G)$-graded cohomology.
Integer-graded nonequivariant cohomology will continue to be written as $H^*$.

\begin{proposition}\label{prop:integerCPinfty}
Let $\CP^\infty$ be the infinite complex projective space considered as a $G$-space
with trivial $G$-action. Then, in $RO(G)$ grading, we have
\begin{align*}
 \Mackey H_G^\bullet(\CP^\infty_+) &\iso \Mackey H_G^\bullet(S^0)[\sigma], \\
 \Mackey H_G^\bullet(\CP^\infty_+\smsh EG_+) &\iso 
        \Mackey H_G^\bullet(S^0)[\sigma]\tensor_{\Mackey H_G^\bullet(S^0)} \Mackey H_G^\bullet(EG_+)
        \iso \Mackey H_G^\bullet(EG_+)[\sigma], \text{ and} \\
 \Mackey H_G^\bullet(\CP^\infty_+\smsh\tE G) &\iso 
        \Mackey H_G^\bullet(S^0)[\sigma]\tensor_{\Mackey H_G^\bullet(S^0)} \Mackey H_G^\bullet(\tE G),
\end{align*}
where $\sigma$ is the Euler class of the canonical complex line bundle $\omega$
over $\CP^\infty$ with trivial $G$-action, so $|\sigma| = 2$. 
Moreover, the long exact sequence coming from the cofibration sequence
\[
 \CP^\infty_+\smsh EG_+ \to \CP^\infty_+ \to \CP^\infty_+\smsh\tE G
\]
is given by tensoring $\Mackey H_G^\bullet(S^0)[\sigma]$ with the long exact sequence
coming from the cofibration sequence
$EG_+\to S^0 \to \tE G$.
\end{proposition}

\begin{proof}
We consider the spectral sequences
\[
 E_2^{p,q} = \Mackey H_G^p(\CP^\infty_+; \Mackey H_G^{\alpha+q}(Y))
  \convto \Mackey H_G^{\alpha+p+q}(\CP^\infty_+\smsh Y)
\]
where $Y$ is $S^0$, $EG_+$, or $\tE G$,
and all spaces are taken as parametrized over a point.
By Proposition~\ref{prop:cohomologyTrivialAction},
we can write the $E_2$ page in terms of nonequivariant cohomology:
\[
 E_2^{p,q} = H^p(\CP^\infty; \Mackey H_G^{\alpha+q}(Y))
  \convto \Mackey H_G^{\alpha+p+q}(\CP^\infty_+\smsh Y)
\]
But, from the nonequivariant calculation of the cohomology of $\CP^\infty$, we know that
\[
 H^*(\CP^\infty;\Mackey H_G^{\alpha+q}(Y))
  \iso H^*(\CP^\infty;\Z)\tensor \Mackey H_G^{\alpha+q}(Y)
  \iso \Z[\sigma]\tensor \Mackey H_G^{\alpha+q}(Y)
\]
with $|\sigma| = 2$.
Because $H^*(\CP^\infty;\Z)$ is concentrated in even grading, any nonzero differential
on $H^p(\CP^\infty;\Z)\tensor \Mackey H_G^{\alpha+q}(Y)$ would have to increase $p$ by an even
amount and decrease $q$ by an odd amount.
Examination of the 
cohomologies of $S^0$, $EG_+$, and $\tE G$ show that there are no nonzero maps
possible on $\Mackey H_G^{\alpha+q}(Y)$ that lower $q$ by an odd amount.
Hence, all differentials must vanish in all three spectral sequences.
Therefore, the pairing
\[
 H^p(\CP^\infty;\Z)\tensor \Mackey H_G^{\alpha+q}(Y) \to 
   \Mackey H_G^{\alpha+p+q}(\CP^\infty_+\smsh Y)
\]
must be an isomorphism in each case.
We may take $\sigma$ to be the equivariant Euler class of the canonical bundle $\omega$,
which restricts to the nonequivariant Euler class.
This gives the results claimed in the proposition.
\end{proof}

Fix a $k$ and consider $B_k\to B$ as a space over $B$.
The following result tells us the $RO(\Pi_G B)$-graded cohomology of $B_k$.

\begin{proposition}\label{prop:calcFixedPoints}
For each $k$, we have the following calculation of
cohomology graded on $RO(\Pi_GB)$:
\begin{align*}
 \Mackey H_G^*((B_k)_+) 
   \iso \Mackey H_G^\bullet(S^0)&[\sigma_k, \zeta_{i,j,k}, \zeta_{i,j,k}^{-1} \mid
             0\leq i\leq p-1,\ i\neq k,\ 1\leq j\leq p/2] \\
 \Mackey H_G^*((B_k)_+\smsh_B EG_+) 
   &\iso \Mackey H_G^*((B_k)_+)\tensor_{\Mackey H_G^\bullet(S^0)} \Mackey H_G^\bullet(EG_+) \\
 \Mackey H_G^*((B_k)_+\smsh_B \tE G) 
   &\iso \Mackey H_G^*((B_k)_+)\tensor_{\Mackey H_G^\bullet(S^0)} \Mackey H_G^\bullet(\tE G),
\end{align*}
where $|\sigma_k| = 2$ and $|\zeta_{i,j,k}| = \Omega_{i,j}$.
Therefore,
\begin{align*}
 \Mackey H_G^*(B^G_+)
   &\iso \Dirsum_{k=0}^{p-1}\Mackey H_G^*((B_k)_+)  \\
 \Mackey H_G^*(B^G_+\smsh_B EG_+) 
   &\iso \Mackey H_G^*(B^G_+)\tensor_{\Mackey H_G^\bullet(S^0)} \Mackey H_G^\bullet(EG_+) \\
 \Mackey H_G^*(B^G_+\smsh \tE G) 
   &\iso \Mackey H_G^*(B^G_+)\tensor_{\Mackey H_G^\bullet(S^0)} \Mackey H_G^\bullet(\tE G).
\end{align*}
\end{proposition}

\begin{proof}
Because $B_k$ is simply connected and has trivial $G$-action, its
fundamental groupoid is equivalent to $\orb G$. The restriction map
$RO(\Pi_G B) \to RO(\Pi_G B_k) \iso RO(G)$ is given by
$(\alpha_i)_i \mapsto \alpha_k$.
By Lemma~\ref{lem:kernelproj}, the kernel is the free abelian subgroup generated
by the elements $\Omega_{i,j}$, $i\neq k$.
For each $\Omega_{i,j}$ with $i\neq k$, this argument shows that we have a well-defined isomorphism
$\Mackey H_G^{\Omega_{i,j}}((B_k)_+) \iso \Mackey H_G^0((B_k)_+)$,
and we let $\zeta_{i,j,k}\in \Mackey H_G^{\Omega_{i,j}}((B_k)_+)$
be the element corresponding to $1 \in \Mackey H_G^0((B_k)_+)$.
Looking at $-\Omega_{i,j}$, we find an element that acts as the inverse
of $\zeta_{i,j,k}$, hence
$\zeta_{i,j,k}$ is a unit.
It is easy to see now that these elements, together with the
$RO(G)$-graded part,
determine the whole $RO(\Pi_GB)$-graded cohomology, in the sense that
\[
 \Mackey H_G^*((B_k)_+) \iso 
   \Mackey H_G^\bullet((B_k)_+)[\zeta_{i,j,k},\zeta_{i,j,k}^{-1} \mid
     0\leq i\leq p-1,\ i\neq k,\ 1\leq j\leq p/2].
\]
The same argument applies to $(B_k)_+\smsh_B EG_+$ and $(B_k)_+\smsh_B\tE G$.
Combined with the preceding proposition, we get the results stated.
\end{proof}

\begin{definition}
Let $\alpha\in RO(\Pi_G B)$ with $\alpha_k = 0$, so
$\alpha = \sum_{i\neq k}\sum_j n_{i,j}\Omega_{i,j}$ for some $n_{i,j}\in\Z$. Then we write
\[
 \zeta_k^\alpha = \prod_{i\neq k}\prod_j \zeta_{i,j,k}^{n_{i,j}}.
\]
\end{definition}

With this notation we can write
\[
 \Mackey H_G^*((B_k)_+) 
  \iso \Mackey H_G^\bullet(S^0)[\sigma_k, \zeta_k^\alpha \mid \alpha\in RO(\Pi_G B) \text{ with }\alpha_k = 0],
\]
where we implicitly understand the relations $\zeta_k^\alpha \zeta_k^\beta = \zeta_k^{\alpha+\beta}$.

We shall need the following calculation in the next section.
Note that $\omega_G$ is an $\omega_G^*$-dimensional bundle
in the sense of \cite{CMW:orientation} and \cite{CW:ordinaryhomology},
so we have the Euler class
$c = e(\omega_G)\in \Mackey H_G^{\omega_G^*}(B_+)$.

\begin{proposition}\label{prop:calcEulerClass}
For $0\leq k\leq p-1$, the restriction of $c = e(\omega_G)$ to $B_k$ is
\[
 c|B_k = 
   (e_k + \xi_k\sigma_k)\zeta_k^{\omega_G^*-\MM_k}
   \in \Mackey H_G^*((B_k)_+).
\]
\end{proposition}

\begin{proof}
Under the map $RO(\Pi_GB)\to RO(G)$ induced by the inclusion of $B_k$,
$\omega_G^*$ maps to $\MM_k$. 
(Recall that $\MM_0 = 2$ and,
if $p=2$, $\MM_1 = \LL^2$.)

Considering the $RO(G)$-graded cohomology $\Mackey H_G^\bullet((B_k)_+)$ for the moment,
Proposition~\ref{prop:integerCPinfty} and
our calculations of the cohomology of a point show that
\[
 \Mackey H_G^{\MM_k}((B_k)_+)\iso 
   \begin{cases}
     \conc\Z\dirsum \Mackey R\Z & \text{if $k\neq 0$} \\
     \Mackey R\Z & \text{if $k=0$,}
   \end{cases}
\]
generated
by $e_k$ and $\xi_k\sigma_k$ (using the convention that $e_0 = 0$),
hence $c|B_k = ae_k + b\xi_k\sigma_k$ for some integers $a$ and $b$.
Consider the restriction to $G/e$: $e_k$ restricts to 0,
$\xi_k$ restricts to 1, and $\sigma_k$ restricts
to the Euler class of the canonical nonequivariant line bundle.
Therefore, $ae_k+b\xi_k\sigma_k$ restricts to $b$ times this Euler class, 
while $c|B_k$ restricts to 1 times the Euler class; therefore, $b=1$.

Now consider the restriction of $c$ to any point in $B_k$:
it must restrict to the Euler class of the corresponding fiber of $\omega_G$,  which
is $e_k$. Noting that $\sigma_k$ restricts to 0 at any point, we must therefore have $a=1$.
Thus, $c|B_k = e_k + \xi_k\sigma_k$ in the $RO(G)$-grading.

Extending the grading to $RO(\Pi_GB)$, we must multiply by the appropriate unit
to adjust the grading of $c|B_k$ to $\omega_G^*$, and that unit is
$\zeta_k^{\omega_G^*-\MM_k}$.
Note that $(\omega_G^*-\MM_k)_k = 0$,
so $\zeta_k^{\omega_G^*-\MM_k}$ is defined.
\end{proof}

Because $\chi$ cyclicly permutes the fixed-point components $B_k$, we get
the following immediate corollary.

\begin{corollary}
$(\chi^i c)|B_k = 
   (e_{k-i} + \xi_{k-i}\sigma_k)\zeta_k^{\chi^i\omega_G^*-\MM_{k-i}}
   \in \Mackey H_G^*((B_k)_+)$.
\qed
\end{corollary}

\section{The cohomologies of $B_GU(1)_+\smsh EG_+$ and $B_GU(1)_+\smsh\tE G$}

We start with an easy result.

\begin{proposition}\label{prop:BEGtilde}
The inclusion $B^G\to B$ induces an isomorphism
\begin{align*}
 \Mackey H_G^*(B_+\smsh_B \tE G) &\iso \Mackey H_G^*(B_+^G\smsh_B \tE G) \\
  &\iso \Dirsum_{k=0}^{p-1}\Mackey H_G^\bullet(S^0)[\sigma_k, \zeta_k^\alpha \mid
             \alpha_k = 0]
           \tensor_{\Mackey H_G^\bullet(S^0)}\Mackey H_G^\bullet(\tE G).
\end{align*}
\end{proposition}

\begin{proof}
For any ex-$G$-space $X$ over $B$, the inclusion $X^G\to X$ induces a weak equivalence
$X^G\smsh_B\tE G\to X\smsh_B\tE G$, from which the first isomorphism follows.
The second is Proposition~\ref{prop:calcFixedPoints}.
\end{proof}

In one sense, the calculation of the cohomology of $B_+\smsh_B EG_+$ is just as easy,
as we have the following result.

\begin{proposition}\label{prop:BEGplus}
Each inclusion $B_k\to B$ induces an isomorphism
\begin{align*}
 \Mackey H_G^*(B_+\smsh_B EG_+)
  &\iso \Mackey H_G^*((B_k)_+\smsh_B EG_+) \\
  &\iso \Mackey H_G^\bullet(EG_+)[\sigma_k, \zeta_k^\alpha \mid \alpha_k = 0].
\end{align*}
Consequently, the inclusion $B^G\to B$ induces a split short exact sequence
\[
 0
 \to \Mackey H_G^*(B_+\smsh_B EG_+)
 \xrightarrow{\eta} \Mackey H_G^*(B_+^G\smsh_B EG_+)
 \xrightarrow{\theta} \Mackey H_G^{*+1}(B/_B B^G\smsh_B EG_+)
 \to 0.
\]
The last group is
\[
 \Mackey H_G^{*+1}(B/_B B^G\smsh_B EG_+)
  \iso \Dirsum_{k=1}^{p-1} \Mackey H_G^\bullet(EG_+)[\sigma_k,\zeta_k^\alpha \mid \alpha_k = 0].
\]
\end{proposition}

\begin{proof}
For each $k$, the inclusion $B_k\to B$ is a nonequivariant equivalence, hence
$B_k\times EG \to B\times EG$ is an equivariant equivalence.
The explicit calculation of $\Mackey H_G^*((B_k)_+\smsh_B EG_+)$
is part of Proposition~\ref{prop:calcFixedPoints}.

We can take as a splitting of $\eta$ the map
\begin{multline*}
 \eta'\colon \Mackey H_G^*(B_+^G\smsh_B EG_+) \iso \Dirsum_{k=0}^{p-1} \Mackey H_G^*((B_k)_+\smsh_B EG_+) \\
  \to \Mackey H_G^*((B_0)_+\smsh_B EG_+) \iso \Mackey H_G^*(B_+\smsh_B EG_+)
\end{multline*}
given by projection to the 0th summand.
We then have
$\Mackey H_G^{*+1}(B/_B B^G\smsh_B EG_+) \iso \ker\eta'$,
giving the proposition.
\end{proof}

We need more information about the image of $\Mackey H_G^*(B_+\smsh_B EG_+)$
in $\Mackey H_G^*(B^G_+\smsh_B EG_+)$.
Recall that we write $c = e(\omega_G) \in \Mackey H_G^{\omega_G^*}(B_+)$.

\begin{definition}
Let $c\in \Mackey H_G^{\omega_G^*}(B_+\smsh EG_+)$ denote the image of
$c\in \Mackey H_G^{\omega_G^*}(B_+)$ under the projection $B\times EG\to B$.
Let
\[
 \xi_{i,j} \in \Mackey H_G^{\Omega_{i,j}}(B_+\smsh EG_+),
   \quad 1\leq i\leq p-1, \quad 1\leq j\leq p/2
\]
be the element that restricts
to $\zeta_{i,j,0}\in \Mackey H_G^*((B_0)_+\smsh EG_+)$.
For $p/2 < j \leq p-1$, we shall write $\xi_{i,j} = \xi_{i,p-j}$.
\end{definition}

\begin{corollary}\label{cor:BtimesEG}
\[
 \Mackey H_G^*(B_+\smsh EG_+)
 \iso \Mackey H_G^\bullet(EG_+)[c, \xi_{i,j}, \xi_{i,j}^{-1} \mid 1\leq i\leq p-1,\ 1\leq j\leq p/2].
\]
\end{corollary}

\begin{proof}
This is clear from the identification of $\Mackey H_G^*(B_+\smsh_B EG_+)$
with $\Mackey H_G^*((B_0)_+\smsh_B EG_+)$ given by the
preceding proposition, except possibly for the use of $c$ as one
of the generators. This is justified by Proposition~\ref{prop:calcEulerClass},
which implies that
$c\cdot\prod_{i\neq 0}\xi_{i,i}^{-1}$ maps to $\sigma_0 \in \Mackey H_G^*((B_0)_+\smsh_B EG_+)$.
\end{proof}

We introduce notation analogous to the notation $\zeta_k^\alpha$.
Recall that $\xi_j$ is invertible in $\Mackey H_G^\bullet(EG_+)$.

\begin{definition}
In $\Mackey H_G^*(B_+\smsh_B EG_+)$, for each $1\leq j \leq p/2$ let
\[
 \xi_{0,j} = \xi_j\cdot\prod_{i=1}^{p-1} \xi_{i,j}^{-1},
\]
so $\xi_{0,j}$ is a unit with $|\xi_{0,j}| = \Omega_{0,j}$, and
\[
 \prod_{i=0}^{p-1} \xi_{i,j} = \xi_j.
\]
For $\alpha\in I^\ev(\Pi_G B)$ (recall Definition~\ref{def:roPi0G}), so
$\alpha = \sum_{i,j}n_{i,j}\Omega_{i,j}$, let
\[
 \bar\xi^\alpha = \prod_{i,j} \xi_{i,j}^{n_{i,j}}.
\]
\end{definition}

With this notation we can write the result of the corollary above as
\[
 \Mackey H_G^*(B_+\smsh EG_+)
 \iso \Mackey H_G^\bullet(EG_+)[c, \bar\xi^\alpha \mid \alpha\in I^\ev(\Pi_G B)],
\]
where we implicitly understand the relations $\bar\xi^\alpha\bar\xi^\beta = \bar\xi^{\alpha+\beta}$
and $\bar\xi^{\alpha} = \xi^\alpha$ if $\alpha\in RO(G)$.

Write elements of $\Mackey H_G^*(B^G_+\smsh_B EG_+)$ as $p$-tuples
$(x_k)_{0\leq k\leq p-1}$ of elements $x_k\in \Mackey H_G^*((B_k)_+\smsh_B EG_+)$.

\begin{proposition}
Under the inclusion 
$\eta\colon \Mackey H_G^*(B_+\smsh_B EG_+) \includesin \Mackey H_G^*(B^G_+\smsh_B EG_+)$
we have
\begin{align*}
  \eta(c) &= 
      ((e_k + \xi_k\sigma_k)\zeta_k^{\omega_G^*-\MM_k})_k \\
  \eta(\bar\xi^\alpha) &= (\xi^{\alpha_k}\zeta_k^{\alpha-\alpha_k})_k 
  		\qquad\text{for } \alpha\in I^\ev(\Pi_G B).
\end{align*}
\end{proposition}

\begin{proof}
The calculation of $\eta(c)$ was done in Proposition~\ref{prop:calcEulerClass}.
We have that 
\[
 \bar\xi^\alpha|B_0 = \xi^{\alpha_0}\zeta_0^{\alpha-\alpha_0},
\]
by definition, so assume that $k\neq 0$.
The element $\xi_{i,j}|B_k$ is a unit in grading $\Omega_{i,j}$.
For $i\neq 0$ and $i\neq k$, it is easy to see from the definitions that
we must have
\[
 \xi_{i,j}|B_k = \zeta_{i,j,k} = \zeta_k^{\Omega_{i,j}} \qquad k \neq 0,\ i\neq 0, k.
\]
Consider $\xi_{k,j}|B_k$: This is a unit in
$\Mackey H_G^{\Omega_{k,j}}((B_k)_+\smsh_B EG_+) \iso \Mackey R\Z$
and, from the definitions, restricts to the identity at level $G/e$.
This characterizes the element 
$\xi_j\prod_{i\neq k}\zeta_{i,j,k}^{-1} = \xi_j\zeta_k^{\Omega_{k,j}-(\MM_j-2)}$,
so we must have
\[
 \xi_{k,j}|B_k =  \xi_j\zeta_k^{\Omega_{k,j}-(\MM_j-2)} \qquad k\neq 0.
\]
Finally, restricting the equation $\xi_{0,j} = \xi_j\prod_{i\neq 0}\zeta_{i,j}^{-1}$
to $B_k$ gives
\begin{align*}
 \xi_{0,j}|B_k 
  &= \xi_j\Big(\prod_{i\neq 0, k}\zeta_{i,j,k}^{-1}\Big) \Big(\xi_j\zeta_k^{\Omega_{k,j}-(\MM_j-2)}\Big)^{-1} \\
  &= \xi_j\zeta_k^{\Omega_{0,j}+\Omega_{k,j} - (\MM_j-2)}\xi_j^{-1}\zeta_k^{\MM_j-2-\Omega_{k,j}} \\
  &= \zeta_k^{\Omega_{0,j}}.
\end{align*}
The proposition now follows.
\end{proof}

The following result shows how the $\chi^i c$ are related when considered
as elements of $\Mackey H_G^*(B_+\smsh_B EG_+)$.

\begin{proposition}\label{prop:chicEG}
In $\Mackey H_G^*(B_+\smsh_B EG_+)$ we have
\[
 \chi^i c 
   = c\cdot \bar\xi^{\chi^i\omega_G^* - \omega_G^*} - ie_1 \bar\xi^{\chi^i\omega_G^*-\MM_1}.
\]
\end{proposition}

\begin{proof}
Because $\eta$ is injective, it suffices to show this equality
in $\Mackey H_G^*(B^G_+\smsh_B EG_+)$.
Recall that
\[
 \eta(\chi^i c) = ((e_{k-i}+\xi_{k-i}\sigma_k)\zeta_k^{\chi^i\omega_G^*-\MM_{k-i}})_k.
\]
In $\Mackey H_G^\bullet(EG_+)$, where $pe_1 = 0$, we have
\[
 e_j = \mu^{\MM_j-\MM_1,j}e_1 = j\lambda^{\MM_j-\MM_1,j^{-1}}e_1 = je_1\xi^{\MM_j-\MM_1}
\]
when $0\leq j \leq p/2$. When $p$ is odd and $(p+1)/2 \leq j \leq p-1$, we have
\[
 e_j = -\mu^{\MM_{p-j}-\MM_1,-j}e_1 = -(-j)\lambda^{\MM_{p-j}-\MM_1,j^{-1}}e_1 = je_1\xi^{\MM_j-\MM_1},
\]
using our usual convention that $\MM_j = \MM_{p-j}$. On the other hand, we have
\[
 \xi_j = \xi_1\xi^{\MM_j-\MM_1}
\]
for any $j$. So we can write
\begin{align*}
 \eta(\chi^i c) &= ((e_{k-i}+\xi_{k-i}\sigma_k)\zeta_k^{\chi^i\omega_G^*-\MM_{k-i}})_k \\
  &= (((k-i)e_1 + \xi_1\sigma_k)\xi^{\MM_{k-i}-\MM_1}\zeta_k^{\chi^i\omega_G^*-\MM_{k-i}})_k \\
  &= ((ke_1 + \xi_1\sigma_k)\xi^{\MM_{k-i}-\MM_1}\zeta_k^{\chi^i\omega_G^*-\MM_{k-i}})_k \\
  &\qquad - (ie_1 \xi^{\MM_{k-i}-\MM_1}\zeta_k^{\chi^i\omega_G^*-\MM_{k-i}})_k \\
  &= ((ke_1+\xi_1\sigma_k)\xi^{\MM_k-\MM_1}\zeta_k^{\omega_G^*-\MM_k}
        \cdot\xi^{\MM_{k-i}-\MM_1}\zeta_k^{\chi^i\omega_G^* - \omega_G^* - \MM_{k-i}+\MM_1})_k \\
  &\qquad  - ie_1(\xi^{\MM_{k-i}-\MM_1}\zeta_k^{\chi^i\omega_G^*-\MM_{k-i}})_k \\
  &= \eta(c\cdot \bar\xi^{\chi^i\omega_G^* - \omega_G^*}) - \eta(ie_1 \bar\xi^{\chi^i\omega_G^*-\MM_1}),
\end{align*}
which proves the claim.
\end{proof}


\section{Preliminary results on the cohomology of $B_GU(1)$}

Many of our arguments in the remainder of our calculation will be based on the following diagram, in which
all of the rows and columns are parts of long exact sequences:
\[
 \xymatrix{
   \Mackey H_G^*(B_+\smsh_B\tE_G) \ar[r]^\eta_\iso \ar[d]_{\psi}
     & \Mackey H_G^*(B_+^G\smsh_B\tE_G) \ar[d]^\psi \\
   \Mackey H_G^*(B_+) \ar[r]^\eta \ar[d]_\phi
     & \Mackey H_G^*(B_+^G) \ar[r]^-\theta \ar[d]^\phi
     & \Mackey H_G^{*+1}(B/_B B^G) \ar[d]_\iso^\phi \\
   \Mackey H_G^*(B_+\smsh_B EG_+) \ar@{>->}[r]^\eta \ar[d]_\delta
     & \Mackey H_G^*(B_+^G\smsh_B EG_+) \ar@{->>}[r]^-\theta \ar[d]^\delta
     & \Mackey H_G^{*+1}(B/_B B^G\smsh_B EG_+) \\
   \Mackey H_G^{*+1}(B_+\smsh_B\tE G) \ar[r]^\eta_\iso
     & \Mackey H_G^{*+1}(B_+^G\smsh_B\tE G)
  }
\]
That the map $\eta$ in the top row of the diagram is an isomorphism
is part of Proposition~\ref{prop:BEGtilde}.
That implies that $\Mackey H_G^*(B/_B B^G\smsh_B\tE G) = 0$;
alternatively, this vanishing follows from the fact that
$(B/_B B^G)^G$ is the trivial ex-space over $B^G$.
In turn, this vanishing implies that
$\Mackey H_G^*(B/_B B^G)\iso \Mackey H_G^*(B/_B B^G\smsh_B EG_+)$.
Finally, that $\eta$ is a monomorphism and $\theta$ is an epimorphism in the third row
is part of Proposition~\ref{prop:BEGplus}.

\begin{proposition}\label{prop:etamono}
$\eta\colon \Mackey H_G^\alpha(B_+)\to \Mackey H_G^\alpha(B_+^G)$ is a monomorphism
for $|\alpha| \leq 0$.
If $p=2$, then $\eta$ is also a monomorphism if $|\alpha| > 0$ and $\alpha^G_k$ is even for all $k$.
If $p$ is odd, then $\eta$ is also a monomorphism if $\alpha^G_k$ is even for all $k$.
\end{proposition}

\begin{proof}
From the long exact sequence, we see that $\eta$ is a monomorphism if the group
$\Mackey H_G^{\alpha}(B/_B B^G) = 0$.
In the diagram above we noted that
$\Mackey H_G^{*}(B/_B B^G) \iso \Mackey H_G^{*}(B/_B B^G\smsh_B EG_+)$
and we calculated the latter in
Proposition~\ref{prop:BEGplus}.
As a module over $\Mackey H_G^\bullet(EG_+)$, 
it is a shift up by one of an algebra with multiplicative generators
in gradings $2$ and $\Omega_{i,j}$, with $|2| > 0$ and $|\Omega_{i,j}| = 0\geq 0$.
Further, $\Mackey H_G^\bullet(EG_+)$ is concentrated in gradings $\alpha$
for which $|\alpha|\geq 0$. If $p$ is odd or if $p=2$ and $|\alpha|>0$, it is also concentrated in gradings
for which $\alpha^G_k$ is even for all $k$.
It follows that
\[
 \Mackey H_G^{\alpha}(B/_B B^G) \iso \Mackey H_G^{\alpha}(B/_B B^G\smsh_B EG_+) = 0
\]
for the $|\alpha|$ specified in the statement of the proposition.
\end{proof}

\begin{corollary}\label{cor:xibar}
For every $\alpha\in RO_+(\Pi_G B)$, there is a unique element
$\bar\xi^{\alpha}\in \Mackey H_G^\alpha(B_+)$ such that
\[
 \eta(\bar\xi^\alpha) = (\xi^{\alpha_k}\zeta_k^{\alpha-\alpha_k})_k,
\]
using the calculation of $\Mackey H_G^*(B_+^G)$ in
Proposition~\ref{prop:calcFixedPoints}.
Further, we have
\begin{align*}
 \bar\xi^\alpha\bar\xi^\beta &= \bar\xi^{\alpha+\beta} \\
 \bar\xi^{M_j-2} &= \xi^{M_j-2} = \xi_j \\
\intertext{and}
 \phi(\bar\xi^\alpha) &= \bar\xi^\alpha \in\Mackey H_G^\alpha(B_+\smsh_B EG_+).
\end{align*}
\end{corollary}

\begin{proof}
Because $\alpha\in RO_+(\Pi_G B)$, the element
$z = (\xi^{\alpha_k}\zeta_k^{\alpha-\alpha_k})_k \in \Mackey H_G^*(B_+^G)$ is defined.
We have
\[
 \theta\phi(z) = \theta\eta(\bar\xi^\alpha) = 0,
\]
where the $\bar\xi^\alpha$ used here is the element of that name in $\Mackey H_G^*(B_+\smsh_B EG_+)$.
Because $\theta\phi(z) = \phi\theta(z)$ and the latter map $\phi$ is an isomorphism, it follows that
$\theta(z) = 0$. Because $|\alpha|=0$, $\eta$ maps $\Mackey H_G^\alpha(B_+)$
isomorphically onto $\ker\theta$, hence we have a unique $\bar\xi^\alpha\in\Mackey H_G^\alpha(B_+)$
such that $\eta(\bar\xi^\alpha) = z$, as claimed.
That $\bar\xi^\alpha\bar\xi^\beta = \bar\xi^{\alpha+\beta}$ and that
$\bar\xi^{M_j-2} = \xi_j$ follow from the similar statements
about the $\zeta_k^\alpha$ and uniqueness.
That $\phi(\bar\xi^\alpha) = \bar\xi^\alpha$ follows from the definition of 
the latter element.
\end{proof}

We emphasize that $\bar\xi^\alpha \in \Mackey H_G^*(B_+)$ is defined only for $\alpha\in RO_+(\Pi_G B)$, 
unlike the $\bar\xi^\alpha\in \Mackey H_G^*(B_+\smsh_B EG_+)$, which are defined for all
$\alpha \in I^\ev(\Pi_G B)$. 
Further, $\bar\xi^\alpha$ is generally not invertible in $\Mackey H_G^*(B_+)$, unlike
its image in $\Mackey H_G^*(B_+\smsh_B EG_+)$.
The collection $\{\bar\xi^\alpha\}$ is a multiplicative
subset of $\Mackey H_G^*(B_+)$ generated by the elements
$\xi_{i,j} = \bar\xi^{\Omega_{i,j}}$.

Note also that $\chi\bar\xi^\alpha = \bar\xi^{\chi\alpha}$, as we can see from the fact
that $\chi$ acts by cyclically permuting the summands in
$\Mackey H_G^*(B_+^G)$.

\section{The additive structure of the cohomology of $B_GU(1)$}\label{sec:oddadditive}

There are three classes of elements of $\Mackey H_G^*(B_+)$ that will turn out to
generate the whole algebra. These are
\begin{enumerate}
\item
$\chi^i c \in \Mackey H_G^{\chi^i\omega_G^*}(B_+)$, $0\leq i \leq p-1$. 
These elements satisfy (and are characterized by)
\[
 \eta(\chi^i c) = ((e_{k-i} + \xi_{k-i}\sigma_k)\zeta_k^{\chi^i\omega_G^* - \MM_{k-i}})_k.
\]

\item
$\bar\xi^\alpha \in \Mackey H_G^{\alpha}(B_+)$, $\alpha\in RO_+(\Pi_G B)$.
These elements satisfy
\[
 \eta(\bar\xi^\alpha) = (\xi^{\alpha_k}\zeta_k^{\alpha-\alpha_k})_k.
\]

\item
$\bar\lambda^{\alpha,a} \in \Mackey H_G^\alpha(B_+)$, $\alpha\in RO_0(\Pi_G B)$
and $a = (a_k)_k$ with $a_k \in \nu(\alpha_k)^{-1}$ for all $k$.
These elements satisfy
\[
 \eta(\bar\lambda^{\alpha,a}) = (\lambda^{\alpha_k,a_k}\zeta_k^{\alpha-\alpha_k})_k.
\]

\end{enumerate}
That the elements $\bar\lambda^{\alpha,a}$ exist follows from the fact that
\[
 \phi((\lambda^{\alpha_k,a_k}\zeta_k^{\alpha-\alpha_k})_k) = \bar\xi^\alpha
\]
is in the image of $\Mackey H_G^*(B_+\smsh_B EG_+)$ for any $a$,
and that $\eta$ is a monomorphism in grading $\alpha$, by
Proposition~\ref{prop:etamono}.
It follows from the properties of the $\zeta_k$ and $\lambda$ that
\begin{align*}
 \bar\lambda^{\alpha,a}\bar\lambda^{\beta,b} &= \bar\lambda^{\alpha+\beta,ab},
\intertext{where the product $ab$ is taken coordinate-wise, and}
 \bar\lambda^{\alpha,a} &= \lambda^{\alpha,a}
\end{align*}
if $\alpha\in RO_0(G)$ and $a$ is constant.

We have both $\phi(\bar\xi^\alpha) = \bar\xi^\alpha$ and
$\phi(\bar\lambda^{\alpha,a}) = \bar\xi^\alpha$, but note that
$RO_+(\Pi_G B) \intersect RO_0(\Pi_G B) = 0$, so there is no real overlap here.

Our first goal will be to show that $\Mackey H_G^*(B_+)$ is a free
module over $\Mackey H_G^\bullet(S^0)$. In order to do this, we shall give an
explicit basis, consisting of monomials of the form
\[
 x = 
 \big(\prod_i (\chi^i c)^{m_i} \big)
 \big(\prod_i (\chi^i c \cdot \xi_{i,1})^{\epsilon_i}\big)
 \big(\prod_i \xi_{i,1}^{n_i}\big)
 \bar\lambda^{\beta,b}
\]
for certain allowed values of the $m_i$, $\epsilon_i$, $n_i$, $\beta$, and $b$.
Let
\[
 \gamma = \omega_G^* - \MM_1 + \Omega_{1,0} = \omega_G^* - 2 - \sum_{i=1}^{p-1} \Omega_{1,i}
   = \sum_{i=2}^{p-1} (\Omega_{i,i} - \Omega_{1,i}),
\]
so $\gamma\in RO_0(\Pi_G B)$ and $\omega_G^* = \MM_1 - \Omega_{1,0} + \gamma$.
Then $x$ has grading
\begin{align*}
 |x| &= \sum_i \big(m_i\chi^i\omega_G^* + \epsilon_i(\chi^i\omega_G^*+\Omega_{1,i})
        + n_i\Omega_{1,i}\big) + \beta \\
     &= \sum_i \big((m_i+\epsilon_i)(\MM_1 + \chi^i\gamma) + (n_i-m_i)\Omega_{1,i}\big) + \beta.
\end{align*}
The integer dimension of this grading is
\[
 \|x\| = \sum_i 2(m_i+\epsilon_i).
\]

For any $\alpha\in RO(\Pi_G B)$, we will be interested in the sequence of fixed-point
dimensions $(\alpha^G_*) = (\alpha^G_k)_{0\leq k \leq p-1}$. In the case of $|x|$, these are
\[
 |x|^G_k = 2(m_k - n_k)
\]
from the formula above.
For any $\alpha$, the class of the sequence $(\alpha^G_*) \in \Z^p/\Z$, quotienting
out the diagonal copy of $\Z$, is an invariant of the coset $\alpha+RO(G)$.
It determines an ordering $(k_i)_i$ of $\{0,\dots, p-1\}$ such that
\begin{itemize}
\item $\alpha^G_{k_{i+1}} \leq \alpha^G_{k_i} \qquad$ for $0\leq i < p-1$ and
\item if $\alpha^G_{k_{i+1}} = \alpha^G_{k_i}$ then $k_{i+1} > k_i$.
\end{itemize}
In other words, we order the fixed-set dimensions from highest to lowest, taking
the natural order among dimensions that are the same. Applied to $|x|$, this gives the ordering
$(k_i)$ such that
\begin{itemize}
\item $m_{k_{i+1}} - n_{k_{i+1}} \leq m_{k_i} - n_{k_i}$ and
\item $m_{k_{i+1}} - n_{k_{i+1}} = m_{k_i} - n_{k_i}$ implies that $k_{i+1} > k_i$.
\end{itemize}

\begin{definition}\label{def:admissiblemonomials}
We say that a monomial
\[
 \big(\prod_i (\chi^i c)^{m_i} \big)
 \big(\prod_i (\chi^i c \cdot \xi_{i,1})^{\epsilon_i}\big)
 \big(\prod_i \xi_{i,1}^{n_i}\big)
 \bar\lambda^{\beta,b}
\]
is {\em admissible} if the following are true:
\begin{enumerate}
\item For all $i$, $m_i\geq 0$, $\epsilon_i = 0$ or $1$, and $n_i\geq 0$.
\item If $m_i > 0$ or $\epsilon_i > 0$, then $n_i = 0$.
\item For at least one $i$, $\epsilon_i = 0 = n_i$.
\item Using the ordering of $m_k-n_k$ from highest to lowest as above,
if $\epsilon_{k_i} = 0$ then $\epsilon_{k_j} = 0$ for all $j > i$.
\item $1 \leq b_i \leq p-1$ for all $i$.
\item If $I$ is the least index such that $\epsilon_{k_I} = 0$, then
$\beta_{k_I} = 0$ (hence $b_{k_I} = 1$).
\end{enumerate}
\end{definition}

Let $A$ be the set of admissible monomials.
Our goal now is to show that $A$
is a basis for $\Mackey H_G^*(B_+)$ as a module over $\Mackey H_G^\bullet(S^0)$.
To that end, let
\[
 \Mackey P^* = \Mackey H_G^\bullet(S^0)\{A\}
\]
denote the free $\Mackey H_G^\bullet(S^0)$-module
determined by the $RO(\Pi_G B)$-graded set $A$.
The set map $A\to \Mackey H_G^*(B_+)$ taking each element of $A$ to the corresponding
element of $\Mackey H_G^*(B_+)$ gives rise to a map $f\colon \Mackey P^*\to \Mackey H_G^*(B_+)$
of $\Mackey H_G^\bullet(S^0)$-modules, which we will show to be an isomorphism.

Because $\Mackey P^*$ is free, tensoring with the long exact sequence of the cofibration
$EG_+\to S^0\to \tE G$ gives us a long exact sequence
\begin{multline*}
 \cdots\to (\Mackey P^*\tensorS \Mackey H_G^\bullet(\tE G))^*
  \to \Mackey P^*
  \to (\Mackey P^*\tensorS \Mackey H_G^\bullet(EG_+))^* \\
  \to (\Mackey P^*\tensorS \Mackey H_G^\bullet(\tE G))^{*+1}
  \to \cdots
\end{multline*}
The map $f$ induces the following map of long exact sequences:
\[
 \xymatrix@C-1.5em{
 \cdots \ar[r]
  & \Mackey P^*\tensorS \Mackey H_G^\bullet(\tE G) \ar[r] \ar[d]_{f_{\tE G}}
  & \Mackey P^* \ar[r] \ar[d]_f
  & \Mackey P^*\tensorS \Mackey H_G^\bullet(EG_+) \ar[r] \ar[d]_{f_{EG_+}}
  & \cdots \\
 \cdots \ar[r]
  & \Mackey H_G^*(B_+\smsh_B \tE G) \ar[r]
  & \Mackey H_G^*(B_+) \ar[r]
  & \Mackey H_G^*(B_+\smsh_B EG_+) \ar[r]
  & \cdots
 }
\]
If we can show that both $f_{EG_+}$ and $f_{\tE G}$ are isomorphisms, then it will
follow that $f$ is an isomorphism.

Because we are comparing modules over $\Mackey H_G^\bullet(S^0)$,
it is useful to restrict to one coset $\alpha+RO(G)$ of grading at a time.
To that end, the following result tells us which elements of $A$ have grading
in $\alpha+RO(G)$.

\begin{lemma}\label{lem:Agrading}
Fix an $\alpha\in RO(\Pi_G B)$ and let
$(k_i)$ be the ordering of the fixed-set dimensions $(\alpha^G_*)$ from highest to lowest,
as earlier in this section. The admissible monomials
\[
 x = 
 \big(\prod_i (\chi^i c)^{m_i} \big)
 \big(\prod_i (\chi^i c \cdot \xi_{i,1})^{\epsilon_i}\big)
 \big(\prod_i \xi_{i,1}^{n_i}\big)
 \bar\lambda^{\beta,b}
\]
in grading $\alpha+RO(G)$ are those such that
\begin{itemize}
\item we have 
\[
 (|x|^G_k)_k = (\alpha^G_k)_k + N
\]
for some integer $N\geq -\alpha^G_{k_0}$; and

\item if $I$ is the least index such that $\epsilon_{k_I} = 0$, then
$\beta$ is the unique element of $RO_0(\Pi_G B)$ such that $\beta_{k_I} = 0$ and
\[
 \beta + \sum_i \big((m_i+\epsilon_i)\chi^i \omega_G^* + (\epsilon_i+n_i)\Omega_{1,i}\big)
   \in \alpha + RO(G).
\]
\end{itemize}
For a given $N\geq -\alpha^G_{k_0}$, the number of admissible monomials with
$(|x|^G_*) = (\alpha^G_*) + N$ is the number of indices $k$ with $\alpha^G_k+N \geq 0$.
On the other hand, there is exactly one $x$ with integer dimension $\|x\| = 2\ell$
for each integer $\ell\geq 0$.
\end{lemma}

\begin{proof}
Let $x$ be an admissible monomial with grading in $\alpha+RO(G)$. Because
$|x| + RO(G) = \alpha+RO(G)$, we must have $(|x|^G_*) = (\alpha^G_*) + N$ for some
integer $N$. Recall that $|x|^G_k = 2(m_k-n_k)$ and that, for each $k$,
at most one of $m_k$ and $n_k$ can be nonzero. From this it follows that $N$ determines
the $m_k$ and $n_k$:
\begin{align*}
 m_k &= \begin{cases}
          (\alpha^G_k+N)/2 & \text{if $\alpha^G_k+N \geq 0$} \\
          0 & \text{if $\alpha^G_k+N < 0$}
        \end{cases} \\
\intertext{and}
 n_k &= \begin{cases}
           0 & \text{if $\alpha^G_k+N \geq 0$} \\
         -(\alpha^G_k+N)/2 & \text{if $\alpha^G_k+N < 0$.}
        \end{cases}
\end{align*}
Further, at least one $n_k$ must be 0.
It follows that the least possible $N$ is $-\alpha^G_{k_0}$.

Once the $m_k$, $\epsilon_k$, and $n_k$ are known, $\beta$ is determined by the requirements
that $\beta_{k_I} = 0$ and $|x| \in \alpha+RO(G)$; the statement in the lemma is just another
way of saying that. Note that
\[
 \Big[ \alpha - \sum_i \big((m_i+\epsilon_i)\chi^i \omega_G^* + (\epsilon_i+n_i)\Omega_{1,i}\big)
  \Big]^G_k
  = \alpha^G_k - 2(m_k-n_k)
\]
is constant, so such a $\beta$ exists.
Thus, $x$ must have the form stated in the lemma.

If we fix an $N\geq -\alpha^G_{k_0}$, it determines the $m_k$ and $n_k$ as above.
$N$ does not determine the $\epsilon_k$, but does restrict which ones can be nonzero:
If any are nonzero than they must be an initial sequence $\epsilon_{k_0}$, $\epsilon_{k_1}$, \dots,
$\epsilon_{k_i}$ with $n_{k_{i+1}} = 0$, so that we satisfy the requirements that
$\epsilon_k = 0 = n_k$ for at least one $k$, and that $\epsilon_{k_j} = 0$ implies
that all subsequent ones are also 0. Thus, including the possibility that
all $\epsilon_k$ are 0, the number of choices possible for the $\epsilon_k$
is equal to the number of indices $k$ with $n_k = 0$, which is the same
as the number of indices with $\alpha^G_k+N\geq 0$, as claimed.

Finally, consider the possible integer dimensions
$\|x\| = \sum_i 2(m_i+\epsilon_i)$.
Certainly we must have $\|x\|\geq 0$, because $m_i\geq 0$
and $\epsilon_i\geq 0$. Consider those $x$ with $(|x|^G_*) = (\alpha^G_*)+N$
for a given $N$. 
As above, $N$ determines the $m_i$ and limits the possible $\epsilon_i$.
Using the $m_i$ so determined, the possible integer dimensions of $x$ are then $2\sum_i m_i$, $2+2\sum_i m_i$, \dots,
$2(j-1)+2\sum_i m_i$, where $j$ is the number of indices $k$
with $\alpha^G_k+N\geq 0$. If we increase $N$ by 1, then there are exactly $j$
exponents $m_i$ that will increase by 1, so that the possible integer dimensions
of elements $x$ such that $(|x|^G_*) = (\alpha^G_*) + N+1$ start at $2j + 2\sum_i m_i$ and go up from there.
It is clear now that there is one $x$ with each even nonnegative integer dimension,
with dimension 0 occurring when $N = -\alpha^G_{k_0}$.
\end{proof}

\begin{proposition}\label{prop:fEG+}
$f_{EG_+}\colon \Mackey P^*\tensorS \Mackey H_G^\bullet(EG_+)
\to \Mackey H_G^*(B_+\smsh_B EG_+)$ is an isomorphism.
\end{proposition}

\begin{proof}
View both the source and target of $f_{EG_+}$ as modules over
$\Mackey H_G^\bullet(EG_+)$.
Fix a coset $\alpha+RO(G)$ of the grading and choose a representative
$\alpha$ with $|\alpha| = 0$ for convenience.

By construction, the source is free with a basis of admissible monomials
as described in Lemma~\ref{lem:Agrading}.
We order this basis by the integer dimensions of the monomials, which,
by the lemma, are the even nonnegative integers, without repeats.

By Corollary~\ref{cor:BtimesEG}, $\Mackey H_G^{\alpha+\bullet}(B_+\smsh_B EG_+)$
is a free $\Mackey H_G^\bullet(EG_+)$-modules with a basis given by the set
\[
 \{ \bar\xi^\alpha,\ c\bar\xi^{\alpha-\omega_G^*+2},
        \ c^2\bar\xi^{\alpha-2\omega_G^*+4},\ \dots,
        \ c^n\bar\xi^{\alpha-n\omega_G^*+2n},\ \dots \}.
\]
We order this basis by the powers of $c$. Since 
$\|c^n\bar\xi^{\alpha-n\omega_G^*+2n}\| = 2n$,
this is the same as ordering by integer dimension.

Consider an admissible monomial
\[
 x = 
 \big(\prod_i (\chi^i c)^{m_i} \big)
 \big(\prod_i (\chi^i c \cdot \xi_{i,1})^{\epsilon_i}\big)
 \big(\prod_i \xi_{i,1}^{n_i}\big)
 \bar\lambda^{\beta,b}
\]
with $|x|\in \alpha+RO(G)$. Its image is
\[
 f_{EG_+}(x) = \big(\prod_i (\chi^i c)^{m_i+\epsilon_i}\big)
                  \bar\xi^\delta
\]
for the appropriate $\delta$ (whose formula we could write down, but is unimportant here).
In $\Mackey H_G^*(B_+\smsh_B EG_+)$ we have
\begin{align*}
 \chi^i c 
   &= c\cdot \bar\xi^{\chi^i\omega_G^* - \omega_G^*} - ie_1 \bar\xi^{\chi^i\omega_G^*-\MM_1} \\
   &= (c - ie_1 \bar\xi^{\omega_G^*-\MM_1})\bar\xi^{\chi^i\omega_G^* - \omega_G^*}
\end{align*}
by Proposition~\ref{prop:chicEG}, so we can write
\[
 f_{EG_+}(x) = \big(\prod_i(c - ie_1\bar\xi^{\omega_G^*-\MM_1})^{m_i+\epsilon_i}\big) \bar\xi^{\delta'}
\]
for the appropriate $\delta'$.
If we were to multiply out, the leading term would be
\[
 c^n \bar\xi^{\delta'} = c^n\bar\xi^{\alpha-n\omega_G^*+2n}\cdot\bar\xi^{\delta''},
\]
where
$n = \sum_i(m_i+\epsilon_i) = \|x\|/2$; the other terms would involve lower powers of $c$.
Now, we must have
\[
 n\omega_G^* + \alpha - n\omega_G^* + 2n + \delta'' = \alpha+2n+\delta'' \in \alpha+RO(G),
\]
hence $\delta''\in RO(G)$ and $\bar\xi^{\delta''} = \xi^{\delta''}$
is an invertible element of $\Mackey H_G^\bullet(EG_+)$.
Thus, the (infinite) matrix of $f_{EG_+}$ with respect to the two chosen ordered bases is
upper triangular, with the diagonal elements invertible in $\Mackey H_G^\bullet(EG_+)$.
The inverse of this matrix therefore exists and is again an upper  triangular matrix.
Hence $f_{EG_+}$ is an isomorphism.
\end{proof}

\begin{proposition}\label{prop:ftEG}
$f_{\tE G}\colon \Mackey P^*\tensorS \Mackey H_G^\bullet(\tE G)
\to \Mackey H_G^*(B_+\smsh\tE G)$ is an isomorphism.
\end{proposition}

\begin{proof}
Rather than view both sides as modules over $\Mackey H_G^\bullet(\tE G)$, which has no unit,
it is more convenient to think of them as modules over $\Mackey H_G^\bullet(S^0)[e_i^{-1}]$.
The source is isomorphic to
\[
 \Mackey P^*[e_i^{-1}]\tensor_{\Mackey H_G^\bullet(S^0)[e_i^{-1}]} \Mackey H_G^\bullet(\tE G)
 =
 \Mackey H_G^\bullet(S^0)[e_i^{-1}]\{A\} \tensor_{\Mackey H_G^\bullet(S^0)[e_i^{-1}]}
  \Mackey H_G^\bullet(\tE G).
\]
By Proposition~\ref{prop:BEGtilde}, the target of $f_{\tE G}$ is isomorphic to
\[
 \Big(\Dirsum_{k=0}^{p-1}\Mackey H_G^\bullet(S^0)[e_i^{-1}][\sigma_k,\zeta_k^\alpha \mid \alpha_k = 0]\Big)
  \tensor_{\Mackey H_G^\bullet(S^0)[e_i^{-1}]} \Mackey H_G^\bullet(\tE G).
\]
The map $f_{\tE G}$ is induced by a map
\[
 \bar f_{\tE G}\colon \Mackey P^*[e_i^{-1}] \to \Dirsum_{k=0}^{p-1}\Mackey H_G^\bullet(S^0)[e_i^{-1}][\sigma_k,\zeta_k^\alpha \mid \alpha_k = 0]
\]
of free modules over $\Mackey H_G^\bullet(S^0)[e_i^{-1}]$, so we work on $\bar f_{\tE G}$
for the moment, where the light is better.
Explicitly, for an admissible monomial
\[
 x = 
 \big(\prod_i (\chi^i c)^{m_i} \big)
 \big(\prod_i (\chi^i c \cdot \xi_{i,1})^{\epsilon_i}\big)
 \big(\prod_i \xi_{i,1}^{n_i}\big)
 \bar\lambda^{\beta,b}
\]
we have
\[
 \bar f_{\tE G}(x) =
 \big(\big(\prod_i (e_{k-i} + \xi_{k-i}\sigma_k)^{m_i+\epsilon_i}\big)
 \cdot b_k e^{\beta_k}\xi_1^{\epsilon_k+n_k} \zeta_k^{|x|-|x|_k}\big)_k.
\]
(The exponent on $\zeta_k$ must be $|x|-|x|_k$ as given, because the other terms
have grading in $RO(G)$, hence the grading of the $\zeta_k$ term must be the unique
element of $|x|+RO(G)$ with $k$th component 0.)

As in the proof of the preceding proposition, we concentrate on the elements in gradings
$\alpha+RO(G)$ for a fixed $\alpha$. The basis for the source of $\bar f_{\tE G}$ is characterized
again as in Lemma~\ref{lem:Agrading}, and again we order it by integer dimension.

As a basis for the target of $\bar f_{\tE G}$ we can take the set
\[
 \{ \sigma_k^n \zeta_k^{\alpha-\alpha_k} \mid 0\leq k \leq p-1 \text{ and } n\geq 0 \},
\]
where we are writing $\sigma_k^n \zeta_k^{\alpha-\alpha_k}$ for the $p$-tuple
in grading $\alpha+2n-\alpha_k$ that
is 0 in all components except for the $k$th, which is $\sigma_k^n \zeta_k^{\alpha-\alpha_k}$
(and similarly if we write, say, just $\zeta_k^{\alpha-\alpha_k}$).
Note that the sequence of fixed-point dimensions of the grading of one of these basis elements is
\[
 (|\sigma_k^n \zeta_k^{\alpha-\alpha_k}|^G_*) = (\alpha^G_*) + 2n - \alpha^G_k.
\]
To describe the order we use on this basis, let
$(k_i)$ be the ordering of the fixed-sets $\alpha^G_k$, from highest to lowest.
We order the basis elements by saying that $\sigma_{k_i}^m\zeta_{k_i}^{\alpha-\alpha_{k_i}}$
precedes $\sigma_{k_j}^n\zeta_{k_j}^{\alpha-\alpha_{k_j}}$ if
$2m - \alpha^G_{k_i} < 2n - \alpha^G_{k_j}$ or, if these are equal, if $i < j$.
Thus, we order them by fixed-point dimensions and, within each such dimension, by the order
chosen for the indices. 
For a fixed $N$ (which must be even), the basis elements with fixed-point dimensions
$(\alpha^G_*) + N$ are, in order,
\[
 \sigma_{k_0}^{(\alpha^G_{k_0}+N)/2}\zeta_{k_0}^{\alpha-\alpha_{k_0}},\ \dots,
 \ \sigma_{k_i}^{(\alpha^G_{k_i}+N)/2}\zeta_{k_i}^{\alpha-\alpha_{k_i}},
\]
where $i$ is the largest index such that $\alpha^G_{k_i}+N \geq 0$ (which may be $i = p-1$ if $N$ is large).
So the number of such basis elements is the number of indices $k$ with $\alpha^G_k + N \geq 0$,
not coincidentally matching the number of admissible monomials of the same fixed-point dimensions,
per Lemma~\ref{lem:Agrading}.

Now consider an admissible monomial
\[
 x = 
 \big(\prod_i (\chi^i c)^{m_i} \big)
 \big(\prod_i (\chi^i c \cdot \xi_{i,1})^{\epsilon_i}\big)
 \big(\prod_i \xi_{i,1}^{n_i}\big)
 \bar\lambda^{\beta,b}
\]
With our assumption that $|x| + RO(G) = \alpha + RO(G)$,
we have $|x|-|x|_k = \alpha - \alpha_k$ for all $k$, so we can write
\begin{align*}
 \bar f_{\tE G}(x) 
 &= \sum_k \big(\prod_i (e_{k-i} + \xi_{k-i}\sigma_k)^{m_i+\epsilon_i}\big)
     \cdot b_k e^{\beta_k}\xi_1^{\epsilon_k+n_k} \zeta_k^{\alpha-\alpha_k} \\
 &= \sum_k b_ke^{\beta_k}\xi_1^{\epsilon_k+n_k}
     (e^\delta\sigma_k^{m_k+\epsilon_k} + \cdots)\zeta_k^{\alpha-\alpha_k}
\end{align*}
where $e^\delta = \prod_{i\neq k} e_{k-i}^{m_i+\epsilon_i}$ is invertible and
the terms omitted involve higher powers of $\sigma_k$ with coefficients having as factors
at least one of the $\xi_{k-i}$ for $k\neq i$.
Thinking of this as a linear combination of the basis elements $\sigma_k^n\zeta_k^{\alpha-\alpha_k}$,
and ordering those basis elements as above, the first basis element that appears will be
$\sigma_{k_I}^{m_{k_I}}\zeta_{k_I}^{\alpha-\alpha_{k_I}}$ where $I$ is the least index such
that $\epsilon_{k_I} = 0$. Its coefficient will be
\[
 b_{k_I}e^{\beta_{k_I}+\delta}\xi_1^{\epsilon_{k_I}+n_{k_I}} = e^\delta
\]
because, by the definition of admissible monomial, $\beta_{k_I} = 0$, $b_{k_I} = 1$,
and $\epsilon_{k_I} = 0 = n_{k_I}$.
Moreover, comparing the orderings of the two bases, we see that $x$ appears in the same
place in its order as $\sigma_{k_I}^{m_{k_I}}\zeta_{k_I}^{\alpha-\alpha_{k_I}}$
does in its. Therefore, the matrix of $\bar f_{\tE G}$ with respect to these ordered
bases is lower triangular with invertible elements on the diagonal.

This matrix has a formal inverse, also lower triangular with invertible elements
on the diagonal, but here we run into an issue: Each column may well contain infinitely
many nonzero elements. Therefore, trying to use it to define an inverse to $\bar f_{\tE G}$
would lead to infinite sums. So let's allow that. As we go down any column in this inverse
matrix, the elements must involve increasing powers of the $\xi_i$, for dimensional reasons.
We are led, therefore, to consider the completions
\[
 \Mackey P^*[e_i^{-1}]^{\wedge}_{\xi_i}
\]
and
\[
 \Dirsum_{k=0}^{p-1}\Mackey H_G^\bullet(S^0)[e_i^{-1}]
   [\sigma_k,\zeta_k^\alpha \mid \alpha_k = 0]^{\wedge}_{\xi_i},
\]
allowing infinite sums involving increasing powers of the $\xi_i$, $i\neq 0$.
What makes this feasible is that every element of $\Mackey H_G^\bullet(\tE G)$
is annihilated by sufficiently high powers of the $\xi_i$, with the result that we have
isomorphisms
\[
 \Mackey P^*[e_i^{-1}]\tensor_{\Mackey H_G^\bullet(S^0)[e_i^{-1}]} \Mackey H_G^\bullet(\tE G)
 \iso \Mackey P^*[e_i^{-1}]^{\wedge}_{\xi_i}\tensor_{\Mackey H_G^\bullet(S^0)[e_i^{-1}]} \Mackey H_G^\bullet(\tE G)
\]
and
\begin{multline*}
 \Big(\Dirsum_{k=0}^{p-1}\Mackey H_G^\bullet(S^0)[e_i^{-1}][\sigma_k,\zeta_k^\alpha \mid \alpha_k = 0]\Big)
  \tensor_{\Mackey H_G^\bullet(S^0)[e_i^{-1}]} \Mackey H_G^\bullet(\tE G) \\
  \iso
 \Big(\Dirsum_{k=0}^{p-1}\Mackey H_G^\bullet(S^0)[e_i^{-1}][\sigma_k,\zeta_k^\alpha \mid \alpha_k = 0]^{\wedge}_{\xi_i}\Big)
  \tensor_{\Mackey H_G^\bullet(S^0)[e_i^{-1}]} \Mackey H_G^\bullet(\tE G).
\end{multline*}
Using the inverse matrix of $\bar f_{\tE G}$ therefore defines a map
\[
 \Dirsum_{k=0}^{p-1}\Mackey H_G^\bullet(S^0)[e_i^{-1}]
   [\sigma_k,\zeta_k^\alpha \mid \alpha_k = 0]
  \to \Mackey P^*[e_i^{-1}]^{\wedge}_{\xi_i}
\]
which, on tensoring with $\Mackey H_G^\bullet(\tE G)$, gives a left inverse to $f_{\tE G}$,
as the composite takes each basis element of $\Mackey P^*[e_i^{-1}]$ to itself.
Similarly, considering the completion $(\bar f_{\tE G})^{\wedge}_{\xi_i}$ and tensoring
with $\Mackey H_G^\bullet(\tE G)$, we see that it also gives a right inverse to $f_{\tE G}$.
Thus, $f_{\tE G}$ is an isomorphism as claimed.
\end{proof}

\begin{theorem}\label{thm:oddadditivestructure}
$\Mackey H_G^*(B_+)$ is a free $\Mackey H_G^\bullet(S^0)$-module.
The set of all admissible monomials, as defined in Definition~\ref{def:admissiblemonomials},
is a basis.
\end{theorem}

\begin{proof}
This follows from Propositions~\ref{prop:fEG+} and~\ref{prop:ftEG}, and the map of
long exact sequences
\[
 \xymatrix@C-1.5em{
 \cdots \ar[r]
  & \Mackey P^*\tensorS \Mackey H_G^\bullet(\tE G) \ar[r] \ar[d]_{f_{\tE G}}
  & \Mackey P^* \ar[r] \ar[d]_f
  & \Mackey P^*\tensorS \Mackey H_G^\bullet(EG_+) \ar[r] \ar[d]_{f_{EG_+}}
  & \cdots \\
 \cdots \ar[r]
  & \Mackey H_G^*(B_+\smsh \tE G) \ar[r]
  & \Mackey H_G^*(B_+) \ar[r]
  & \Mackey H_G^*(B_+\smsh EG_+) \ar[r]
  & \cdots
 }
\]
\end{proof}

\begin{corollary}\label{cor:oddmultgenerators}
$\Mackey H_G^*(B_+)$ is generated as an algebra over $\Mackey H_G^\bullet(S^0)$ by
the elements
\begin{align*}
 \chi^i c \quad& 0 \leq i \leq p-1 \\
 \bar\xi^\alpha \quad&  \alpha\in RO_+(\Pi_G B) \qquad\text{and} \\
 \bar\lambda^{\beta,b} \quad& \beta\in RO_0(\Pi_G B).
\end{align*}
\qed
\end{corollary}

\begin{remark}\label{rem:generators}
Note that, rather than taking all the elements $\bar\xi^\alpha$, it suffices to take
just the $p$ elements $\xi_{i,1}$, $0\leq i \leq p-1$.
Similarly, rather than take all $\bar\lambda^{\beta,b}$, we could restrict to just those
in which $\beta_k = 0$ for all but one $k$, and $b_k = 1$ except for that same index.
\end{remark}

In the following section we shall investigate the relations among these generators.

\section{The multiplicative structure of the cohomology of $B_GU(1)$}

In Corollary~\ref{cor:oddmultgenerators} we gave a set of multiplicative generators for
$\Mackey H_G^*(B_+)$ (which can be reduced further as in Remark~\ref{rem:generators}). 
In this section we describe a set of relations that
generate all the relations among these generators. We make no claim that
this set of relations is minimal.

We shall need some notation and calculations.

\begin{definition}
For $1\leq i \leq p-1$, let
\[
 s(i) = \begin{cases}
             1 & \text{if $i \leq p/2$} \\
             -1 & \text{if $i > p/2$.}
        \end{cases}
\]
We extend the definition to all $i\not\equiv 0 \pmod p$ by setting
$s(i) = s(j)$ where $j \equiv i \pmod p$ and $1\leq j \leq p-1$.
\end{definition}

\begin{lemma}\label{lem:elambda}
Let $1 \leq i, j \leq p-1$.
\begin{enumerate}
\item Let $a\in \nu(\MM_j-\MM_i)^{-1}$. Then $s(i)s(j)a \equiv ij^{-1} \pmod p$ and
\[
 e_i\lambda^{\MM_j-\MM_i,a} = s(i)s(j)ae_j
\]
in $\Mackey H_G^\bullet(S^0)$.
\item Let $b$ be such that $b\equiv ij^{-1} \pmod p$. Then
$s(i)s(j)b \in \nu(\MM_j-\MM_i)^{-1}$ and
\[
 be_j = e_i\lambda^{\MM_j-\MM_i,s(i)s(j)b}
\]
in $\Mackey H_G^\bullet(S^0)$.
\end{enumerate}
\end{lemma}

\begin{proof}
If $p=2$ the result is trivial.
For $p$ odd,
it amounts to the observation that the congruence class of $s(i)i$, modulo $p$,
contains an element between $1$ and $(p-1)/2$.
\end{proof}

In the statement of the theorem below we will use the following elements, defined in terms of
our multiplicative generators. We use a notation introduced in
Definition~\ref{def:ekappa}.

\begin{definition}
Suppose that $\alpha\in RO(G)$ with $\alpha^G = 0$, $\beta\in RO_0(\Pi_G B)$, and $0\leq k\leq p-1$.
Then let
\[
 e^\alpha\bar\kappa_k^\beta
  = e^\alpha\Big(\prod_{i=1}^{p-1} e_i\Big)^{-1}\kappa^{\beta_k}
   \Big(\prod_{i\neq k} \chi^i c\cdot \xi_{i,i-k}\Big) \bar\lambda^{\beta-\beta_k,b}
   \in \Mackey H_G^{\alpha+\beta}(B_+)
\]
where $b$ is an appropriate integer vector with $b_k = 1$.
(The result is independent of the other entries of $b$, as will be shown
in Theorem~\ref{thm:multstructure}.)
\end{definition}

The point of this expression is that it defines an element having
the following image under $\eta$.

\begin{lemma}
The image $\eta(e^\alpha\bar\kappa^\beta_k)\in \Mackey H_G^*(B^G_+)$ is given by
\[
 \eta(e^\alpha\bar\kappa^\beta_k)_\ell =
  \begin{cases}
   e^\alpha\kappa^{\beta_k}\zeta_k^{\beta-\beta_k} & \text{if $\ell = k$} \\
   0 & \text{if $\ell\neq k$.}
  \end{cases}
\]
\end{lemma}

\begin{proof}
If $\ell\neq k$, then 
\[
 \eta(\xi_{\ell,\ell-k})_\ell = \xi_{\ell-k}\zeta_\ell^{\Omega_{\ell,\ell-k}-\MM_{\ell-k}+2}.
\]
Because $\kappa\xi_{\ell-k} = 0$, the presence of this factor
makes $\eta(e^\alpha\bar\kappa^\beta_k)_\ell = 0$.

So now consider $\eta(e^\alpha\bar\kappa^\beta_k)_k$. We calculate
\begin{align*}
 \Big(\prod_{i\neq k} \chi^i c\cdot \xi_{i,i-k}\Big)_k
  &= \prod_{i\neq k} (e_{k-i} + \xi_{k-i}\sigma_k)\zeta_k^{\chi^i\omega_G^* - \MM_{k-i} + \Omega_{i,i-k}} \\
  &= \Big(\prod_{i=1}^{p-1} e_i\Big) + \text{terms with factors $\xi_j$},
\end{align*}
(it is part of this calculation that
$\sum_{i\neq k}(\chi^i\omega_G^* - \MM_{k-i} + \Omega_{i,i-k}) = 0$)
and the latter terms vanish when we multiply by $\kappa$, so we get
\begin{align*}
 \eta\Big(e^\alpha\Big(\prod_{i=1}^{p-1} e_i\Big)^{-1}\kappa^{\beta_k}
   \Big(\prod_{i\neq k} \chi^i c\cdot {}&\xi_{i,i-k}\Big) \bar\lambda^{\beta-\beta_k,b}\Big)_k \\
  &= e^\alpha\Big(\prod_{i=1}^{p-1} e_i\Big)^{-1}\kappa^{\beta_k}\Big(\prod_{i=1}^{p-1} e_i\Big)
     \zeta_k^{\beta-\beta_k} \\
  &= e^\alpha\kappa^{\beta_k}\zeta_k^{\beta-\beta_k}. \qedhere
\end{align*}
\end{proof}

\begin{theorem}\label{thm:multstructure}
$\Mackey H_G^*(B_+)$ is generated as an algebra over $\Mackey H_G^\bullet(S^0)$ by
the elements
\begin{align*}
 \chi^i c &\in \Mackey H_G^{\chi^i\omega_G^*}(B_+) && 0 \leq i \leq p-1 \\
 \bar\xi^\alpha &\in \Mackey H_G^\alpha(B_+) &&  \alpha\in RO_+(\Pi_G B)  \\
 \bar\lambda^{\alpha,a} &\in \Mackey H_G^\alpha(B_+) 
    && \alpha\in RO_0(\Pi_G B),\ a_k\in\nu(\alpha_k)^{-1} \\
\end{align*}
All relations among these generators are consequences of the following relations,
where $\delta_i$ is the integer vector with $(\delta_i)_i = 1$ and all other entries $0$:
\begin{align*}
 \bar\xi^\alpha &= \xi^\alpha && \text{if $\alpha\in RO(G)$} 
\\
 \bar\lambda^{\alpha,a} &= \lambda^{\alpha,a}
   &&\text{if $\alpha\in RO(G)$ and $a$ is constant} 
\\
 \bar\xi^\alpha \bar\xi^\beta &= \bar\xi^{\alpha+\beta} 
\\
 \bar\lambda^{\alpha,a}\bar\lambda^{\beta,b} &= \bar\lambda^{\alpha+\beta,ab}
   &&\text{where $ab$ is taken coordinate-wise} 
\\
 \bar\xi^\alpha\bar\lambda^{\beta,b} &= \bar\xi^\gamma\bar\lambda^{\delta,d}
   &&\parbox[m]{2in}{\raggedright if $\alpha+\beta = \gamma+\delta$ 
           and \linebreak $b_k = d_k$ when $\alpha_k = 0$} 
\\
 e_1^{-m}\kappa \xi_{i,1}\bar\lambda^{\beta,b} &= e_1^{-m}\kappa \xi_{i,1}\bar\lambda^{\beta,b'}
   &&\text{if $b_k = b'_k$ for $k\neq i$}
\\
 \bar\lambda^{\alpha,a+p\delta_i} - \bar\lambda^{\alpha,a} &=
   \bar\kappa_i^\alpha
   && \text{for all $a$} 
\\
 e_1^{-m}\kappa\bar\lambda^{\alpha,a} &= \sum_{k=0}^{p-1} a_k e_1^{-m}\bar\kappa^\alpha_k
\\
 \bar\lambda^{\alpha,a}\cdot e_1^{-m}\bar\kappa_i^\beta &= a_i e_1^{-m}\bar\kappa_i^{\alpha+\beta}
\\
 \chi^i c \cdot \xi_{i,1} 
   &= \mathrlap{\chi^j c \cdot \xi_{j,1}
      \bar\lambda^{\chi^i\omega_G^* + \Omega_{i,1} - \chi^j\omega_G^* - \Omega_{j,1},a}
      - e_{i-j}\bar\lambda^{\chi^i \omega_G^* + \Omega_{i,1} - \MM_{i-j},b}} \\
   & \quad\mathrlap{\text{if $b_k = \begin{cases}
                        1 &\text{$k=j$ and $p$ odd}  \\
                        -1 &\text{$k=j$ and $p=2$}  \\
                        a_i & k=i \\
                        s(i-j)(s(k-j)a_k-s(k-i)) & k \neq i, j
                     \end{cases}$}}
\end{align*}
The following relations follows from the others:
\begin{align*}
 \bar\xi^\alpha \cdot e_1^{-m}\bar\kappa_i^\beta &= 0
   &&\text{if $\alpha_i \neq 0$}
\\
 e^\alpha\bar\kappa^\beta_i &= e^\gamma\bar\kappa^\delta_i
   &&\text{if $\alpha+\beta = \gamma+\delta$} 
\\
 \chi^i c \cdot \xi_{i,1} e_1^{-m}\bar\kappa^\alpha_k
   &= s(k-i)e_1^{-(m-1)}\bar\kappa_k^{\alpha+\chi^i\omega_G^*+\Omega_{i,1}-\MM_1}
   && \text{if $k\neq i$}
\end{align*}
\end{theorem}

\begin{proof} 
{\bf The relations hold.}
We first check that the relations given actually hold.
Each relation listed takes place in a dimension with even fixed points,
so, by Proposition~\ref{prop:etamono}, it suffices to check if the relation holds
after applying $\eta$. Recall that the images of the generators are as follows:
\begin{align*}
 \eta(\chi^i c)_k &= (e_{k-i} + \xi_{k-i}\sigma_k)\zeta_k^{\chi^i\omega_G^*-\MM_{k-i}}
\\
 \eta(\bar\xi^\alpha)_k &= \xi^{\alpha_k}\zeta_k^{\alpha-\alpha_k}
\\
 \eta(\bar\lambda^{\alpha,a})_k &= \lambda^{\alpha_k,a_k}\zeta_k^{\alpha-\alpha_k}.
\end{align*}
The first five relations are clear, following from similar
relations in $\Mackey H_G^\bullet(S^0)$ after applying $\eta$.

The relation
$e_1^{-m}\kappa \xi_{i,1}\bar\lambda^{\beta,b} = e_1^{-m}\kappa \xi_{i,1}\bar\lambda^{\beta,b'}$,
if $b_k = b'_k$ for $k\neq i$, follows from the fact that the $i$th component of each
side is 0 after applying $\eta$, while the remaining components are identical.
One reason we need this relation is that it directly implies that $e^\alpha\bar\kappa^\beta$
does not depend on the choice of $b$ used in its definition,
so $e^\alpha\bar\kappa^\beta$ is well-defined based just on the relations listed
in the theorem.

The relation
$\bar\lambda^{\alpha,a+p\delta_i} - \bar\lambda^{\alpha,a} =
   \bar\kappa_i^\alpha$
follows from the calculation of $\eta(\bar\kappa_i^\alpha)$ in the lemma above,
and the identity
\[
 \lambda^{\alpha,a+p} - \lambda^{\alpha,a} = \kappa^\alpha
\]
in $\Mackey H_G^\bullet(S^0)$,
with the right side not depending on $a$. 

The relation
$e_1^{-m}\kappa\bar\lambda^{\alpha,a} = \sum_{k} a_ke_1^{-m}\bar\kappa^\alpha_k$
follows from the identity $e_1^{-m}\kappa\lambda^{\beta,b} = be_1^{-m}\kappa^{\beta}$
in $\Mackey H_G^\bullet(S^0)$.
The same identity is used to show the relation
$\bar\lambda^{\alpha,a}\cdot e_1^{-m}\bar\kappa_i^\beta = a_i e_1^{-m}\bar\kappa_i^{\alpha+\beta}$.

Finally, we compare $\chi^i c \cdot \xi_{i,1}$ and
$\chi^j c \cdot \xi_{j,1}
      \bar\lambda^{\chi^i\omega_G^* + \Omega_{i,1} - \chi^j\omega_G^* - \Omega_{j,1},a}$,
where we assume that $i\neq j$. 
We have
\begin{align*}
 \eta(\chi^i c \cdot \xi_{i,1})_k
  &= \begin{cases}
      \xi_1\sigma_i\zeta_i^{\chi^i\omega_G^*+\Omega_{i,1}-\MM_1} & k = i \\
      (e_{k-i}+\xi_{k-i}\sigma_k)\zeta_k^{\chi^i\omega_G^*+\Omega_{i,1}-\MM_{k-i}} & k\neq i
     \end{cases}
\\
\intertext{and}
 \chi^j c \cdot \xi_{j,1}
      \bar\lambda&^{\chi^i\omega_G^* + \Omega_{i,1} - \chi^j\omega_G^* - \Omega_{j,1},a}
\\
  &= \begin{cases}
      \xi_1\sigma_j\lambda^{\MM_{j-i}-\MM_1,a_j}\zeta_0^{\chi^i\omega_G^*+\Omega_{i,1}-\MM_{i}} & k = j \\
      (e_{i-j}+\xi_{i-j}\sigma_i)\lambda^{\MM_{1}-\MM_{i-j},a_i}\zeta_i^{\chi^i\omega_G^*+\Omega_{i,1}-\MM_{1}} & k = i
\\
      (e_{k-j}+\xi_{k-j}\sigma_k)\lambda^{\MM_{k-i}-\MM_{k-j},a_k}\zeta_k^{\chi^i\omega_G^*+\Omega_{i,1}-\MM_{k-i}} & k\neq i,j
     \end{cases}
\end{align*}
This gives
\begin{multline*}
 \eta(\chi^j c \cdot \xi_{j,1}
      \bar\lambda^{\chi^i\omega_G^* + \Omega_{i,1} - \chi^j\omega_G^* - \Omega_{j,1},a}
      - \chi^i c \cdot \xi_{i,1})_k
\\
 = \begin{cases}
     -e_{j-i} \zeta_j^{\chi^i\omega_G^*+\Omega_{i,1}-\MM_{i-j}} & k = j \\
     e_{i-j}\lambda^{\MM_1-\MM_{i-j},a_{i}} \zeta_i^{\chi^i\omega_G^*+\Omega_{i,1}-\MM_{1}} & k = i \\
     (e_{k-j}\lambda^{\MM_{k-i}-\MM_{k-j},a_{k}} - e_{k-i})\zeta_k^{\chi^i\omega_G^*+\Omega_{i,1}-\MM_{k-i}} & k\neq i,j.
   \end{cases}
\end{multline*}
By Lemma~\ref{lem:elambda}, in the case $k\neq i,j$,
\begin{align*}
 e_{k-j}\lambda^{\MM_{k-i}-\MM_{k-j},a_{k}} - e_{k-i}
  &= s(k-i)s(k-j)a_{k}e_{k-i} - e_{k-i} \\
  &= (s(k-i)s(k-j)a_{k} - 1) e_{k-i}
\end{align*}
with $s(k-i)s(k-j)a_{k} \equiv (k-j)(k-i)^{-1} \pmod p$. So
\begin{align*}
 s(k-i)s(k-j)a_{k} - 1 
  &\equiv (k-j)(k-i)^{-1} - 1 \\
  &\equiv (k-j)(k-i)^{-1} - (k-i)(k-i)^{-1} \\
  &\equiv (i-j)(k-i)^{-1} \pmod p.
\end{align*}
Hence, applying the lemma again,
\begin{align*}
  (s(k-i)s(k-j)a_{k} - 1) e_{k-i}
  &= e_{i-j}\lambda^{\MM_{k-i}-\MM_{i-j},s(k-i)s(i-j)(s(k-i)s(k-j)a_k-1)} \\
  &= e_{i-j}\lambda^{\MM_{k-i}-\MM_{i-j},s(i-j)(s(k-j)a_k-s(k-i))}.
\end{align*}
Thus, we can write
\begin{multline*}
 \eta(\chi^j c \cdot \xi_{j,1}
      \bar\lambda^{\chi^i\omega_G^* + \Omega_{i,1} - \chi^j\omega_G^* - \Omega_{j,1},a}
      - \chi^i c \cdot \xi_{i,1})_k
\\
 = \begin{cases}
     -e_{j-i} \zeta_j^{\chi^i\omega_G^*+\Omega_{i,1}-\MM_{i-j}} & k = j \\
     e_{i-j}\lambda^{\MM_1-\MM_{i-j},a_{i}} \zeta_i^{\chi^i\omega_G^*+\Omega_{i,1}-\MM_{1}} & k = i \\
     e_{i-j}\lambda^{\MM_{k-i}-\MM_{i-j},s(i-j)(s(k-j)a_k-s(k-i))}\zeta_k^{\chi^i\omega_G^*+\Omega_{i,1}-\MM_{k-i}} & k\neq i,j.
   \end{cases}
\end{multline*}
Noting that $-e_{j-i} = e_{i-j}$ if $p$ is odd, but $-e_{j-i} = -e_{i-j}$ if $p=2$,
it follows that
\[
 \chi^j c \cdot \xi_{j,1}
      \bar\lambda^{\chi^i\omega_G^* + \Omega_{i,1} - \chi^j\omega_G^* - \Omega_{j,1},a}
      - \chi^i c \cdot \xi_{i,1}
  = e_{i-j}\bar\lambda^{\chi^i\omega_G^*+\Omega_{i,1}-\MM_{i-j},b}
\]
with $b$ as in the statement of the theorem.

We now show that the other relations given in the statement of the theorem follow from
the ones we have just shown. (We could verify them by applying $\eta$, but the point
is not just that they're true, but that they follow from the other relations.)

Consider $\xi_{i,j}\cdot e_1^{-m}\bar\kappa_i^\beta$. If we substitute the definition
of $e_1^{-m}\kappa_i^\beta$, we will have as a factor in the result
\[
 \xi_{i,j}\prod_{k\neq i}\xi_{k,k-i}
  = \xi_1\bar\lambda^{\alpha,a}
\]
where $\alpha = \Omega_{i,j} + \sum_{k\neq i}\Omega_{k,k-i} - (\MM_1-2)$, and
$a$ is any suitable vector. 
(This is a special case of the identity 
$\bar\xi^\alpha\bar\lambda^{\beta,b} = \bar\xi^\gamma\bar\lambda^{\delta,d}$
with the conditions as in the statement of the theorem.)
Also present is a factor of $\kappa$, and $\kappa\xi_1 = 0$,
hence $\xi_{i,j}\cdot e_1^{-m}\bar\kappa_i^\beta = 0$.
From this follows the general result that 
$\bar\xi^\alpha\cdot e_1^{-m}\bar\kappa_i^\beta = 0$ if $\alpha_i \neq 0$.

The relation $e^\alpha\bar\kappa^\beta_i = e^\gamma\bar\kappa^\delta_i$
for suitable $\alpha$, $\beta$, $\gamma$, and $\delta$, follows from the definition
of $e^\alpha\bar\kappa^\beta_i$ and the similar identity
$e^\alpha\kappa^{\beta_i} = e^\gamma\kappa^{\delta_i}$ in $\Mackey H_G^\bullet(S^0)$.

For the final relation, if $k\neq i$, we have
\begin{align*}
 \chi^i c \cdot \xi_{i,1} &e_1^{-m}\bar\kappa_k^\alpha \\
  &= (\chi^k c\cdot \xi_{k,1}\bar\lambda^{\chi^i\omega_G^*+\Omega_{i,1}-\chi^k\omega_G^*-\Omega_{k,1},a}
      - e_{i-k}\bar\lambda^{\chi^i\omega_G^*+\Omega_{i,1}-\MM_{i-k},b}) e_1^{-m}\bar\kappa_k^\alpha \\
  &= -e_{i-k}\bar\lambda^{\chi^i\omega_G^*+\Omega_{i,1}-\MM_{i-k},b} e_1^{-m}\bar\kappa_k^\alpha \\
  &= -b_k e_{i-k}e_1^{-m}\bar\kappa_k^{\alpha+\chi^i\omega_G^*+\Omega_{i,1}-\MM_{i-k}}
\end{align*}
For suitatble $a$ and $b$. Now, if $p=2$, then $b_k = -1$ and $e_{i-k} = e_1$, so the relation follows in this case.
If $p$ is odd, then $b_k = 1$ and $-e_{i-k} = e_{k-i}$. Further, we have
\[
 e_{k-i} = s(k-i)e_1\mu^{\MM_{k-i}-\MM_1,d}, 
\]
for $d\in \nu(\MM_{k-i}-\MM_1)$,
so again we get the result stated.

{\bf Reducing to admissible monomials, almost.}
We now want to show that the relations given so far determine the whole algebra.
By Theorem~\ref{thm:oddadditivestructure}, it suffices to show that these
relations allow us to rewrite any monomial in the generators
as a linear combination of admissible monomials with coefficients in
$\Mackey H_G^\bullet(S^0)$.
So consider a general monomial
\[
 x = \big(\prod_i (\chi^i c)^{m_i}\big) \bar\xi^\alpha\bar\lambda^{\beta,b}.
\]
Using the relations 
$\bar\xi^{\Omega_{i,j}} = \xi_{i,1}\bar\lambda^{\Omega_{i,j}-\Omega_{i,1},j^{-1}\delta_i}$
and regrouping elements, we may rewrite $x$ in the form
\[
 x = \big(\prod_i (\chi^i c)^{m_i}\big) \big(\prod_i (\chi^i c \cdot\xi_{i,1})^{q_i}\big)
    \big(\prod_i \xi_{i,1}^{n_i} \big) \bar\lambda^{\beta,b}
\]
where, for each $i$, at least one of $m_i$ or $n_i$ is $0$ (so each $q_i$ is as large as possible).
This property will be maintained in each step to follow.

There will be an implicit induction on the degree $\sum m_i + \sum q_i + \sum n_i$ in the following.
For example, if, for each $i$, $q_i + n_i > 0$, then $x$ contains a factor of
$\xi_1 = \prod_i \xi_{i,1}$, so can be written as $\xi_1$ times a monomial of lower degree,
which, by induction, we can assume can be written in terms of admissible monomials.

Suppose that $\sum q_i$ is at least as large as the number of $i$ for which $n_i = 0$.
Using the relation
\[
 \chi^i c \cdot \xi_{i,1} 
   = \chi^j c \cdot \xi_{j,1}
      \bar\lambda^{\chi^i\omega_G^* + \Omega_{i,1} - \chi^j\omega_G^* - \Omega_{j,1},a}
      - e_{i-j}\bar\lambda^{\chi^i \omega_G^* + \Omega_{i,1} - \MM_{i-j},b}
\]
(for suitable $a$ and $b$)
we can, at the expense of spinning off terms of lower degree that can be handled by induction, change factors
$\chi^i c \cdot \xi_{i,1}$ to factors $\chi^j c \cdot \xi_{j,1}$
in such a way that, for each $i$, at least one of $q_i$ or $n_i$ is nonzero.
We may then factor out a $\xi_1$ to get to a monomial of lower degree.
Hence, we may assume that $\sum q_i$ is less than the number of terms for which $n_i = 0$.

Let $(k_i)$ be an ordering of the fixed points $|x|^G_k$ from highest to lowest.
Because $m_i > 0$ implies $n_i = 0$, the nonzero $m_i$ must all occur before the nonzero $n_i$
in this order. 
By changing factors $\chi^i c\cdot\xi_{i,1}$ to factors $\chi^j c\cdot\xi_{j,1}$ as above,
and using the assumption that $\sum q_i$ is less than the number of $i$ for
which $n_i = 0$, we can arrange that each $q_i$
is either $0$ or $1$ and that the nonzero $q_i$ occur first in the order, before the zero $q_i$.
Let us rename $q_i$ as $\epsilon_i$, and note that we have now reduced to the case
\[
 x = \big(\prod_i (\chi^i c)^{m_i}\big) \big(\prod_i (\chi^i c \cdot\xi_{i,1})^{\epsilon_i}\big)
    \big(\prod_i \xi_{i,1}^{n_i} \big) \bar\lambda^{\beta,b}
\]
with the $m_i$, $\epsilon_i$, and $n_i$ satisfying the first four
requirements in Definition~\ref{def:admissiblemonomials} for an admissible monomial.

To achieve the sixth requirement in that definition,
that $\beta_{k_I} = 0$, where $I$ is the least index such that $\epsilon_{k_I} = 0$,
use the relations to write
\[
 \bar\lambda^{\beta,b} = \lambda^{\beta_{k_I},b_{k_I}}\bar\lambda^{\beta-\beta_{k_I},a}
   + \sum_j \tfrac{1}{p}(b_j - b_{k_I}a_j)\bar\kappa^\beta_j
\]
for some $a$. 
We can then write
\begin{align*}
 x &= \lambda^{\beta_{k_I},b_{k_I}}\big(\prod_i (\chi^i c)^{m_i}\big) \big(\prod_i (\chi^i c \cdot\xi_{i,1})^{\epsilon_i}\big)
    \big(\prod_i \xi_{i,1}^{n_i} \big) \bar\lambda^{\beta-\beta_{k_I},a} \\
   &\qquad {} + \sum_j \big(\prod_i (\chi^i c)^{m_i}\big) \big(\prod_i (\chi^i c \cdot\xi_{i,1})^{\epsilon_i}\big)
    \big(\prod_i \xi_{i,1}^{n_i} \big) \tfrac{1}{p}(b_j - b_{k_I}a_j)\bar\kappa^\beta_j.
\end{align*}
This exhibits $x$ as the sum of a multiple of a monomial satisfying all but the fifth condition
and a linear combination of terms involving the $\bar\kappa^\beta_i$. We will see below
how to deal with the latter terms.

So let us assume that $x$ now satisfies all but the fifth condition, which requires that
$1 \leq b_i \leq p-1$ for all $i$. Let $b'$ be the vector with $1\leq b'_i \leq p-1$ and
$b'_i \equiv b_i \pmod p$ for all $i$. Then
\begin{align*}
 x &= \big(\prod_i (\chi^i c)^{m_i}\big) \big(\prod_i (\chi^i c \cdot\xi_{i,1})^{\epsilon_i}\big)
    \big(\prod_i \xi_{i,1}^{n_i} \big) \bar\lambda^{\beta,b'} \\
   &\qquad {} + \sum_j \big(\prod_i (\chi^i c)^{m_i}\big) \big(\prod_i (\chi^i c \cdot\xi_{i,1})^{\epsilon_i}\big)
    \big(\prod_i \xi_{i,1}^{n_i} \big) \tfrac{1}{p}(b_j - b'_j)\bar\kappa^\beta_j 
\end{align*}
displays $x$ as the sum of an admissible monomial and a linear combination of terms of the form
\[
 \big(\prod_i (\chi^i c)^{m_i}\big) \big(\prod_i (\chi^i c \cdot\xi_{i,1})^{\epsilon_i}\big)
    \big(\prod_i \xi_{i,1}^{n_i} \big) \bar\kappa^\beta_j.
\]

{\bf Terms involving $\bar\kappa_j^\alpha$.}
To finish the last two steps, it now remains to show that terms of the form displayed just above can be written as linear combinations of admissible monomials.
It is tempting to substitute the definition of $\bar\kappa_j^\beta$, but that will not produce
an admissible monomial and will just take us back to the start of this whole argument.
Instead, we proceed as follows.

Because of the relation $\xi_{j,1}\bar\kappa^\beta_j = 0$, we may assume that $\epsilon_j = 0 = n_j$.
Using the relation
$\chi^{i} c \cdot\xi_{i,1}\bar\kappa^\beta_{j} 
 = e_1\bar\kappa^{\beta+\chi^i\omega_G^*+\Omega_{i,1}-M_1}_{j}$
when $i\neq j$, we may further reduce to the case where all $\epsilon_i = 0$.
So it remains to show that all elements of the form
\[
 \big(\prod_i (\chi^i c)^{m_i}\big) 
    \big(\prod_i \xi_{i,1}^{n_i} \big)
    \bar\kappa^{\alpha}_{j},
\]
where $n_j = 0$ and, for each $i$, at least one of $m_i$ or $n_i$ is zero, 
can be written as linear combinations of admissible monomials.

Consider then a fixed $\alpha\in RO_0(\Pi_G B)$ and a fixed collection of nonnegative integers
$m_i$ and $n_i$ such that, for each $i$, at least one of $m_i$ or $n_i$ is $0$.
Let $(k_i)$ be the ordering of the integers $m_i-n_i$ from highest to lowest.
Let $N$ be the number of $i$ such that $n_i = 0$.
To ease notation, write
\[
 \beta^i = \chi^{k_i}\omega_G^* + \Omega_{k_i,1} - \MM_1 \in RO_0(\Pi_G B).
\]
For $0\leq j \leq N-1$, write $a^{j} = (a^{j}_k)_k$ where
$1\leq a^{j}_k \leq p-1$ and
\[
 a^{j}_k \in \nu(\alpha_k - \alpha_{k_j} - \sum_{i=0}^{j-1}(\beta^{i}_k-\beta^{i}_{k_j})).
\]
Note that $a^j_{k_j} = 1$.
For $0\leq j \leq N-1$, consider
\begin{multline*}
 e_1^{-j}\kappa^{\alpha_{k_j} - \sum_{i=0}^{j-1}\beta^{i}_{k_j}} \big(\prod_i (\chi^i c)^{m_i}\big) 
    \big(\prod_{i=0}^{j-1} \chi^{k_i} c \cdot\xi_{k_i,1}\big)
    \big(\prod_i \xi_{i,1}^{n_i} \big)
    \bar\lambda^{\alpha-\alpha_{k_{j}}-\sum_{i=0}^{j-1}(\beta^{i}-\beta^{i}_{k_j}), a^j} \\
  = \sum_{\ell=j}^{M-1} \big(\prod_i (\chi^i c)^{m_i}\big) 
    \big(\prod_{i=0}^{j-1} \chi^{k_i} c \cdot\xi_{k_i,1}\big)
    \big(\prod_i \xi_{i,1}^{n_i} \big)
    a^j_{k_\ell}e_1^{-j}
      \bar\kappa^{\alpha-\sum_{i=0}^{j-1}\beta^{i}_{k_\ell}} \\
  = \sum_{\ell=j}^{M-1} \pm a^j_{k_\ell}\big(\prod_i (\chi^i c)^{m_i}\big) 
    \big(\prod_i \xi_{i,1}^{n_i} \big)
    \bar\kappa^{\alpha}_{k_\ell}
\end{multline*}
For the first equality we use the relations
$e_1^{-j}\kappa\lambda^{\alpha,a} = \sum_k a_k e_1^{-j}\bar\kappa^\alpha_k$,
$\xi_{i,1}e_1^{-j}\bar\kappa^\alpha_i = 0$, and 
$\mu^{\alpha,a}e_1^{-j}\bar\kappa^\beta = e_1^{-j}\bar\kappa^{\alpha+\beta}$.
For the second equality we use the relation
$\chi^{k_i} c \cdot\xi_{k_i,1}e_1^{-m}\bar\kappa^\alpha_{k_\ell} 
 = \pm e_1^{-(m-1)}\bar\kappa^{\alpha+\beta^i}_{k_\ell}$
when $i\neq \ell$.
In the display above, on the left is a multiple of an admissible monomial.
On the right is a linear combination of elements of the form
$\big(\prod_i (\chi^i c)^{m_i}\big) 
    \big(\prod_i \xi_{i,1}^{n_i} \big)
    \bar\kappa^{\alpha}_{k_\ell}$,
for the $N$ indices $k_\ell$ for which $n_{k_\ell} = 0$.
Letting $j$ vary from $0$ to $N-1$, we get a system of $N$ linear equations
in the $N$ elements 
$\big(\prod_i (\chi^i c)^{m_i}\big) 
    \big(\prod_i \xi_{i,1}^{n_i} \big)
    \bar\kappa^{\alpha}_{k_\ell}$.
The matrix of this system is
lower triangular, with diagonal entries equal to $\pm 1$ because we have chosen
$a^j_{k_j} = 1$ for each $j$. Therefore, we can solve this system. 
The result is to write every element of the
form
$\big(\prod_i (\chi^i c)^{m_i}\big) 
    \big(\prod_i \xi_{i,1}^{n_i} \big)
    \bar\kappa^{\alpha}_{j}$,
$n_j = 0$, as a linear combination of admissible monomials, as required.
\end{proof}

The statement for the case $p=2$ can be simplified greatly, so we state it
here explicitly.

\begin{corollary}\label{cor:evenResult}
If $p=2$, then $\Mackey H_G^*(B_+)$ is generated as an algebra over $\Mackey H_G^\bullet(S^0)$
by the elements
\begin{align*}
 c &\in \Mackey H_G^{\omega_G^*}(B_+) \\
 \chi c &\in \Mackey H_G^{\chi\omega_G^*}(B_+) \\
 \xi_{0,1} &\in \Mackey H_G^{\Omega_{0,1}}(B_+) = \Mackey H_G^{\chi\omega_G^* - 2}(B_+) \qquad\text{and}\\
 \xi_{1,1} &\in \Mackey H_G^{\Omega_{1,1}}(B_+) = \Mackey H_G^{\omega_G^* - 2}(B_+)
\end{align*}
subject to the two relations
\begin{align*}
 \xi_{0,1}\xi_{1,1} &= \xi_1 \qquad\text{and} \\
 \chi c \cdot \xi_{1,1} &= (1-\kappa)c\cdot \xi_{0,1} + e_1
\end{align*}
\end{corollary}

\begin{proof}
We first note that we do not need to include any elements $\bar\lambda^{\alpha,a}$
as generators, largely because $RO_0(\Pi_G B) = 0$. Explicitly, from the definition we have
\begin{align*}
 \bar\kappa_0^0 &= e_1^{-1}\kappa\cdot\chi c\cdot \xi_{1,1} \quad\text{and} \\
 \bar\kappa_1^0 &= e_1^{-1}\kappa\cdot c\cdot \xi_{0,1},
\end{align*}
and then, using $\bar\lambda^{0,(1,1)} = 1$, we can write
\begin{align*}
 \bar\lambda^{0,a} 
  &= 1 + \frac{a_0-1}{2}\bar\kappa_0^0 + \frac{a_1-1}{2}\bar\kappa_1^0 \\
  &= 1 + \frac{a_0-1}{2}e_1^{-1}\kappa\cdot\chi c\cdot \xi_{1,1}
      + \frac{a_1-1}{2}e_1^{-1}\kappa\cdot c\cdot \xi_{0,1},
\end{align*}
exhibiting $\bar\lambda^{0,a}$ in terms of the other generators.

For the relations, we can dispense with those that merely describe how the
$\bar\lambda^{\alpha,a}$ behave.
We do need the relation $\xi_{0,1}\xi_{1,1} = \xi_1$, which suffices to
describe the behavior of the $\bar\xi^\alpha$.

The only other relation required then is the one relating $\chi^i c\cdot\xi_{i,1}$
to $\chi^j c \cdot \xi_{j,1}$, and it suffices to consider the case $i = 1$ and $j=0$.
Taking $a = (-1,-1)$ in the general relation requires $b = (-1,-1)$ and we note that
$\bar\lambda^{0,(-1,-1)} = \lambda^{0,-1} = 1-\kappa$. Using the fact that
$e_1(1-\kappa) = -e_1$, the relation becomes
\[
 \chi c \cdot \xi_{1,1} = (1-\kappa)c\cdot \xi_{0,1} + e_1
\]
as stated. 
Note that, if we multiply by $1-\kappa$ and use the identity $(1-\kappa)^2 = 1$, we
get the equivalent relation
\[
 c\cdot\xi_{0,1} = (1-\kappa)\chi c\cdot\xi_{1,1} + e_1.
\]
We leave to the reader to show that the relations arising from
all other choices of $a$ can now be derived from the relations given so far.
\end{proof}

\section{Comparison to Lewis's calculation in $RO(G)$-grading}

In \cite{Le:projectivespaces}, Gaunce Lewis calculated $\Mackey H_G^\bullet(B_+)$,
the $RO(G)$-graded part
of $\Mackey H_G^*(B_+)$. In this section we compare our calculation to his.

We begin with the case $p=2$. 
Part of Lewis's Theorem~5.1 shows that $\Mackey H_G^\bullet(B_+)$ is generated by
two elements we shall call $\gamma \in \Mackey H_G^{\MM_1}(B_+)$ (Lewis's $c$) and
$\Gamma\in\Mackey H_G^{\MM_1+2}(B_+)$ (Lewis's $C(1)$), with the single relation
\[
 \gamma^2 = \xi_1\Gamma + e_1\gamma.
\]
This results in an additive basis consisting of
\[
 1,\ \gamma,\ \Gamma,\ \Gamma\gamma,\ \Gamma^2,\ \Gamma^2\gamma,\ \ldots
\]
The additive calculation we did gives a basis for 
$\Mackey H_G^\bullet(B_+)$ consisting of
\[
 1,\ c\cdot\xi_{0,1},\ c\cdot\chi c,\ c\cdot\chi c \cdot(c\cdot\xi_{0,1}),\ 
 (c\cdot\chi c)^2, (c\cdot\chi c)^2\cdot(c\cdot\xi_{0,1}),\ \ldots
\]
If we let $\gamma = c\cdot\xi_{0,1}$ and $\Gamma = c\cdot\chi c$, we recover
Lewis's generators. 
Moreover, referring to Corollary~\ref{cor:evenResult}, if we multiply the relation
$c\cdot\xi_{0,1} = (1-\kappa)\chi c\cdot\xi_{1,1} + e_1$ by $c\cdot\xi_{0,1}$,
and apply the equality $(1-\kappa)\xi_{0,1}\xi_{1,1} = (1-\kappa)\xi_1 = \xi_1$,
we recover the relation $\gamma^2 = \xi_1\Gamma + e_1\gamma$ found by Lewis.

One last comment on the case $p=2$: In his Remark~5.3, Lewis
introduces an element $\tilde\gamma = (1-\kappa)\gamma + e_1$ (he writes $\tilde c$) and points
out that $\tilde\gamma$ could be used as a generator in place of $\gamma$.
In fact, $\tilde\gamma = \chi c\cdot\xi_{1,1}$, and his equation relating
$\gamma$ and $\tilde\gamma$ is our basic relation.

When $p$ is odd, Lewis's description of $\Mackey H_G^\bullet(B_+)$ becomes much more
complicated, partly because he tries to choose specific integers in places
analogous to those where
we allow more freedom, for example in our $\bar\lambda^{\alpha,a}$.
However, most of the complications come from the fact that he was working strictly
in $RO(G)$ grading, when there are real simplifications to be had in this case
by grading on $RO(\Pi_G B)$, where the more natural generators and relations live. 
We shall not attempt a complete dictionary relating his results to ours, but only
sketch the comparison. We start by reviewing what our results say about
$\Mackey H_G^\bullet(B_+)$.

To ease notation in stating our additive result, let
\[
 \beta^k = \sum_{i=0}^{k-1}(\MM_{k-i} - \chi^i\omega_G^* - \Omega_{i,1}) \in RO_0(\Pi_G B),
\]
and notice that $\beta^k_k = 0$. Let $b^k = (b^k_i)_i$ be such that
$b^k_i \in \nu(\beta^k_i)^{-1}$ with $1\leq b^k_1\leq (p-1)/2$. In particular, $b^k_k = 1$.
The basis of admissible monomials in $RO(G)$ grading then takes the form
\[
 \Big\{ \Big(\prod_{i=0}^{p-1}\chi^i c\Big)^{m}
		\Big(\prod_{i=0}^{k-1} \chi^i c\cdot \xi_{i,1}\Big)
		\bar\lambda^{\beta^k,b^k}
		\Bigm| m\geq 0 \text{ and } 0\leq k\leq p-1 \Big\}.
\]
The grading of each monomial is given by
\[
 \big\lvert \Big(\prod_{i=0}^{p-1}\chi^i c\Big)^{m}
		\Big(\prod_{i=0}^{k-1} \chi^i c\cdot \xi_{i,1}\Big)
		\bar\lambda^{\beta^k,b^k} \big\rvert
 = m\Big(2 + \sum_{i=1}^{p-1} \MM_i\Big) + \sum_{i=1}^k \MM_i.
\]
If we were to plot these gradings as usual, with fixed-point dimension horizontally
and dimension vertically, we would see the ``stair-step'' pattern mentioned by
Lewis at the end of his \S3. (Think of the dictionary order on the pairs $(m,k)$.)

Notice that $\Gamma = \prod_{i=0}^{p-1}\chi^i c$ lies in $RO(G)$ grading and generates
a polynomial subalgebra. This element corresponds to Lewis's $C_0(1)$
in his Theorem~5.5.
The product $\Delta_k = (\prod_{i=0}^{k-1} \chi^i c\cdot \xi_{i,1})\bar\lambda^{\beta^k,b^k}$
corresponds to his $D_k$. 
The elements $\Gamma$ and $\Delta_k$, $1\leq k \leq p-1$ (we have $\Delta_0 = 1$), 
generate $\Mackey H_G^\bullet(B_+)$ as an algebra over $\Mackey H_G^\bullet(S^0)$.
(We say ``corresponds to'' rather than ``equals'' because there are a number of arbitrary
choices to be made in writing down these generators and we have not checked that our
choices exactly match Lewis's.)

At this point we should write down the relations among products of these
multiplicative generators.
This amounts to writing the products $\Delta_k \cdot \Delta_\ell$ in terms of admissible monomials;
the corresponding result is one of the most complicated parts of Lewis's discussion.
(See his Proposition~5.12.)
The proof of our Theorem~\ref{thm:multstructure} amounts to an algorithm
for doing this, using the relations in the statement of that theorem.
To carry out the calculation would be possible but, with no immediate need for the resulting
complicated and unilluminating formulas, we shall stop here.
We simply note that Lewis was obliged to do this calculation, because he was working only in
$RO(G)$ grading, whereas we know that such relations follow from the 
much simpler ones listed in Theorem~\ref{thm:multstructure}.

\section{Cohomology with other coefficient systems}

Knowing the cohomology with $\Mackey A_{G/G}$ coefficients, we can relatively
easily calculate the cohomology with some other coefficient systems.
The case of $\Mackey R\Z$ coefficients appears in the literature
as ``constant $\Z$'' coefficients,
for example, in \cite{Kro:SerreSS}, \cite{Dug:Grassmannians}, and \cite{Green:fourapproaches}.
To calculate this cohomology we use the short exact sequence
\[
 0 \to \conc\Z \xrightarrow{\kappa} \Mackey A_{G/G} \to \Mackey R\Z \to 0,
\]
where the first map takes $1$ to $\kappa\in A(G)$.
So we begin by looking at the cohomology with $\conc\Z$ coefficients.
We refer to Proposition~\ref{prop:concZcohomologypoint} for the calculation
of $\Mackey H_G^\bullet(S^0;\conc\Z)$.

\begin{proposition}\label{prop:BGU1ConcZ}
\[
 \Mackey H_G^*(B_+;\conc\Z) \iso
  \Dirsum_{k=0}^{p-1} \Mackey H_G^\bullet(S^0;\Mackey A_{G/G})[\sigma_k, \zeta_k^\alpha \mid \alpha_k = 0]
   \tensor_{\Mackey H_G^\bullet(S^0;\Mackey A_{G/G})}\Mackey H_G^\bullet(S^0;\conc\Z).
\]
The map $\Mackey H_G^*(B_+;\conc\Z) \to \Mackey H_G^*(B_+;\Mackey A_{G/G})$ is injective, with
its image being the ideal 
$\langle e^\alpha\kappa^\beta_k\rangle \subset \Mackey H_G^*(B_+;\Mackey A_{G/G})$.
\end{proposition}

\begin{proof}
It follows from Proposition~\ref{prop:concCcohomology} that
$\Mackey H_G^*(B_+;\conc\Z) \iso \Mackey H_G^*(B^G_+;\conc\Z)$.
The computation of $\Mackey H_G^*(B_+;\conc\Z)$ then follows much as in
Proposition~\ref{prop:calcFixedPoints}.
Moreover, Proposition~\ref{prop:concZcohomologypoint} implies that the
map 
\begin{multline*}
 \Mackey H_G^*(B_+;\conc\Z)
 \iso \Mackey H_G^*(B_+\smsh \tE G;\conc\Z) \to \Mackey H_G^*(B_+\smsh\tE G;\Mackey A_{G/G}) \\
 \iso \Dirsum_{k=0}^{p-1}\Mackey H_G^\bullet(S^0;\Mackey A_{G/G})[\sigma_k, \zeta_k^\alpha \mid
             \alpha_k = 0]
           \tensor_{\Mackey H_G^\bullet(S^0;\Mackey A_{G/G})}\Mackey H_G^\bullet(\tE G;\Mackey A_{G/G})
\end{multline*}
is the inclusion of the submodule of elements with coefficients
from $\Mackey H_G^\bullet(\tE G;\Mackey A_{G/G})$
in gradings $\alpha$ with $\alpha^G = 0$.
From our computations, this submodule injects into
$\Mackey H_G^*(B_+;\Mackey A_{G/G})$ with image the ideal
$\langle e^\alpha\kappa^\beta_k\rangle$, as claimed.
(This is clearest by looking at the characterization of $e^\alpha\kappa^\beta_k$ by
its image $\eta(e^\alpha\kappa_k^\beta)$.)
\end{proof}

We now turn to the cohomology with $\Mackey R\Z$ coefficients,
referring to Theorems~\ref{thm:pointEvenRZcoeffs} and~\ref{thm:pointOddRZcoeffs}
for the structure of $\Mackey H_G^\bullet(S^0;\Mackey R\Z)$.

\begin{theorem}\label{thm:BGU1cohomRZ}
$\Mackey H_G^*(B_+;\Mackey R\Z)$ is generated as an algebra over $\Mackey H_G^\bullet(S^0;\Mackey R\Z)$
by elements
\begin{align*}
 \chi^i c &\in \Mackey H_G^{\chi^i\omega_G^*}(B_+;\Mackey R\Z) && 0 \leq i \leq p-1 \\
 \bar\xi^\alpha &\in \Mackey H_G^\alpha(B_+\Mackey R\Z) &&  \alpha\in RO_+(\Pi_G B)  \\
 \bar\lambda^{\alpha} &\in \Mackey H_G^\alpha(B_+\Mackey R\Z) && \alpha\in RO_0(\Pi_G B) \\
\end{align*}
All relations among these generators are consequences of the following relations:
\begin{align*}
 \bar\xi^\alpha &= \xi^\alpha && \text{if $\alpha\in RO(G)$} 
\\
 \bar\lambda^{\alpha} &= \lambda^{\alpha}
   &&\text{if $\alpha\in RO(G)$} 
\\
 \bar\xi^\alpha \bar\xi^\beta &= \bar\xi^{\alpha+\beta} 
\\
 \bar\lambda^{\alpha}\bar\lambda^{\beta} &= \bar\lambda^{\alpha+\beta}
\\
 \bar\xi^\alpha\bar\lambda^{\beta} &= \bar\xi^\gamma\bar\lambda^{\delta}
   &&\text{if $\alpha+\beta = \gamma+\delta$} 
\\
 \chi^i c \cdot \xi_{i,1} 
   &= \mathrlap{\chi^j c \cdot \xi_{j,1}
      \bar\lambda^{\chi^i\omega_G^* + \Omega_{i,1} - \chi^j\omega_G^* - \Omega_{j,1}}
      - e_{i-j}\bar\lambda^{\chi^i \omega_G^* + \Omega_{i,1} - \MM_{i-j}}}
   && \vrule width 1.5in height 0pt depth 0pt
\end{align*}
$\Mackey H_G^*(B_+;\Mackey R\Z)$ is a free $\Mackey H_G^\bullet(S^0;\Mackey R\Z)$-module,
with a basis given by the admissible monomials, which here means ones of the form
\[
 \big(\prod_i (\chi^i c)^{m_i} \big)
 \big(\prod_i (\chi^i c \cdot \xi_{i,1})^{\epsilon_i}\big)
 \big(\prod_i \xi_{i,1}^{n_i}\big)
 \bar\lambda^{\beta}
\]
satisfying:
\begin{enumerate}
\item For all $i$, $m_i\geq 0$, $\epsilon_i = 0$ or $1$, and $n_i\geq 0$.
\item If $m_i > 0$ or $\epsilon_i > 0$, then $n_i = 0$.
\item For at least one $i$, $\epsilon_i = 0 = n_i$.
\item Using the ordering of $m_k-n_k$ from highest to lowest,
if $\epsilon_{k_i} = 0$ then $\epsilon_{k_j} = 0$ for all $j > i$.
\item If $I$ is the least index such that $\epsilon_{k_I} = 0$, then
$\beta_{k_I} = 0$.
\end{enumerate}
\end{theorem}

\begin{proof}
Consider the long exact sequence induced by the short exact sequence
\[
 0 \to \conc\Z \xrightarrow{\kappa} \Mackey A_{G/G} \to \Mackey R\Z \to 0.
\]
From the preceding proposition we conclude that 
$\Mackey H_G^*(B_+;\Mackey A_{G/G}) \to \Mackey H_G^*(B_+;\Mackey R\Z)$
is surjective, with kernel the ideal $\langle e^\alpha\kappa^\beta_k\rangle$.
As a result, $\bar\lambda^{\alpha,a}$ is identified with
$\bar\lambda^{\alpha,a'}$ for all $a$ and $a'$. We call the common image of
these elements $\bar\lambda^\alpha$. The generators and relations in the statement
of the theorem are then those left over from Theorem~\ref{thm:multstructure}
after eliminating any that 
become trivial on setting the $e^\alpha\kappa$ and $e^\alpha\bar\kappa^\beta_k$
equal to 0.

Moreover, because $\Mackey H_G^*(B_+;\Mackey A_{G/G})$ is a free
module over $\Mackey H_G^\bullet(S^0;\Mackey A_{G/G})$,
with basis given by the admissible monomials,
the preceding proposition implies that $\Mackey H_G^*(B_+;\Mackey R\Z)$ is a
free module over $\Mackey H_G^\bullet(S^0;\Mackey R\Z)$ on the image
of that basis, which is as given in the statement of the theorem.
\end{proof}

Notice that the elements $\bar\lambda^{\alpha}$ are invertible now.

For $p=2$, we can state the result very simply:
$\Mackey H_G^*(B_+; \Mackey R\Z)$ is generated as an algebra over $\Mackey H_G^\bullet(S^0; \Mackey R\Z)$
by the elements
\begin{align*}
 c &\in \Mackey H_G^{\omega_G^*}(B_+; \Mackey R\Z) \\
 \chi c &\in \Mackey H_G^{\chi\omega_G^*}(B_+; \Mackey R\Z) \\
 \xi_{0,1} &\in \Mackey H_G^{\Omega_{0,1}}(B_+; \Mackey R\Z) = \Mackey H_G^{\chi\omega_G^* - 2}(B_+; \Mackey R\Z) \qquad\text{and}\\
 \xi_{1,1} &\in \Mackey H_G^{\Omega_{1,1}}(B_+; \Mackey R\Z) = \Mackey H_G^{\omega_G^* - 2}(B_+; \Mackey R\Z)
\end{align*}
subject to the two relations
\begin{align*}
 \xi_{0,1}\xi_{1,1} &= \xi_1 \qquad\text{and} \\
 \chi c \cdot \xi_{1,1} &= c\cdot \xi_{0,1} + e_1
\end{align*}

Corollaries~\ref{cor:manyIsoCoeffs} and~\ref{cor:manyOddIsoCoeffs}
imply calculations of $\Mackey H_G^*(B_+;\Mackey T)$
with $\Mackey T$ equal to $\Mackey L\Z$, $\Mackey R\Z_-$,
or $\Mackey L\Z_-$.
Each of these is best understood as a free module over $\Mackey H_G^*(B_+;\Mackey R\Z)$
on a single generator.

\part{On equivariant ordinary cohomology}\label{part:point}

\section{Grading on $RO(\Pi_G B)$}\label{sec:ROGGrading}

We have referred to cohomology as being graded on $RO(\Pi_G B)$,
but cohomology as it comes to us is a functor on the category of representations of $\Pi_G B$.
Though it is convenient for computations to grade on the group $RO(\Pi_G B)$
of isomorphism classes of representations,
doing so involves a ``decategorification'' which, as always,
leads to sign ambiguities if not done {\em very} carefully.
(See \cite{Le:projectivespaces} and \cite{May:alaska}.)
As this doesn't appear in the literature in the detail it ought, even for $RO(G)$, we outline
an approach here.

As a first step, we look at some small subcategories of $v\V_G$.
We have in mind $G = \Z/p$, but some of what we say may be applicable more widely.

\begin{definition}
\hspace{2em}
\begin{enumerate}
\item
Let $V_0$, $V_1$, \dots, $V_k$ be an enumeration of the irreducible representations of $G$,
with $V_0 = \R$, and let
\[
 \U = \Dirsum_{i=0}^k V_i^\infty,
\]
where $V_i^\infty = V_i\dirsum V_i\dirsum\cdots$ is the direct sum of countably infinitely many copies of $V_i$.

\item
For $0\leq n < \infty$, write $V_i^n \subset V_i^\infty$ for the sum of the 
first $n$ copies of $V_i$, so $V_i^{n+1} = V_i^n \dirsum V_i$
and $V_i^\infty = \Union_n V_i^n$.
In particular, when we write $\R^n$ we mean specifically $V_0^n$.
Call a finite-dimensional $V\subset \U$ an {\em initial segment} of $\U$ if
\[
 V = \Dirsum_{i=0}^k V_k^{n_i}
\]
for integers $0\leq n_i < \infty$.

\item
Let $v\Vinit_G$ be the full subcategory of $v\V_G$ on the objects of the form
$G/H\times (V\ominus W)$ with $V$ and $W$ initial segments of $\U$.

\item
Let $v\Vlow_G$ be the full subcategory of $v\Vinit_G$ on the objects of the form
$G/H\times (V\ominus W)$ with $V\intersect W = 0$.
We say in this case that $V\ominus W$ is {\em in lowest terms.}
\end{enumerate}
\end{definition}

If $\bar\pi\colon v\Vlow_G\to \orb G$ denotes the restriction of
$\pi\colon v\V_G\to \orb G$, then $\bar\pi^{-1}(G/G)$ is a skeleton of
$\pi^{-1}(G/G)$.
When $G = \Z/p$, $v\Vlow_G$ is equivalent to $v\V_G$, but for more general groups
it may be missing some virtual bundles over some orbits $G/H$ for $H$ a proper subgroup.

The intermediate category $v\Vinit_G$ is more convenient for describing direct sums:
If $T$ and $U$ are initial segments in $\U$, let $T\dirsum U$ denote the initial segment
isomorphic to the direct sum. Explicitly, if
$T = \Dirsum V_i^{m_i}$ and $U = \Dirsum V_i^{n_i}$, then
\[
 T\dirsum V = \Dirsum_{i=0}^k V_i^{m_i+n_i}.
\]
To make this natural in $T$ and $U$, we take
the inclusion $T\to T\dirsum U$  to be the inclusion of the first $m_i$ summands
for each $i$, and $U\to T\dirsum U$ to be the inclusion of the last $n_i$ summands.
Using this direct sum, we define
\[
 [G/H\times (T\ominus U)] \dirsum [G/H\times (V\ominus W)]
  = G/H\times [(T\dirsum V) \ominus (U\dirsum W)].
\]
This defines a direct sum functor $\dirsum\colon v\Vinit_G \times_{\orb G} v\Vinit_G \to v\Vinit_G$.

We describe how to get $RO(G) = RO(\Pi_G(*))$ from this setup;
we'll consider $RO(\Pi_G B)$ for some other spaces $B$ later.
If $T$ and $U$ are initial segments, write
\[
 \Trep\ominus \Urep\colon \Pi_G(*) = \orb G \to v\Vinit_G
\]
for the functor taking $G/H$ to $G/H\times (T\ominus U)$ and taking a map
$\alpha\colon G/H\to G/K$ to $\alpha\times 1$.
If $X$ is another initial segment, write
\[
 \sigma^X\colon \Trep\ominus\Urep \to (\Trep\dirsum\Xrep)\ominus(\Urep\dirsum\Xrep)
\]
for the ``suspension'' map given by the stable map 
$T\ominus U\to (T\dirsum X)\ominus (U\dirsum X)$ specified by 
the identity maps $T\dirsum X \to T\dirsum X$
and $U\dirsum X\to U\dirsum X$,
using the direct sum on $v\Vinit_G$.

Now, there are several ways to think of an element $\alpha\in RO(G)$.
We can consider $\alpha$ to be a formal sum $\sum_{i=0}^k m_i V_i$, 
an equivalence class of formal differences of representations
of $G$, or an equivalence class of representations of $\orb G$ in $v\Vinit_G$.
It is the latter description that will be most useful here.
To be very precise, we think of $\alpha$ as a particular {\em diagram} of
representations of $\orb G$ in $v\Vinit_G$:

\begin{definition}
\hspace{2em}
\begin{enumerate}
\item
Let $\vRO(G)$ denote the category of representations of $\orb G$ in $v\Vinit_G$
and natural transformations.

\item
For $\alpha\in RO(G)$,
let $\D_\alpha$ be the category whose objects are pairs of initial segments
$(T,U)$ where $\alpha = [T\ominus U] \in RO(G)$,
with a map $(T,U)\to (V,W)$ when there are inclusions $T\subset V$ and $U\subset W$.
When there is such a map, we have $(V,W) = (T\dirsum X, U\dirsum X)$ for a unique
initial segment $X$. 

\item
For $\alpha\in RO(G)$, let $\{\alpha\}\colon \D_\alpha\to \vRO(G)$ be the functor defined by
$\{\alpha\}(T,U) = \Trep\ominus\Urep$ on objects, and defined to take an inclusion
$(T,U)\to (T\dirsum X, U\dirsum X)$ to
$\sigma^X\colon \Trep\ominus\Urep \to (\Trep\dirsum\Xrep)\ominus(\Urep\dirsum\Xrep)$.

\end{enumerate}
\end{definition}

We effectively identify the element $\alpha$ with the diagram $\{\alpha\}$,
which specifies all the formal differences equivalent to $\alpha$ as well as choices of
isomorphisms between them.

All this machinery starts to pay off when we look at the sum in $RO(G)$ and its relation
to the direct sum of virtual representations.
We first need to introduce notation for some units in the Burnside ring.

\begin{definition}
\hspace{2em}
\begin{enumerate}
\item
If $V_i$ is an irreducible representation of $G$, write
\[
 \gamma(V_i^m, V_i^n)\in A(G)
\]
for the stable homotopy class of
\[
 S^{V_i^{m+n}} = S^{V_i^m}\smsh S^{V_i^n} \xrightarrow{\gamma}
  S^{V_i^n}\smsh S^{V_i^m} = S^{V_i^{m+n}},
\]
where the map $\gamma$ interchanges the two factors.
Note that $\gamma(V_i^m, V_i^n)^2 = 1$ for any $m$ and $n$.

\item
If $V = \Dirsum_i V_i^{m_i}$ and $W = \Dirsum_i V_i^{n_i}$ are initial segments, write
\[
 \gamma(V,W) = \prod_i \gamma(V_i^{m_i}, V_i^{n_i}) \in A(G).
\]

\item
If $T\ominus U$ and $V\ominus W$ are virtual representations given by initial segments, let
\[
 \gamma(T\ominus U, V\ominus W) = \gamma(T,V)\gamma(U,V)\gamma(T,W)\gamma(U,W) \in A(G).
\]
(We arguably should have written $\gamma(T,V)\gamma(U,V)^{-1}\gamma(T,W)^{-1}\gamma(U,W)$,
but each of these elements is its own inverse in $A(G)$.)

\item
If $\alpha$ and $\beta$ are elements of $RO(G)$, write $\gamma(\alpha,\beta)\in A(G)$
for $\gamma(T\ominus U, V\ominus W)$ for any choice of initial segments with
$\alpha = [T\ominus U]$ and $\beta = [V\ominus W]$.

\end{enumerate}
\end{definition}

When $V_i$ is irreducible, $\gamma(V_i^m, V_i^n)$ is the same element of $A(G)$
as that given by multiplication by $(-1)^{mn}$ on $S^{V_i}$, as can be seen by a little
linear algebra.
When $G = \Z/2$ and $\LL$ is the nontrivial irreducible representation,
multiplication by $-1$ on $S^{\LL}$ is the element 
\[
 \gamma(\LL,\LL) = 1 - g = \kappa - 1 \in A(\Z/2).
\]
(Notice that $(1-g)^2 = 1$.)
When $G = \Z/p$ for $p$ odd, all the nontrivial irreducible representations are
two-dimensional and multiplication by $-1$ on any of them gives a map homotopic to the identity,
so the only interesting unit that arises in this case is $-1$ and comes from
permutation of trivial representations.

\begin{definition}
\hspace{2em}
\begin{enumerate}
\item
If $\alpha$ and $\beta$ are elements of $RO(G)$, let $\{\alpha\}\dirsum\{\beta\}$
denote the composite
\[
 \D_\alpha\times \D_\beta \xrightarrow{\{\alpha\}\times\{\beta\}}
  \vRO(G)\times\vRO(G) \xrightarrow{\dirsum} \vRO(G).
\]

\item
Let $\gamma\colon \D_\alpha\times \D_\beta \to \D_\beta\times \D_\alpha$
denote the interchange functor.

\item
Let $\dirsum\colon \D_\alpha\times\D_\beta\to \D_{\alpha+\beta}$ be the functor defined by
$(T,U)\dirsum (V,W) = (T\dirsum V, U\dirsum W)$.

\end{enumerate}
\end{definition}

We first compare the two functors $\{\alpha\}\dirsum\{\beta\}$ and $\{\alpha+\beta\}\circ\dirsum$.
Suppose that $T$, $U$, $V$, $W$, $X$ and $Y$ are all initial segments, and consider the
following commutative diagram:
\[
 \xymatrix@C-2.5em{
  (\Trep\ominus \Urep)\dirsum (\Vrep\ominus \Wrep) \ar@{=}[r] 
             \ar[d]_{\sigma^\Xrep\dirsum\sigma^\Yrep}
    & (\Trep\dirsum \Vrep) \ominus (\Urep\dirsum \Wrep) \ar[dd]^{\sigma^{\Xrep\dirsum \Yrep}} \\
  [(\Trep\dirsum \Xrep)\ominus (\Urep\dirsum \Xrep)] \dirsum [(\Vrep\dirsum \Yrep)\ominus (\Wrep\dirsum \Yrep)]
             \ar@{=}[d] \\
  (\Trep\dirsum \Xrep\dirsum \Vrep\dirsum \Yrep) \ominus (\Urep\dirsum \Xrep\dirsum \Wrep\dirsum \Yrep)
             \ar[r]_\gamma
  & (\Trep\dirsum \Vrep\dirsum \Xrep\dirsum \Yrep) \ominus (\Urep\dirsum \Wrep\dirsum \Xrep\dirsum \Yrep)
 }
\]
The composite down the left side is the effect of
$\{\alpha\}\dirsum\{\beta\}$ on
the map $(T,U)\times (V,W) \to (T\dirsum X,U\dirsum X)\times (V\dirsum Y,W\dirsum Y)$.
The composite across the top and down the right side is the effect of
$\{\alpha+\beta\}\circ\dirsum$
on the same map. 
If we view the two objects on the bottom of the diagram as the same object,
we see that
the diagram would not commute without $\gamma$, which we interpret as the sign
$\gamma(X,V\ominus W)$. In order to define a natural transformation
\[
 \mu\colon \{\alpha\}\dirsum\{\beta\} \to \{\alpha+\beta\}\circ\dirsum,
\]
we must therefore take signs
into account, and the most straightforward way to do that is to say that the $(T,U)\times(V,W)$ component
of $\mu$ is the map
\begin{multline*}
 \mu\colon (\{\alpha\}\dirsum\{\beta\})((T,U)\times(V,W))
   = (\Trep\dirsum \Vrep) \ominus (\Urep\dirsum \Wrep) \\
   \xrightarrow{\gamma(T,V\ominus W)} (\Trep\dirsum \Vrep) \ominus (\Urep\dirsum \Wrep) 
   = \{\alpha+\beta\}(T\dirsum V, U\dirsum W).
\end{multline*}
From the diagram above, it is straightforward to check that this does define
a natural transformation.

Now for commutativity relations. We first note that
$\{\alpha+\beta\}\circ\dirsum$ and $\{\beta+\alpha\}\circ\dirsum\circ\gamma$ are exactly
the same functors from $\D_\alpha\times \D_\beta$ to $\vRO(G)$.
(In fact, $\dirsum\circ\gamma = \dirsum$.)
On the other hand, $\{\alpha\}\dirsum\{\beta\}$ and $(\{\beta\}\dirsum\{\alpha\})\circ\gamma$
are not the same. There is, however, a natural transformation
$\gamma\colon \{\alpha\}\dirsum\{\beta\}\to (\{\beta\}\dirsum\{\alpha\})\circ\gamma$
whose $(T,U)\times (V,W)$ component
\[
 \gamma\colon (\Trep\dirsum\Vrep)\ominus(\Urep\dirsum\Wrep)
  \to (\Vrep\dirsum\Trep)\ominus(\Wrep\dirsum\Urep)
\]
interchanges summands in the obvious way.
If we view the source and target as the same object, this map is multiplication by the unit
$\gamma(T,V)\gamma(U,W)$.

We can now compare the composite natural transformations
\[
 \{\alpha\}\dirsum\{\beta\} \xrightarrow{\mu}
  \{\alpha + \beta\}\circ\dirsum = \{\beta+\alpha\}\circ\dirsum\circ\gamma
\]
and
\[
 \{\alpha\}\dirsum\{\beta\} \xrightarrow{\gamma}
  (\{\beta\}\dirsum\{\alpha\})\circ\gamma \xrightarrow{\mu\circ\gamma}
  \{\beta+\alpha\}\circ\dirsum\circ\gamma.
\]
Evaluating at $(T,U)\times(V,W)$ we get the following diagram:
\[
 \xymatrix@C+2em{
  (\Trep\dirsum\Vrep) \ominus (\Urep\dirsum\Wrep)
    \ar[r]^{\gamma(T,V\ominus W)}
    \ar[d]_{\gamma(T,V)\gamma(U,W)}
  & (\Trep\dirsum\Vrep) \ominus (\Urep\dirsum\Wrep) \ar@{=}[d] \\
  (\Vrep\dirsum\Trep) \ominus (\Wrep\dirsum\Urep)
    \ar[r]_{\gamma(T\ominus U,V)}
  & (\Vrep\dirsum\Trep) \ominus (\Wrep\dirsum\Urep)
 }
\]
This diagram does not commute on the nose, due to the signs indicated, but does up to 
their combination, which is
\[
 \gamma(T,V\ominus W)\gamma(T,V)\gamma(U,W)\gamma(T\ominus U,V)
  = \gamma(T\ominus U,V\ominus W) = \gamma(\alpha,\beta).
\]
As we shall see in a moment, this is the source of the familiar sign in the
anticommutativity of the cup product in cohomology.

So, let us return to the ordinary cohomology $\Mackey H_G^*(X;\Mackey T)$ and say exactly what we mean
by $RO(G)$ grading. (We are thinking again of the case where all spaces are considered
as parametrized over a point, $*$.)
Given $\alpha\in RO(G)$, we have $\{\alpha\}\colon \D_\alpha\to \vRO(G)$
and, on applying $\Mackey H_G^* = \Mackey H_G^*(X;\Mackey T)$, a functor
$\Mackey H_G^{\{\alpha\}}\colon \D_\alpha\to \Ab$.
(We are suppressing the space and coefficient system to concentrate on the grading.)
Note that all the maps in this diagram are isomorphisms. We then define
\[
 \Mackey H_G^\alpha = \colim_{\D_\alpha} \Mackey H_G^{\{\alpha\}}
  = \colim_{\D_\alpha} \Mackey H_G^{\Trep\ominus\Urep}.
\]
Thus, we remain uncommitted as to which particular representative of $\alpha$ we shall use,
but have specified isomorphisms between the various possibilities.

The cup product is a natural pairing
\[
 \cup\colon\Mackey H_G^{\Trep\ominus\Urep}\boxprod \Mackey H_G^{\Vrep\ominus\Wrep}
  \to \Mackey H_G^{(\Trep\dirsum\Vrep)\ominus (\Urep\dirsum\Wrep)}.
\]
Precisely, it is natural when we view both sides as functors on $\vRO(G)\times \vRO(G)$,
the right hand side factoring through $\dirsum\colon \vRO(G)\times \vRO(G)\to \vRO(G)$.
As such, composition gives a natural transformation
\[
 \cup\colon \Mackey H_G^{\{\alpha\}}\boxprod \Mackey H_G^{\{\beta\}}
  \to \Mackey H_G^{\{\alpha\}\dirsum\{\beta\}}.
\]
Now, we would prefer the target to be 
$\Mackey H_G^{\{\alpha+\beta\}\circ\dirsum}$,
whose colimit is $\Mackey H_G^{\alpha+\beta}$,
so we need to apply the natural transformation
$\mu$ defined above. This means that we need to introduce signs, replacing $\cup$ with
\[
 \gamma(T,V\ominus W){\cup}\colon
  \Mackey H_G^{\Trep\ominus\Urep}\boxprod \Mackey H_G^{\Vrep\ominus\Wrep}
  \to \Mackey H_G^{(\Trep\dirsum\Vrep)\ominus (\Urep\dirsum\Wrep)}.
\]
With this adjustment of signs we get a natural transformation
\[
 \Mackey H_G^{\{\alpha\}}\boxprod \Mackey H_G^{\{\beta\}}
  \to \Mackey H_G^{\{\alpha+\beta\}\circ\dirsum}
\]
and, on taking colimits,
\[
 \Mackey H_G^\alpha \boxprod \Mackey H_G^\beta \to \Mackey H_G^{\alpha+\beta}.
\]
Finally, the commutativity analysis above leads to the conclusion that the diagram
\[
 \xymatrix@R-1.5em{
   \Mackey H_G^\alpha \boxprod \Mackey H_G^\beta \ar[dr] \ar[dd]_\gamma \\
   & \Mackey H_G^{\alpha+\beta} \\
   \Mackey H_G^\beta \boxprod \Mackey H_G^\alpha \ar[ur]
 }
\]
commutes up to the sign $\gamma(\alpha,\beta)$, i.e., with this understanding of $RO(G)$ grading,
the cup product is anticommutative, meaning that
\[
 y \cup x = \gamma(|x|,|y|)(x\cup y).
\]

We can be more explicit for the groups we are most interested in.
We use the notation introduced in Definition~\ref{def:irrRepresentations}.
In the case $G = \Z/2$, if $\alpha = m_0\R + m_1\LL$ and
$\beta = n_0\R + n_1\LL$, then
\[
 \gamma(\alpha, \beta) = (-1)^{m_0 n_0}(1-g)^{m_1 n_1}.
\]
In the case $G = \Z/p$ with $p$ odd, if
$\alpha = m_0\R + \sum_k m_k\MM_k$ and
$\beta = n_0\R + \sum_k n_k\MM_k$,
then we have simply
\[
 \gamma(\alpha, \beta) = (-1)^{m_0 n_0}.
\]

\begin{remark}[Restriction to $G/e$]\label{rem:restrictionG/e}
It is tempting to think of restriction to the $G/e$ level of our
Mackey-functor valued cohomology as simply restriction from
$G$-cohomology to nonequivariant cohomology, but that is somewhat misleading.
Notice that, in the case of $G = \Z/2$, anticommutativity at the $G/e$ level is not
what we might think: The map $A(G)\to \Z$ takes $(-1)^{m_0 n_0}(1-g)^{m_1 n_1}$ to
$(-1)^{m_0 n_0 + m_1 n_1}$, {\em not} $(-1)^{(m_0+m_1)(n_0+n_1)}$.
So, if we think about the forgetful map 
$\tilde H_G^\alpha\to \tilde H_e^{\alpha} \iso \tilde H^{|\alpha|}$,
the commutation law on the source does not agree with the usual anticommutativity on the target.
We really need to think of $\tilde H_e^*$ as $RO(G)$-graded
in this context, distinguishing the parts of the grading coming from $\R$ from those
coming from $\LL$.
\end{remark}

What about grading on $RO(\Pi_G B)$, with $B = B_GU(1)$?
(Here we restrict to $G = \Z/p$.)
The outline of the approach is similar to the discussion above of $RO(G)$ grading.
Cohomology of spaces over $B$ is most naturally graded on the category
of representations of $\Pi_G B$, meaning functors
$\Pi_GB\to v\V_G$ over $\orb G$.
We may restrict to functors $\Pi_GB\to v\Vinit_G$ taking values in formal
differences of initial segments.
For any $\alpha = (\alpha_0,\dots,\alpha_{p-1}) \in RO(\Pi_GB)$,
we consider the category $\D_\alpha$ with objects $p$-tuples of
pairs $((V_k,W_k))_k$, where the $V_k$ and $W_k$ are initial segments with
$\alpha_k = [V_k\ominus W_k]$ and with $|V_0| = \cdots = |V_{p-1}|$;
the maps in $\D_\alpha$ are the inclusions.
There is then a functor $\{\alpha\}$ from $\D_\alpha$ to representations
$\Pi_GB\to v\Vinit_G$ that takes $((V_k,W_k))_k$ to the representation
specified on $b_k$ by $V_k\ominus W_k$, and specified on $b$ by 
$G\times(V_0\ominus W_0)$,
with the maps induced by $b\to b_k$ given by the obvious identifications.
(In the case $p=2$ and $|\alpha|-|\alpha^G|$ odd, 
$G\times(V_0\ominus W_0)$ will have the nonidentity self-map.)
$\{\alpha\}$ takes maps to the suspension maps given by the inclusions.
We then define $\Mackey H_G^\alpha$ to be the colimit of the diagram $\Mackey H_G^{\{\alpha\}}$.

From here the analysis goes much as for the case of $RO(G)$,
and anticommutativity works perhaps more nicely than expected:
For $\alpha\in RO(\Pi_GB)$, recall that the $|\alpha_k|$ are all equal,
the $\alpha_k^G$ all have the same parity, and so the
$|\alpha_k-\alpha_k^G|$ all have the same parity as well.
Write $|\alpha|$ for the common dimension of the $|\alpha_k|$, write
$|\alpha^G|\in\Z/2$ for the common parity of the $\alpha_k^G$,
and write $|\alpha-\alpha^G|\in\Z/2$ for the common parity of the $|\alpha_k-\alpha_k^G|$.
Then, for $x\in\Mackey H_G^\alpha$ and $y\in\Mackey H_G^\beta$,
\begin{align*}
 yx &= \gamma(\alpha,\beta)xy,
\intertext{where }
  \gamma(\alpha,\beta) &= (-1)^{|\alpha^G|\cdot|\beta^G|}
                         (1-g)^{|\alpha-\alpha^G|\cdot|\beta-\beta^G|}
      && \text{if $p = 2$, and} \\
\gamma(\alpha,\beta) &= (-1)^{|\alpha|\cdot|\beta|} && \text{if $p$ is odd.}
\end{align*}

\section{Preliminary results on the cohomology of a point}\label{sec:pointprelim}

Over the next several sections we will calculate the $RO(G)$-graded cohomology of a point
for $G = \Z/p$, verifying the results stated in \S\ref{sec:pointSummary}.
In that section, we stated $\Mackey H_G^\bullet(S^0)$ first
and then gave the structures of $\Mackey H_G^\bullet(EG_+)$ and
$\Mackey H_G^\bullet(\tE G)$ as modules over the cohomology of a point.
However, we shall calculate $\Mackey H_G^\bullet(EG_+)$ and
$\Mackey H_G^\bullet(\tE G)$ first, as they are actually relatively simple to handle.
We then use those results to derive the cohomology of a point, using
the cofibration sequence $EG_+\to S^0\to \tE G$.
But we begin in this section by identifying
several elements of $\Mackey H_G^\bullet(S^0)$ that come from the
equivariant (stable) homotopy of spheres.

Recall the Euler classes $e_V \in \Mackey H_G^{V}(S^0)$ defined in Definition~\ref{def:EulerClasses}.
It is not meant to be obvious that any of these classes are nonzero.
In fact, if $V^G \neq 0$, then the map $S^0\to S^V$ is equivariantly null homotopic, so $e_V = 0$.
In particular, $e_0 = e_{\MM_0} = 0$.
However, if $V^G = 0$, the map is equivariantly essential, and we shall see that,
at least for the groups $G = \Z/p$, the resulting cohomology class is nonzero.

Note that, for any based $G$-space $X$, the map
\[
 \Mackey H_G^\alpha(X) \iso \Mackey H_G^{\alpha+V}(X\smsh S^V)
  \xrightarrow{1\smsh e_V^*} \Mackey H_G^{\alpha+V}(X)
\]
is the same as by multiplication by $e_V\in \Mackey H_G^{V}(S^0)$.

Now, $e_V$ is invariant with respect to isomorphisms of $V$. That is, given
an isomorphism $f\colon V\to W$, $f_*e_V = e_W$ regardless of which isomorphism is chosen.
This follows from the fact that the suspension isomorphism is invariant, so
that the following diagram commutes:
\[
 \xymatrix@C=.1em{
   & \Mackey H_G^0(S^0) \ar[dl]_\iso \ar[dr]^\iso \\
   \Mackey H_G^V(S^V) \ar[rr] \ar[d] & & \Mackey H_G^W(S^W) \ar[d] \\
   \Mackey H_G^V(S^0) \ar[rr] & & \Mackey H_G^W(S^0).
 }
\]
Here, the map $\Mackey H_G^V(S^V)\to \Mackey H_G^W(S^W)$
is the one induced by $f$ on the grading and
$f^{-1}\colon S^W\to S^V$, while the map
$\Mackey H_G^V(S^0)\to \Mackey H_G^W(S^0)$
is induced by $f$ on the grading.

\begin{remark}
At first glance, the invariance of the Euler classes with respect to isomorphisms
seems strange. For example, when $p=2$, we shall show that $e = e_\LL$
generates $\Mackey H_G^\LL(S^0) \iso \conc\Z$. The invariance then says that the negation
map $f\colon\LL\to\LL$ has to act by the identity on this group, when we might have expected it
to act by negation. Remember, however, that $S^f\colon S^\LL\to S^\LL$ represents
$1-g\in A(G)$, and $g$ acts as 0 on $\conc\Z$, so the negation map
on $\LL$ does act as the identity on $\Mackey H_G^\LL(S^0)$.
\end{remark}

We mentioned in \S\ref{sec:pointSummary} that, 
for $p$ odd, we are also interested in the Euler classes of the
complex representations $\C_k$ for all $1\leq k\leq p-1$ or, equivalently,
the Euler classes of the ``oriented'' representations $\MM_k$.
If $1\leq k\leq (p-1)/2$, this is just the Euler class $e_k$.
For $(p+1)/2 \leq k \leq p-1$, the Euler class lives naturally in
$\tilde H_G^{\MM_k}(S^0)$, if we think of cohomology as a functor on
actual representations of $G$, but because we are grading on $RO(G)$ we need to
identify this group with $\tilde H_G^{\MM_{p-k}}(S^0)$.
We need, therefore, to choose and fix an identification of these groups,
which we do by choosing a stable isomorphism between $M_k$ and $M_{p-k}$.
There are two candidates, one being the (unstable) $G$-isomorphism
$\phi\colon \MM_k\to \MM_{p-k}$.
However, when we look at the cohomology of $B_GU(1)$, and the Euler classes
of the fibers of the canonical line bundle $\omega_G$,
those fibers are all identified nonequivariantly with $\R^2$ by the
isomorphisms mentioned in Definition~\ref{def:irrRepresentations}.
Hence, in this context, it is more natural to use
the stable isomorphism that agrees nonequivariantly with 
those identifications.

\begin{definition}\label{def:chosenidentification}
Let $\psi\colon \MM_k\to \MM_{p-k}$ be the stable isomorphism given by
the pair of maps $\MM_k\dirsum\R\to \MM_{p-k}\dirsum\R$ and $\R\to \R$,
where the former is the sum of the $G$-isomorphism $\phi$ and $-1$, and the latter is the identity.
We henceforth use $\psi_*$ to identify 
$\tilde H_G^{\MM_k}(X) \iso \tilde H_G^{\MM_{p-k}}(X)$
for any $X$.
\end{definition}

Here is the consequence for the Euler classes: The invariance of the
Euler classes given above tells us that $\phi_* e_k = e_{p-k}$,
but $\phi_* = -\psi_*$, so, identifying $e_k$ with $\psi_*(e_k)$, we see that
\[
 e_k = e(\C_k) = e(\MM_k) = -e_{p-k} \in \tilde H_G^{\MM_{p-k}}(S^0)
  \quad\text{for $(p+1)/2 \leq k \leq p-1$.}
\]

For $p$ odd we have the following elements,
which are interesting primarily when $p>3$.

\begin{definition}\label{def:mu}
Let $p$ be odd, let $1 \leq j, k \leq (p-1)/2$, and let $d\in \Z$ be such that
$d \equiv kj^{-1} \pmod p$. Using our chosen nonequivariant 
identifications of $\MM_k$ with $\C$,
so that the generator $t$ of $G$ acts as $e^{2\pi i k/p}$,
let
\[
 \mu_{j,k,d}\colon S^{\MM_j} \to S^{\MM_k}
\]
be the $G$-map induced by 
\[
 z \mapsto 
  \begin{cases}
    z^d & \text{if $d>0$} \\
    \bar z^{|d|} & \text{if $d<0$}
  \end{cases}
\]
for $z\in \MM_j$.
We also write $\mu_{j,k,d}\in \tilde H_G^{\MM_k - \MM_j}(S^0)$ for the image of 1 under
the map
\[
 \tilde H_G^0(S^0) \iso \tilde H_G^{\MM_k}(S^{\MM_k}) \xrightarrow{\mu_{j,k,d}^*}
  \tilde H_G^{\MM_k}(S^{\MM_j}) \iso \tilde H_G^{\MM_k - \MM_j}(S^0).
\]
\end{definition}

\begin{proposition}\label{prop:rhoMu}
In $\Mackey H_G^\bullet(S^0)$, we have
$\rho(\mu_{j,k,d}) = d\iota_k\iota_j^{-1}$.
\end{proposition}

\begin{proof}
This is the observation that,
under our standard nonequivariant identifications of $\MM_k$ with $\R^2$, 
$\mu_{j,k,d}\colon S^2\to S^2$ has nonequivariant degree $d$.
\end{proof}

\begin{proposition}\label{prop:muTimesE}
For $1 \leq j, k \leq (p-1)/2$, we have $\mu_{j,k,d}e_j = e_k$ in $\Mackey H_G^{M_k}(S^0)$.
\end{proposition}

\begin{proof}
This comes down to two observations: That the two inclusions
\[
 e_j\smsh 1, 1\smsh e_j\colon S^{\MM_j}\to S^{\MM_j}\smsh S^{\MM_j}
\]
are homotopic, and that the following diagram commutes:
\[
 \xymatrix@C-1em{
  & S^0 \ar[dl]_{e_j} \ar[dr]^{e_k} \\
  S^{\MM_j} \ar[rr]_{\mu_{j,k,d}} & & S^{\MM_k}
 }
\]
We leave to the reader then to draw the diagram that
verifies the proposition.
\end{proof}

\section{The cohomology of $EG_+$}

In this section we calculate $\Mackey H_G^\bullet(EG_+)$
for $G = \Z/p$.
Unless specified otherwise, we take all coefficients in $\Mackey A_{G/G}$.

\begin{proposition}\label{prop:cohomologyG}
For $\alpha\in RO(G)$,
\[
 \Mackey H_G^\alpha(G/e_+) \iso
  \begin{cases}
   \Mackey A_{G/e} & \text{if $|\alpha| = 0$} \\
   0 & \text{otherwise.}
  \end{cases}
\]
Let $\gamma\colon G/e_+\to G/e_+$ be the $G$-map given by multiplication by
the generator $t$.
\begin{itemize}
\item
If $p$ is odd then $\gamma^*$ is the identity on 
$\Mackey H_G^\alpha(G/e_+)(G/G)$ and multiplication by $t\in\Z[G]$ on
$\Mackey H_G^\alpha(G/e_+)(G/e)$.
\item
If $p = 2$, then $\gamma^*$
acts as $(-1)^{\alpha^G}$ on
$\Mackey H_G^\alpha(G/e_+)(G/G)$
and as $(-1)^{\alpha^G}t$ on
$\Mackey H_G^\alpha(G/e_+)(G/e)$.
\end{itemize}
\end{proposition}

\begin{proof}
From \cite{CW:ordinaryhomology} we have the Wirthm\"uller isomorphism
\[
 \tilde H_G^\alpha(G_+\smsh_e X;\Mackey A_{G/G}) \iso \tilde H^{|\alpha|}(X;\Z)
\]
for any nonequivariant based space $X$. In particular, taking $X = G/G_+$ and $X = G/e_+$
in turn
gives the calculation of the Mackey functors $\Mackey H_G^\alpha(G/e_+)$.

To see the effect of $\gamma$, first notice that the action of $\gamma$
on $\Mackey H_G^0(G/e_+)$ is $1$ at level $G/G$ and
multiplication by $t$ at level $G/e$, from the dimension axiom.
For a general $\alpha$ with $|\alpha| = 0$ (the action of $\gamma$ is obvious otherwise),
write $\alpha = V - W$ where $V$ and $W$ are actual representations of $G$ of the
same nonequivariant dimension, say $n$.
Consider $V = \R^n$ with the action of $G$ given via a homomorphism
$r_V\colon G\to O(n)$, and similarly let $W = \R^n$ with action given
by $r_W\colon G\to O(n)$. Then we have a $G$-homeomorphism
\[
 \eta\colon G/e_+\smsh S^V \to G/e_+\smsh S^W
\]
given by
\[
 \eta(h,v) = (h, r_W(h)r_V(h)^{-1}v)
\]
for $h\in G$ and $v\in V = \R^n$. We can redo the calculation of $\Mackey H_G^\alpha(G/e_+)$
as
\begin{align*}
 \Mackey H_G^{V-W}(G/e_+)
  &\iso \Mackey H_G^V(G/e_+\smsh S^W) \\
  &\xrightarrow{\eta^*} \Mackey H_G^V(G/e_+\smsh S^V) \\
  &\iso \Mackey H_G^0(G/e_+) \\
  &\iso \Mackey A_{G/e},
\end{align*}
where the last isomorphism comes from the dimension axiom. We also have the following
commutative diagram:
\[
 \xymatrix{
  G/e_+\smsh S^V \ar[d]_{\gamma\smsh 1} \ar[r]^\eta
   & G/e_+\smsh S^W \ar[d]^{\gamma\smsh r_W(t)r_V(t)^{-1}} \\
  G/e_+\smsh S^V \ar[r]_\eta & G/e_+\smsh S^W.
 }
\]
It follows that the action of $\gamma$ on $\Mackey H_G^{V-W}(G/e_+)$ is the same as its
action on $\Mackey H_G^0(G/e_+)$ modified by the nonequivariant sign
of $r_W(t)r_V(t)^{-1}\colon S^n\to S^n$. If $p$ is odd, this sign is always $1$.
If $p$ is even, it is $(-1)^a$ where $a$ is the number of copies of $\LL$
appearing in $V-W$. But, $V-W \iso a(\LL-1)$, so we can also write the sign
as $(-1)^{\alpha^G}$.
\end{proof}

We can now calculate the additive structure of $\Mackey H_G^\bullet(EG_+)$,
verifying the additive parts of Theorems~\ref{thm:EvenEG} and~\ref{thm:OddEG}.

\begin{proposition}\label{prop:EGAdditive}
If $p$ is odd, then
\[
 \Mackey H_G^{\alpha}(EG_+) \iso
  \begin{cases}
   \Mackey R\Z & \text{if $|\alpha| = 0$} \\
   \conc{\Z/p} & \text{if $|\alpha| > 0$ and $\alpha^G$ is even} \\
   0 & \text{otherwise.}
  \end{cases}
\]
If $p$ is even, then
\[
 \Mackey H_G^{\alpha}(EG_+) \iso
  \begin{cases}
   \Mackey R\Z & \text{if $|\alpha| = 0$ and $\alpha^G$ is even} \\
   \Mackey R\Z_- & \text{if $|\alpha| = 0$ and $\alpha^G$ is odd} \\
   \conc{\Z/2} & \text{if $|\alpha|>0$ and $\alpha^G$ is even} \\
   0 & \text{otherwise.}
  \end{cases}
\]
\end{proposition}

\begin{proof}
Consider the equivariant bundle $G/e \to EG \to BG$
and let $\beta\in RO(G)$ with $|\beta| = 0$. 
There is a Serre spectral sequence
\[
 E_2^{a,b} = \Mackey H_G^{a}(BG_+;\Mackey H_G^{b+\beta}(G/e_+;\Mackey A_{G/G}))
  \convto \Mackey H_G^{a+b+\beta}(EG_+;\Mackey A_{G/G}).
\]
In the $E_2$ term we are using local coefficients on $BG$.
Applying Proposition~\ref{prop:cohomologyG}, we see that the spectral sequence
collapses to the line $b = 0$, giving us
\[
 \Mackey H_G^{a+\beta}(EG_+;\Mackey A_{G/G})
  \iso \Mackey H_G^a(BG_+;(\Mackey A_{G/e})_\beta),
\]
where $(\Mackey A_{G/e})_\beta$ denotes the possibly twisted coefficient system
on $BG$ determined by the action of $\gamma$ computed in
Proposition~\ref{prop:cohomologyG}.
Proposition~\ref{prop:cohomologyTrivialAction} gives us
\[
 \Mackey H_G^a(BG_+;(\Mackey A_{G/e})_\beta)(G/e) \iso H_G^a(BG;\Z[G])
  \iso \begin{cases}
         \Z & \text{if $a = 0$} \\
         0 & \text{otherwise}
       \end{cases}
\]
and
\[
 \Mackey H_G^a(BG_+;(\Mackey A_{G/e})_\beta)(G/G)
  \iso H_G^a(BG;\Z_\beta)
\]
where $\Z_\beta$ denotes the local coefficient system on $BG$ determined by the action
from Proposition~\ref{prop:cohomologyG}. If $p$ is odd, this action is trivial,
so we get the usual cohomology of the group $G$, which is easily computed and well known:
\[
 H_G^a(BG;\Z) \iso
  \begin{cases}
    \Z & \text{if $a = 0$} \\
    \Z/p & \text{if $a > 0$ is even} \\
    0 & \text{otherwise.}
  \end{cases}
\]
If $p=2$ and $\beta^G$ is even, then the action is again trivial and the same
calculation holds. However, if $\beta^G$ is odd, then we have the nontrivial
coefficient system and we see the cohomology of $G$ with coefficients in the
nontrivial $\Z[G]$-module $\Z_-$. This is also well-known:
\[
 H_G^a(BG;\Z_-) \iso
  \begin{cases}
   \Z/2 & \text{if $a>0$ is odd} \\
   0 & \text{otherwise.}
  \end{cases}
\]
Note that, because we are assuming $\beta^G$ odd here, the condition that
$a$ is odd can be replaced by the condition that $(a+\beta)^G$ is even.

With $\beta = \alpha - |\alpha|$, we now have almost all of the calculations
in the statement of the proposition
except for one: 
In the case $p$ odd, or $p=2$ and $\beta^G$ even, we need to know the restriction map
\[
 \Z = H_G^0(BG;\Z) \to H_G^0(BG;\Z[G]) = \Z
\]
and the transfer map going the other way. But, these are induced by
$\cdot N\colon \Z\to \Z[G]$ and $\epsilon\colon \Z[G]\to \Z$, respectively, on coefficients,
so are easily calculated to be the identity and multiplication by $p$, respectively.
This confirms that $\Mackey H_G^\beta(EG_+) \iso \Mackey R\Z$ when
$p$ is odd or when $p=2$ and $\beta^G$ is even.
\end{proof}

To determine the multiplicative structure of $\Mackey H_G^\bullet(EG_+)$,
we need the following lemmas.

\begin{lemma}\label{lem:SMk}
If $p$ is odd and $S(\MM_1)$ denotes the unit sphere in $\MM_1$, then
\[
 \Mackey H_G^{\alpha}(S(\MM_1)_+) \iso
  \begin{cases}
    \Mackey R\Z & \text{if $|\alpha| = 0$} \\
    \Mackey L\Z & \text{if $|\alpha| = 1$} \\
    0 & \text{otherwise.}
  \end{cases}
\]
\end{lemma}

\begin{proof}
$G$ acts freely on $S(\MM_1)$ by rotations. 
If we pick any orbit $G/e\to S(\MM_1)$, we get the following cofibration sequence:
\[
 G/e_+ \to S(\MM_1)_+ \to \susp(G/e_+) \xrightarrow{1-t} \susp(G/e_+).
\]
This is essentially an expression of the cell structure of $S(\MM_1)$, with one
free 0-cell and one free 1-cell, attached to the 0-cell by 1 on one end and $t$ on the other.
Taking cohomology, we get the following exact sequence, 
for any $\beta$ with $|\beta| = 0$:
\[
 0 \to \Mackey H_G^{\beta}(S(\MM_1)_+) \to \Mackey A_{G/e}
  \xrightarrow{1-t} \Mackey A_{G/e}
  \to \Mackey H_G^{1+\beta}(S(\MM_1)_+) \to 0,
\]
and $\Mackey H_G^{a+\beta}(S(\MM_1)_+) = 0$ for any integer $a$ other than 0 or 1.
The kernel of $1-t$ is $\Mackey R\Z$ and its cokernel is $\Mackey L\Z$,
so we get the calculation stated in the lemma.
\end{proof}

Recall that, when $p=2$, we write $\MM_1 = \LL^2$.

\begin{lemma}\label{lem:EGInvertibles}
For each integer $1\leq k\leq p/2$,
our chosen nonequivariant identification of $\MM_k$ with $\R^2$ determines
an invertible element
$\xi_k\in \tilde H_G^{\MM_k-2}(EG_+)$.
\end{lemma}

\begin{proof}
Consider the vector bundle $q\colon EG\times \R^2\to EG$. Because $EG$ has free $G$-action,
and the action of $G$ on $\MM_k$ preserves nonequivariant orientation,
$q$ is an $\MM_k$-bundle in the sense of \cite{CMW:orientation}.
By \cite{CW:ordinaryhomology}, it has a Thom class
\[
 \xi_k \in \tilde H_G^{\MM_k}(\susp^2EG_+) \iso \tilde H_G^{\MM_k-2}(EG_+)
\]
determined by the nonequivariant identification of $\MM_k$ and $\R^2$,
and multiplication by $\xi_k$ gives the Thom isomorphism,
\[
 -\cup \xi_k\colon \tilde H_G^{\alpha}(EG_+) \xrightarrow{\iso} \tilde H_G^{\alpha+\MM_k}(\susp^2EG_+)
  \iso \tilde H_G^{\alpha+\MM_k-2}(EG_+).
\]
Because multiplication by $\xi_k$ is an isomorphism, $\xi_k$ is invertible in this ring.
\end{proof}

\begin{remark}\label{rem:extendedXi}
We have been careful to emphasize that $\xi_k$ is determined by our chosen
identification of $\MM_k$ with $\R^2$. When $p$ is odd, we also want to define elements $\xi_k$
for $(p+1)/2\leq k \leq p-1$ corresponding to the complex representations $\C_k$
or the oriented representations $\MM_k$.
Because the identification chosen in Definition~\ref{def:chosenidentification}
respects the nonequivariant identification,
we get
\[
 \xi_k = \xi_{p-k} \quad\text{for $(p+1)/2 \leq k\leq p-1$.}
\]
\end{remark}

Recall the notation $\xi^\alpha$ introduced in Definition~\ref{def:xipower}.
We can now verify the multiplicative parts of
Theorems~\ref{thm:EvenEG} and~\ref{thm:OddEG}.

\begin{theorem}\label{thm:EGCohomology}
If $p$ is odd, then
\[
 \Mackey H_G^\bullet(EG_+) \iso \Mackey R\Z[e_1,\xi_k,\xi_k^{-1} \mid 1\leq k \leq (p-1)/2]
   /\langle \rho(e_1) \rangle,
\]
where 
\begin{align*}
 e_1 &\in\Mackey H_G^{\MM_1}(EG_+)(G/G) \quad\text{and} \\
 \xi_k &\in \Mackey H_G^{\MM_k-2}(EG_+)(G/G).
\end{align*}
If $p=2$, then
\[
 \Mackey H_G^\bullet(EG_+) \iso
  \Mackey R\Z[e, \iota, \iota^{-1}, \xi, \xi^{-1}]
  /\langle \rho(e), \tau\iota, \rho(\xi)-\iota^2 \rangle,
\]
where
\begin{align*}
 e &\in \Mackey H^{\LL}(EG_+)(G/G), \\
 \iota &\in \Mackey H^{\LL-1}(EG_+)(G/e), \quad\text{and} \\
 \xi &\in \Mackey H^{2\LL-2}(EG_+)(G/G).
\end{align*}
\end{theorem}

\begin{proof}
Consider the case $p$ odd first.
From Lemma~\ref{lem:EGInvertibles}, we can see that, if $|\alpha| = 0$
(hence $\alpha^G$ is even),
then $\Mackey H_G^\alpha(EG_+) \iso \Mackey R\Z$ is generated by $\xi^\alpha$.

To see the effect of multiplication by $e_1$, we look at the Gysin sequence of
the bundle $EG\times \MM_1\to EG$. Note that the projection
$EG\times S(\MM_1)\to S(\MM_1)$ is a $G$-homotopy equivalence because $S(\MM_1)$ is free.
Therefore, the Gysin sequence takes the following form:
\[
 \Mackey H_G^{\alpha-1+\MM_1}(S(\MM_1)_+) \to
  \Mackey H_G^{\alpha}(EG_+) \xrightarrow{\cdot e_1} \Mackey H_G^{\alpha+\MM_1}(EG_+)
  \to \Mackey H_G^{\alpha+\MM_1}(S(\MM_1)_+).
\]
From Lemma~\ref{lem:SMk} we see that multiplication by $e_1$ is an isomorphism
if $|\alpha|>0$ and
an epimorphism if $|\alpha| = 0$.
Inductively, we see that, if $|\alpha| = 0$ and $n\geq 1$, then
$\Mackey H_G^{\alpha+n\MM_1}(EG_+) \iso \conc{\Z/p}$ is generated by
$e_1^n\xi^\alpha$. 
Note that $\rho(e_1) = 0$ and that this relation is necessary and sufficient
to say that $e_1$ generates a copy of the $\Mackey R\Z$-module $\conc{\Z/p}$.
From the additive calculation in
Proposition~\ref{prop:EGAdditive}, 
we can see that $\Mackey R\Z[e_1,\xi_k,\xi_k^{-1} \mid 1\leq k < p/2]$ maps
onto $\Mackey H_G^\bullet(EG_+)$, with kernel $\langle \rho(e_1)\rangle$, hence
we get the multiplicative result of the theorem.

Now consider the case $p=2$.
We take $\xi = \xi_1$, the invertible element from
Lemma~\ref{lem:EGInvertibles}.
We defined $\iota$ in Definition~\ref{def:iota}.
As noted after that definition, $\iota$ is invertible, meaning that there is an element
$\iota^{-1}$ such that $\iota\cdot\iota^{-1} = \rho(1)$.
Note that $\tau\iota = 0$ (it lives in a 0 group)
and that this relation is sufficient to say that $\iota$
generates a copy of $\Mackey R\Z_- \iso \Mackey H_G^{\LL-1}(EG_+)$.

Now $\iota^2$ is the similar element we would choose using the induced nonequivariant
identification of $\LL^2$ with $\R^2$. On the other hand, so is $\rho(\xi)$ because
$\xi$ was obtained using that same identification. Hence, $\rho(\xi) = \iota^2$.

If $|\alpha| = 0$ and $\alpha^G = 2m$ is even, it follows that
$\Mackey H_G^{\alpha}(EG_+)\iso \Mackey R\Z$ is generated by $\xi^{-m}$, 
because $\xi$ is invertible.

Finally, we consider the Gysin sequence of the bundle $EG\times\LL\to EG$.
Because $S(\LL) \iso G/e$, the sequence takes the form
\[
 \Mackey H_G^{\alpha-1+\LL}(G/e_+) \to
  \Mackey H_G^{\alpha}(EG_+) \xrightarrow{\cdot e} \Mackey H_G^{\alpha+\LL}(EG_+)
  \to \Mackey H_G^{\alpha+\LL}(G/e_+).
\]
It follows that multiplication by $e$ is an epimorphism if $|\alpha| = 0$ and
an isomorphism if $|\alpha|>0$. Inductively, we see that, if $|\alpha|=0$,
$\alpha^G = 2m$ is even,
and $n\geq 1$, then $\Mackey H_G^{\alpha+n\LL}(EG_+) \iso \conc{\Z/2}$
is generated by $e^n\xi^{-m}$. Again, the single relation $\rho(e) = 0$ suffices
to make $e$ generate a copy of $\conc{\Z/2}$.

From the additive calculation in
Proposition~\ref{prop:EGAdditive}, 
we see that $\Mackey R\Z[e, \iota, \iota^{-1}, \xi, \xi^{-1}]$ maps
onto $\Mackey H_G^\bullet(EG_+)$, with kernel 
$\langle \rho(e), \tau\iota, \rho(\xi)-\iota^2 \rangle$, hence
we get the multiplicative result of the theorem.
\end{proof}

Figure~\ref{fig:EvenEG} shows $\Mackey H_G^\bullet(EG_+)$ for $p$ even, while
Figure~\ref{fig:OddEG} shows the result for $p$ odd.

In the case $p$ odd, what about the other Euler classes $e_k$, for $k>1$?
Recall the elements $\mu_{j,k,d}$ from Definition~\ref{def:mu}.
The following result gives the actions of these elements and also locates the
Euler classes $e_k$.

\begin{proposition}\label{prop:muTimesXi}
Let $p$ be odd, let $1 \leq j, k \leq (p-1)/2$, and suppose $d \equiv kj^{-1} \pmod p$.
Then, in $\Mackey H_G^\bullet(EG_+)$ we have
\[
  \mu_{1,j,j}e_1 = e_j \quad\text{and}\quad \mu_{j,k,d}\xi_j = d\xi_k.
\]
From these it follows that
\begin{align*}
  \mu_{j,k,d} &= d\xi_j^{-1}\xi_k,  \quad\text{and} \\
  e_j &= j\xi_1^{-1}\xi_j e_1.
\end{align*}
\end{proposition}

\begin{proof}
The equality $\mu_{1,j,j}e_1 = e_j$ was shown
in Proposition~\ref{prop:muTimesE}. To see the second equality, consider
the map
\[
 1\smsh\mu_{j,k,d}\colon EG_+\smsh S^{\MM_j} \to EG_+\smsh S^{\MM_k}.
\]
This induces the following diagram in cohomology:
\[
 \xymatrix{
  \Mackey H_G^2(EG_+\smsh S^{\MM_k})(G/G) \ar[r]^{\mu^*_{j,k,d}} \ar[d]_{\rho}
    & \Mackey H_G^2(EG_+\smsh S^{\MM_j})(G/G) \ar[d]^{\rho} \\
  \Mackey H_G^2(EG_+\smsh S^{\MM_k})(G/e) \ar[r]_{\mu^*_{j,k,d}}
   & \Mackey H_G^2(EG_+\smsh S^{\MM_j})(G/e)
 }
\]
Because $\Mackey H_G^2(EG_+\smsh S^{\MM_k}) \iso \Mackey H_G^{2-\MM_k}(EG_+) \iso \Mackey R\Z$,
each group in the diagram above is a copy of $\Z$, and the vertical maps $\rho$ are identities.
Because $EG_+ \hmtpc S^0$ nonequivariantly, and $\mu_{j,k,d}$ has nonequivariant degree $d$
as a map from $S^2$ to itself, the bottom horizontal map is multiplication by $d$,
hence the top horizontal map is as well.
But the top left group is generated by $\xi_k^{-1}$ and the top right by $\xi_j^{-1}$,
hence we get
\[
 \mu_{j,k,d}\xi_k^{-1} = d\xi_j^{-1},
\]
so
\[
 \mu_{j,k,d}\xi_j = d\xi_k \quad\text{and}\quad \mu_{j,k,d} = d\xi_j^{-1}\xi_k
\]
as claimed. 
The last equation of the
proposition is
\[
 e_j = \mu_{1,j,j}e_1 = j\xi_1^{-1}\xi_j e_1.\qedhere
\]
\end{proof}

The last equality of the proposition can also be written as
\[
 \xi_j^{-1}e_j = j\xi_1^{-1}e_1,
\]
which compares two elements of $\Mackey H_G^2(EG_+) \iso \conc{\Z/p}$.

\section{The cohomology of $\tE G$}

We now calculate $\Mackey H_G^\bullet(\tE G)$. 
Recall that $\Mackey H_G^\bullet(S^0)$ acts on $\Mackey H_G^\bullet(\tE G)$,
so we can consider the action of the Euler classes on $\Mackey H_G^\bullet(\tE G)$.

\begin{proposition}\label{prop:tEGPeriodicity}
For $p=2$, multiplication by $e$ induces an isomorphism on $\Mackey H_G^\bullet(\tE G)$.
For $p$ odd, the same is true for multiplication by each $e_k$.
\end{proposition}

\begin{proof}
For any based $G$-space, the inclusion $\tE G\smsh X^G\to \tE G\smsh X$
is a weak equivalence.
Apply this to $X = S^V$ with $V^G = 0$:
The inclusion $1\smsh e_V\colon \tE G\to \tE G\smsh S^V$ is an equivalence,
hence the induced map in cohomology is an isomorphism.
But, this map is multiplication by the Euler class $e_V$,
so multiplication by this Euler class is an isomorphism on cohomology.
For $p = 2$, this applies to $e$.
For $p$ odd, this applies to each $e_k$.
\end{proof}

This suggests introducing the following ring.

\begin{definition}\label{def:ealpha}
\hspace{2em}
\begin{enumerate}
\item
If $p=2$, let $\Mackey E^\bullet = \Mackey A_{G/G}[e,e^{-1}]$ be 
the $RO(G)$-graded ring generated by an invertible element
$e \in \Mackey E^\LL$.
If $\alpha\in RO(G)$ with $\alpha^G = 0$, so $\alpha = n\LL$, then
$\Mackey E^{\alpha}$ is generated by $e^{n} = e^{\alpha}$.

\item
If $p$ is odd,
let $\Mackey E^\bullet = \Mackey A_{G/G}[e_k, e_k^{-1} \mid 1 \leq k \leq (p-1)/2]$
be the $RO(G)$-graded ring
generated by invertible elements $e_k \in \Mackey E^{\MM_k}$.
If $\alpha\in RO(G)$ satisfies $\alpha^G = 0$, 
then $\Mackey E^\alpha$ is generated by $e^\alpha$.
\end{enumerate}
\end{definition}

By Proposition~\ref{prop:tEGPeriodicity}, $\Mackey E^\bullet$ acts
on $\Mackey H_G^\bullet(\tE G)$, where $e$ or $e_k$ acts as the Euler class of that name.
It follows, then, that it suffices to calculate the integer-graded part of the cohomology,
but this is easy to do from our previous calculations.
The following holds for any prime $p$.

\begin{proposition}\label{prop:intTEGCohomology}
If $n\in\Z$, then
\[
 \Mackey H_G^n(\tE G) \iso
  \begin{cases}
   \conc{\Z} & \text{if $n = 0$} \\
   \conc{\Z/p} & \text{if $n \geq 3$ is odd} \\
   0 & \text{otherwise.}
  \end{cases}
\]
Moreover, we have a short exact sequence
\[
 0 \to \Mackey H_G^0(\tE G) \to \Mackey A_{G/G} \xrightarrow{\epsilon} \Mackey R\Z \to 0
\]
and the connecting homomorphism $\delta$ induces isomorphisms
\[
 \delta\colon \Mackey H_G^{2k}(EG_+) \xrightarrow{\iso} \Mackey H_G^{2k+1}(\tE G)
\]
for $k\geq 1$.
\end{proposition}

\begin{proof}
Consider the long exact cohomology sequence in integer grading induced by the cofibration
sequence $EG_+\to S^0\to \tE G$. By the dimension axiom, the only nonzero cohomology of $S^0$
in integer grading is $\Mackey H_G^0(S^0) \iso \Mackey A_{G/G}$.
The map $\Mackey H_G^0(S^0)\to \Mackey H_G^0(EG_+) \iso \Mackey R\Z$ is the unit
map $\epsilon$, which is an epimorphism.
This implies the
short exact sequence
\[
 0 \to \Mackey H_G^0(\tE G) \to \Mackey H_G^0(S^0) \to \Mackey H_G^0(EG_+) \to 0
\]
and the computation that
\[
 \Mackey H_G^0(\tE G)\iso \ker\epsilon \iso \conc\Z,
\]
generated by the element $\kappa\in A(G)$.
The dimension axiom also then implies the isomorphism 
$\conc{\Z/p}\iso\Mackey H_G^{2k}(EG_+) \iso \Mackey H_G^{2k+1}(\tE G)$ for $k\geq 1$. 
Together with our previous computations of $\Mackey H_G^\bullet(EG_+)$,
we get that $\Mackey H_G^n(\tE G) = 0$ except for the cases listed
in the proposition.
\end{proof}

Putting everything together, we get the following restatement of
Theorems~\ref{thm:EvenTEG} and~\ref{thm:OddTEG},
although here we give the structure as a module over $\Mackey E^\bullet$
rather than $\Mackey H_G^\bullet(S^0)[e_k^{-1}]$.

\begin{theorem}\label{thm:tEGCohomology}
\[
 \Mackey H_G^\alpha(\tE G) \iso
  \begin{cases}
    \conc{\Z} & \text{if $\alpha^G = 0$} \\
    \conc{\Z/p} & \text{if $\alpha^G\geq 3$ is odd} \\
    0 & \text{otherwise.}
  \end{cases}
\]
For $p=2$,
as a module over $\Mackey E^\bullet$,
$\Mackey H_G^\bullet(\tE G)$ is generated by elements
\begin{align*}
 \kappa &\in \Mackey H_G^0(\tE G)(G/G) \quad\text{and}\\
 \delta\xi^{-k} &\in \Mackey H_G^{1 - 2k(\LL-1)}(\tE G)(G/G)
\end{align*}
for $k\geq 1$,
such that $\rho(\kappa) = 0$, $\rho(\delta\xi^{-k}) = 0$, and
$2\delta\xi^{-k} = 0$.
For $p$ odd,
as a module over $\Mackey E^\bullet$,
$\Mackey H_G^\bullet(\tE G)$ is generated by elements
\begin{align*}
 \kappa &\in \Mackey H_G^0(\tE G)(G/G) \quad\text{and}\\
 \delta\xi_1^{-k} &\in \Mackey H_G^{1 - k(\MM_1-2)}(\tE G)(G/G)
\end{align*}
for $k \geq 1$,
such that $\rho(\kappa) = 0$, $\rho(\delta\xi_1^{-k}) = 0$, and
$p\delta\xi_1^{-k} = 0$.
\end{theorem}

\begin{proof}
The additive structure follows from Propositions~\ref{prop:tEGPeriodicity}
and~\ref{prop:intTEGCohomology}.

For the module structure,
consider the case $p=2$ first. By Proposition~\ref{prop:intTEGCohomology},
$\Mackey H_G^0(\tE G)\iso\conc\Z$ is generated by $\kappa$, with the relation $\rho(\kappa)=0$.
It follows that, for every integer $n$, $\Mackey H_G^{n\LL}(\tE G) \iso\conc\Z$,
generated by $e^n\kappa$.

By that same proposition, for an integer $k\geq 1$,
\[
 \delta\colon \Mackey H_G^{2k}(EG_+)\iso \Mackey H_G^{2k+1}(\tE G) \iso\conc{\Z/2}.
\]
The generator of $\Mackey H_G^{2k}(EG_+)$ is $e^{2k}\xi^{-k}$, so the generator
of $\Mackey H_G^{2k+1}(\tE G)$ is $\delta(e^{2k}\xi^{-k})$.
Now, $\delta$ is a map of $\Mackey H_G^\bullet(S^0)$-modules, so
$\delta(e^{2k}\xi^{-k}) = e^{2k}\delta\xi^{-k}$.
It follows that, for every integer $n$,
$\Mackey H_G^{2k+1+n\LL}(\tE G)$ is generated by
$e^{2k+n}\delta\xi^{-k}$.
This gives the claimed generators of the module structure.

For $p$ odd the argument is similar.
We use that $e_1^k\xi_1^{-k}$ generates $\Mackey H_G^{2k}(EG_+)$ for $k\geq 1$
to see that $e_1^k\delta\xi_1^{-k}$ generates $\Mackey H_G^{2k+1}(\tE G)$,
hence, for $\alpha^G = |\alpha| = 0$,
$\Mackey H_G^{2k+1+n\MM_1 + \alpha}(\tE G)$ is generated by 
$e^\alpha e_1^{k+n}\delta\xi_1^{-k}$.
\end{proof}

Figure~\ref{fig:EvenTEG} shows $\Mackey H_G^\bullet(\tE G)$ for $p$ even
and Figure~\ref{fig:OddTEG} shows the result for $p$ odd.

\section{The cohomology of a point for $p=2$}

We can now put the calculations of the cohomologies of $EG$ and $\tE G$ together
to get the cohomology of a point. The cases $p=2$ and $p$ odd are sufficiently different
that we shall consider them separately, considering the case $p=2$ in this section.

We first need to know the map 
$\delta\colon \Mackey H_G^\bullet(EG_+)\to \Mackey H_G^{\bullet+1}(\tE G)$.
To help visualize it, the map takes Figure~\ref{fig:EvenEG}
to Figure~\ref{fig:EvenTEG} with a shift of one to the right and up.
The following result gives the calculation.

\begin{proposition}
The map $\delta\colon \Mackey H_G^\bullet(EG_+)\to \Mackey H_G^{\bullet+1}(\tE G)$ is given by
\begin{align*}
  \delta(\iota^k) &= 0 \text{ for all $k$ and} \\
  \delta(e^m\xi^n) &=
   \begin{cases}
    e^m\delta\xi^n & \text{if $n\leq -1$} \\
    0 & \text{otherwise.}
   \end{cases}
\end{align*}
\end{proposition}

\begin{proof}
That $\delta(\iota^k) = 0$ follows because $\Mackey H_G^\bullet(\tE G)(G/e) = 0$ 
($\tE G$ is nonequivariantly contractible).

For $\delta(e^m\xi^n)$, if $n\geq 0$ then the target of $\delta$ is 0.
On the other hand, that $\delta(e^m\xi^n) = e^m\delta\xi^n$ for $n\leq -1$ was
shown in the proof of Theorem~\ref{thm:tEGCohomology}.
\end{proof}

This determines most of the structure
of $\Mackey H_G^\bullet(S^0)$, the rest being edge cases.
The following verifies the additive part of Theorem~\ref{thm:evenCohomPoint}.

\begin{theorem}\label{thm:evenCohomPointRedux}
Let $p = 2$. Then
\[
 \Mackey H_G^\alpha(S^0) \iso
  \begin{cases}
   \Mackey A_{G/G} & \text{if $\alpha = 0$} \\
   \Mackey R\Z & \text{if $|\alpha| = 0$ and $\alpha^G < 0$ is even} \\
   \Mackey R\Z_- & \text{if $|\alpha| = 0$ and $\alpha^G \leq 1$ is odd} \\
   \Mackey L\Z & \text{if $|\alpha| = 0$ and $\alpha^G > 0$ is even} \\
   \Mackey L\Z_- & \text{if $|\alpha| = 0$ and $\alpha^G \geq 3$ is odd} \\
   \conc\Z & \text{if $|\alpha| \neq 0$ and $\alpha^G = 0$} \\
   \conc{\Z/2} & \text{if $|\alpha| > 0$ and $\alpha^G < 0$ is even} \\
   \conc{\Z/2} & \text{if $|\alpha| < 0$ and $\alpha^G \geq 3$ is odd} \\
   0 & \text{otherwise.}
  \end{cases}
\]
When $\alpha^G<0$, the map $\Mackey H_G^\alpha(S^0)\to \Mackey H_G^\alpha(EG_+)$ is an isomorphism.
When $|\alpha|<0$, $\Mackey H_G^\alpha(\tE G)\to \Mackey H_G^\alpha(S^0)$ is an isomorphism.
We have a short exact sequence
\[
 \xymatrix@R=2.5ex{
 0 \ar[r] & \Mackey H_G^0(\tE G) \ar@{=}[d] \ar[r]
  & \Mackey H_G^0(S^0) \ar@{=}[d] \ar[r]
  & \Mackey H_G^0(EG_+) \ar@{=}[d] \ar[r] & 0 \\
 & \conc\Z \ar[r] 
  & \Mackey A_{G/G} \ar[r]
  & \Mackey R\Z\mathrlap{.}
 }
\]
If $\alpha^G = 0$ and $|\alpha| > 0$, we have a short exact sequence
\[
 \xymatrix@R=2.5ex{
 0 \ar[r] & \Mackey H_G^\alpha(\tE G) \ar@{=}[d] \ar[r]
  & \Mackey H_G^\alpha(S^0) \ar@{=}[d] \ar[r]
  & \Mackey H_G^\alpha(EG_+) \ar@{=}[d] \ar[r] & 0 \\
 & \conc\Z \ar[r]_2
  & \conc\Z \ar[r]_{1}
  & \conc{\Z/2}\mathrlap{.}
 }
\]
If $|\alpha| = 0$ and $\alpha^G\geq 3$ is odd, we have a short exact sequence
\[
 \xymatrix@R=2.5ex{
 0 \ar[r] & \Mackey H_G^\alpha(\tE G) \ar@{=}[d] \ar[r]
  & \Mackey H_G^\alpha(S^0) \ar@{=}[d] \ar[r]
  & \Mackey H_G^\alpha(EG_+) \ar@{=}[d] \ar[r] & 0 \\
 & \conc{\Z/2} \ar[r]
  & \Mackey L\Z_- \ar[r]
  & \Mackey R\Z_-\mathrlap{.}
 }
\]
\end{theorem}

\begin{proof}
We have that $\Mackey H_G^\alpha(\tE G) = 0$ for $\alpha^G < 0$, and
$\delta$ is zero on $\Mackey H_G^\alpha(EG_+)$ for $\alpha^G< 0$, which implies that
\[
 \Mackey H_G^{\alpha}(S^0) \xrightarrow{\iso} \Mackey H_G^\alpha(EG_+) \quad\text{if $\alpha^G<0$.}
\]
Similarly, $\Mackey H_G^\alpha(EG_+) = 0$ for $|\alpha| < 0$, which implies that
\[
 \Mackey H_G^\alpha(\tE G) \xrightarrow{\iso} \Mackey H_G^\alpha(S^0) \quad\text{if $|\alpha|<0$.}
\]
On the other hand, $\delta\colon \Mackey H_G^\alpha(EG_+)\to \Mackey H_G^{\alpha+1}(\tE G)$
is an epimorphism when $|\alpha| = 0$ and $\alpha^G>0$, and an isomorphism when
$|\alpha|>0$ and $\alpha^G>0$. That implies that
\[
 \Mackey H_G^\alpha(S^0) = 0 \quad\text{if $|\alpha| > 0$ and $\alpha^G>0$.}
\]
Referring to Figure~\ref{fig:EvenCohomPoint}, the names we are using for the generators
of the groups outside of the first quadrant are the names of the corresponding elements
of $\Mackey H_G^\bullet(EG_+)$ or $\Mackey H_G^\bullet(\tE G)$.
To complete the calculation, what remains is to consider the axes
bordering the first quadrant.

From the dimension axiom, we know that $\Mackey H_G^0(S^0) \iso \Mackey A_{G/G}$,
and, in Proposition~\ref{prop:intTEGCohomology}, we pointed out that we have a short exact
sequence
\[
 \xymatrix@R=2.5ex{
  0 \ar[r] & \Mackey H_G^0(\tE G) \ar@{=}[d] \ar[r]
    & \Mackey H_G^0(S^0) \ar@{=}[d] \ar[r]
    & \Mackey H_G^0(EG_+) \ar@{=}[d] \ar[r]
    & 0 \\
   & \conc\Z \ar[r]_\kappa & \Mackey A_{G/G} \ar[r]_\epsilon & \Mackey R\Z 
 }
\]
Examination of $\delta$ shows that we also have the following short exact sequence for $m\geq 1$:
\[
 \xymatrix@R=2ex{
  0 \ar[r] & \Mackey H_G^{m\LL}(\tE G) \ar@{=}[d] \ar[r]
    & \Mackey H_G^{m\LL}(S^0) \ar[r]
    & \Mackey H_G^{m\LL}(EG_+) \ar@{=}[d] \ar[r]
    & 0 \\
  & \conc\Z & & \conc{\Z/2}
 }
\]
Multiplication by $e^m$ for $m\geq 1$ then induces a map of short exact sequences:
\[
 \xymatrix{
  0 \ar[r] & \conc\Z \ar[d]_\iso \ar[r]^\kappa & \Mackey A_{G/G} \ar[d]_{e^m} \ar[r]^{\epsilon}
    & \Mackey R\Z \ar[d]^\pi \ar[r] & 0 \\
  0 \ar[r] & \conc\Z \ar[r] & \Mackey H_G^{m\LL}(S^0) \ar[r] & \conc{\Z/2} \ar[r] & 0
 }
\]
Here $\pi$ is the evident projection at level $G/G$.
Proposition~\ref{prop:extensions2} then implies that, 
for $m\geq 1$, $\Mackey H_G^{m\LL}(S^0) \iso \conc\Z$ and $e^m$ is a generator.

For the remaining axis, we note first that 
\[
 \Mackey H_G^{1-\LL}(S^0) \xrightarrow{\iso} \Mackey H_G^{1-\LL}(EG_+) \iso \Mackey R\Z_-.
\]
If $k\geq 1$, we have a short exact sequence
\[
 \xymatrix@R=2ex{
 0 \ar[r] & \Mackey H_G^{2k(1-\LL)}(S^0) \ar[r]
  & \Mackey H_G^{2k(1-\LL)}(EG_+) \ar@{=}[d] \ar[r]^\delta
  & \Mackey H_G^{1+2k(1-\LL)}(\tE G) \ar@{=}[d] \ar[r] & 0 \\
 & & \Mackey R\Z \ar[r]_{\pi} & \conc{\Z/2}
 }
\]
Therefore, $\Mackey H_G^{2k(1-\LL)}(S^0) \iso \ker\pi \iso \Mackey L\Z$ for $k\geq 1$.
On the other hand, at the odd multiples of $1-\LL$, we get the following short exact sequence
(where, again, $k\geq 1$):
\[
 \xymatrix@R=2ex@C-1em{
  0 \ar[r] & \Mackey H_G^{(2k+1)(1-\LL)}(\tE G) \ar@{=}[d] \ar[r]
    & \Mackey H_G^{(2k+1)(1-\LL)}(S^0) \ar[r]
    & \Mackey H_G^{(2k+1)(1-\LL)}(EG_+) \ar@{=}[d] \ar[r] & 0 \\
  & \conc{\Z/2} & & \Mackey R\Z_-
 }
\]
There are two possible extensions, $\Mackey L\Z_-$ and $\conc{\Z/2}\dirsum\Mackey R\Z_-$.
Here is a way to resolve the extension problem:
Consider the cofibration sequence $S^{2k\LL}\to S^{(2k+1)\LL} \to \susp^{2k+1}G/e_+$.
(This is the top-dimensional cell in a based $G$-CW structure on $S^{(2k+1)\LL}$.)
We have 
\[
 \Mackey H_G^{2k+1}(S^{2k\LL}) \iso \Mackey H_G^{(2k+1)(1-\LL) + \LL}(S^0) = 0
\]
from our previous calculations.
This gives us the following exact sequence:
\[
 \xymatrix@R=2ex{
  \Mackey H_G^{2k+1}(\susp^{2k+1}G/e_+) \ar@{=}[d] \ar[r]
    & \Mackey H_G^{2k+1}(S^{(2k+1)\LL}) \ar[r]
    & \Mackey H_G^{2k+1}(S^{2k\LL}) \ar@{=}[d] \\
  \Mackey A_{G/e} & & 0
 }
\]
Therefore, $\Mackey H_G^{2k+1}(S^{(2k+1)\LL}) \iso \Mackey H_G^{(2k+1)(1-\LL)}(S^0)$
is a quotient of $\Mackey A_{G/e}$. Any quotient of $\Mackey A_{G/e}$ must have
its transfer function $\tau$ being an epimorphism, which is true for
$\Mackey L\Z_-$ but not for $\conc{\Z/2}\dirsum\Mackey R\Z_-$.
Therefore, $ \Mackey H_G^{(2k+1)(1-\LL)}(S^0) \iso \Mackey L\Z_-$.
(Alternatively, the map 
$S^{(2k+1)\LL}\to \susp^{2k+1}G/e_+ \iso G/e_+\smsh S^{(2k+1)\LL}$
represents the transfer map, so, when evaluated in cohomology at level $G/G$,
gives $\tau$. The exact sequence above, evaluated at level $G/G$, then
shows that $\tau$ is an epimorphism.
It would be nice to have a completely algebraic resolution of this extension
problem, along the lines of the argument given for $\Mackey H_G^{m\LL}(S^0)$.)

The elements $\iota^k$, $k\leq -1$, generate the groups at level $G/e$
when $|\alpha| = 0$ and $\alpha^G > 0$,
hence generate the corresponding Mackey functors, which all have the form
$\Mackey R\Z_-$, $\Mackey L\Z$, or $\Mackey L\Z_-$.
\end{proof}

We now verify the multiplicative part of Theorem~\ref{thm:evenCohomPoint}.

\begin{theorem}\label{thm:evenCohomPointProved}
For $p = 2$, $\Mackey H_G^\bullet(S^0)$ is a strictly commutative $RO(G)$-graded ring,
generated multipicatively by elements
\begin{align*}
 \iota &\in \Mackey H_G^{\LL-1}(S^0)(G/e) \\
 \iota^{-1} &\in \Mackey H_G^{1-\LL}(S^0)(G/e) \\
 \xi &\in \Mackey H_G^{2(\LL-1)}(S^0)(G/G) \\
 e &\in \Mackey H_G^{\LL}(S^0)(G/G) \\
 e^{-m}\kappa &\in \Mackey H_G^{-m\LL}(S^0)(G/G) & & m\geq 1 \\
 e^{-m}\delta\xi^{-n} &\in \Mackey H_G^{1 - m\LL - 2n(\LL-1)}(S^0)(G/G)
   & & m, n \geq 1.
\end{align*}
These generators satisfy the following {\em structural} relations:
\begin{align*}
 \tau(\iota^{-1}) &= 0 \\
 \tau(\iota^{-2n-1}) &= e^{-1}\delta\xi^{-n} & & \text{for $n\geq 1$} \\
 \kappa\xi &= 0 \\
 \rho(\xi) &= \iota^2 \\
 \rho(e) &= 0 \\
 \rho(e^{-m}\kappa) &= 0 & & \text{for $m\geq 1$}\\
 \rho(e^{-m}\delta\xi^{-n}) &= 0 & & \text{for $m\geq 2$ and $n\geq 1$}\\
 2e^{-m}\delta\xi^{-n} &= 0 & & \text{for $m\geq 2$ and $n\geq 1$}
\intertext{and the following {\em multiplative} relations:}
  \iota\cdot \iota^{-1} &= \rho(1) \\
  e\cdot e^{-m}\kappa &= e^{-m+1}\kappa & &\text{for $m\geq 1$} \\
 \xi\cdot e^{-m}\kappa &= 0 & & \text{for $m\geq 1$} \\
 e^{-m}\kappa\cdot e^{-n}\kappa &= 2e^{-m-n}\kappa & & \text{for $m\geq 0$ and $n\geq 0$} \\
 e\cdot e^{-m}\delta\xi^{-n} &= e^{-m+1}\delta\xi^{-n} & &\text{for $m\geq 2$ and $n\geq 1$} \\
 \xi\cdot e^{-m}\delta\xi^{-n} &= e^{-m}\delta\xi^{-n+1} & & 
   \text{for $m\geq 1$ and $n\geq 2$} \\
 \xi\cdot e^{-m}\delta\xi^{-1} &= 0 & & \text{for $m\geq 2$} \\
 e^{-m}\kappa \cdot e^{-n}\delta\xi^{-k} &= 0 & & \text{if $m\geq 0$, $n\geq 1$, and $k\geq 1$}\\
 e^{-m}\delta\xi^{-k}\cdot e^{-n}\delta\xi^{-\ell} &= 0 & & 
   \text{if $m, n, k, \ell\geq 1$}
\end{align*}
The following relations are implied by the preceding ones:
\begin{align*}
 \kappa e &= 2e \\
 2e^m\xi^n &= 0 & & \text{if $m>0$ and $n>0$} \\
 t\iota^k &= (-1)^k\iota^k & & \text{for all $k$} \\
 \xi\cdot\tau(\iota^k) &= \tau(\iota^{k+2}) & &\text{for all $k$} \\
 e\cdot\tau(\iota^k) &= 0 & &\text{for all $k$} \\
 e^{-m}\kappa \cdot \tau(\iota^k) &= 0 & & \text{for all $m\geq 1$ and $k$} \\
 e^{-m}\delta\xi^{-n}\cdot \tau(\iota^k) &= 0 & & \text{for all $m, n\geq 1$ and $k$} \\
 \tau(\iota^k)\cdot\tau(\iota^\ell) &= 0 & &\text{if $k$ or $\ell$ is odd} \\
 \tau(\iota^{2k})\cdot\tau(\iota^{2\ell}) &= 2\tau(\iota^{2(k+\ell)}) & &\text{for all $k$ and $\ell$}\\
 \tau(\iota^{2k+1}) &= 0 & & \text{if $k\geq 0$} \\
 e\cdot e^{-1}\delta\xi^{-n} &= 0 & & \text{if $n\geq 1$}
\end{align*}
\end{theorem}

\begin{proof}
The structural relations listed in the theorem follow from the additive calculation
and the discussion of generators and relations for Mackey functors
in \S\ref{sec:MackeyFunctors}.
Note that the structural relation $\tau(\iota^{-2n-1}) = e^{-1}\delta\xi^{-n}$
follows from the proof of the structure of $\Mackey H_G^{(2n+1)(1-\LL)}(S^0)$.

At level $G/e$, we have
\[
 \Mackey H_G^\alpha(S^0)(G/e) \iso \Mackey H_G^\alpha(EG_+)(G/e) \iso  \tilde H^{|\alpha|}(S^0;\Z)
\]
for all $\alpha$, so $\iota$ is invertible at level $G/e$ as already noted for $EG_+$.

The relation $e\cdot e^{-m}\kappa = e^{-m+1}\kappa$, $m\geq 1$, follows from the same
identity in the cohomology of $\tE G$.
The identity $\xi\cdot e^{-m}\kappa = 0$ for $m\geq 1$ follows because the group
in which the product would live is 0.

The relation $e\cdot e^{-m}\delta\xi^{-n} = e^{-m+1}\delta\xi^{-n}$ ($m\geq 2$ and $n\geq 1$)
follows from the same identity in the cohomology of $\tE G$.

For $n\geq 2$, we have 
$\xi\cdot e^{-m}\delta\xi^{-n} = e^{-m}\delta(\xi\cdot\xi^{-n}) = e^{-m}\delta\xi^{-n+1}$.
On the other hand, $\xi\cdot e^{-m}\delta\xi^{-1} = 0$ for $m\geq 2$ because the product
lives in a 0 group.

To determine the product $e^{-m}\kappa \cdot e^{-n}\kappa$, multiply by $e^{m+n}$:
\[
 e^{m+n}(e^{-m}\kappa \cdot e^{-n}\kappa) = \kappa^2 = 2\kappa,
\]
hence $e^{-m}\kappa \cdot e^{-n}\kappa = 2e^{-m-n}\kappa$. Similarly,
\[
 e^{m}(e^{-m}\kappa \cdot e^{-n}\delta\xi^{-k})
  = \kappa e^{-n}\delta\xi^{-k} = 0
\]
because both $2e^{-n}\delta\xi^{-k} = 0$ and $ge^{-n}\delta\xi^{-k} = 0$;
hence $e^{-m}\kappa \cdot e^{-n}\delta\xi^{-k} = 0$.

Finally for the basic relations, $e^{-m}\delta\xi^{-k}\cdot e^{-n}\delta\xi^{-\ell} = 0$
when either $m\geq 2$ or $n\geq 2$ because the product lives in a 0 group.
When $m = n = 1$, this is
\[
 e^{-1}\delta\xi^{-k}\cdot e^{-1}\delta\xi^{-\ell} 
   = \tau(\iota^{-2k-1})\cdot\tau(\iota^{-2\ell-1}),
\]
hence follows from the more general calculation to be done below.

Turning to the remaining relations listed in the theorem,
$ge = \tau\rho(e) = 0$, so $\kappa e = (2 - g)e = 2e$.
We then have
\[
 2e\xi = \kappa e\xi = e(\kappa\xi) = 0,
\]
which implies that $2e^m\xi^n = 0$ for all $m\geq 1$ and $n\geq 1$.

We have
\[
 (1+t)\iota^{-1} = \rho\tau(\iota^{-1}) = 0,
\]
so $t\iota^{-1} = -\iota^{-1}$. For any $k$, we then have
\[
 t\iota^{k} = (t\iota^{-1})^{-k} = (-\iota^{-1})^{-k} = (-1)^k\iota^{k}.
\]
In particular, $\rho\tau(\iota^k) = (1+t)\iota^k = 2\iota^k$ if $k$ is even,
but equals 0 if $k$ is odd.

The next batch of relations follow from the Frobenius relation:
\begin{align*}
 \xi\tau(\iota^k) &= \tau(\rho(\xi)\iota^k) = \tau(\iota^2\iota^k) = \tau(\iota^{k+2}) \\
 e\tau(\iota^k) &= \tau(\rho(e)\iota^k) = 0 \\
 e^{-m}\kappa \tau(\iota^k) &= \tau(\rho(e^{-m}\kappa)\iota^k) = 0 \\
 e^{-m}\delta\xi^{-n} \tau(\iota^k) &= \tau(\rho(e^{-m}\delta\xi^{-n})\iota^k) = 0.
\end{align*}
If $k$ is odd, then $\rho\tau(\iota^k) = 0$, so
\[
 \tau(\iota^k)\tau(\iota^\ell) = \tau(\rho\tau(\iota^k)\iota^\ell) = 0,
\]
and similarly if $\ell$ is odd. On the other hand,
$\rho\tau(\iota^{2k}) = 2\iota^{2k}$, so
\[
 \tau(\iota^{2k})\tau(\iota^{2\ell}) = \tau(\rho\tau(\iota^{2k})\iota^{2\ell})
  = 2\tau(\iota^{2(k+\ell)}).
\]

We then have
\[
 \tau(\iota^{2k+1}) = \xi^{k+1}\tau(\iota^{-1}) = 0
\]
for $k\geq 0$.

We also have
\[
 e\cdot e^{-1}\delta\xi^{-n} = e\tau(\iota^{-2n-1}) = 0.
\]

The relations $\tau(\iota^{2k+1}) = 0$ for $k\geq 0$ and
$t\iota^{-2k} = \iota^{-2k}$ for $k\geq 1$
fill in the remaining structural relations we need to see
that all the Mackey functors are determined correctly.

We can now check that all products of the generators have been computed, so
the relations listed suffice to determine the multiplicative structure
of $\Mackey H_G^\bullet(S^0)$.
Moreover, the only possibility for anticommutativity to introduce a sign among
products of generating elements is in the following:
\begin{itemize}
\item
Products at level $G/e$: The only nonzero elements occur in gradings of the form
$a-a\LL$, and the sign introduced by commuting an element in grading $a-a\LL$ with
one in grading $b-b\LL$ is $(-1)^{ab+ab} = 1$.

\item
Products involving $\tau(\iota^k)$: The Frobenius relation and the preceding observation
show that such products strictly commute.

\item
Products involving $e$: The only unit that might be generated in such a product
is $1-g$, but $(1-g)e = e$ because $ge = 0$, so such a product strictly commutes.

\item
Products involving $e^{-m}\kappa$ for $m\geq 1$: Again, 
the only possible sign is $1-g$, but $(1-g)e^{-m}\kappa = e^{-m}\kappa$,
so any such product strictly commutes.

\item
Products involving $e^{-m}\delta\xi^{-n}$: Here, we have
\[
 -e^{-m}\delta\xi^{-n} = e^{-m}\delta\xi^{-n} = (1-g)e^{-m}\delta\xi^{-n},
\]
so all such products must strictly commute.
\end{itemize}
Therefore, $\Mackey H_G^\bullet(S^0)$ is strictly commutative.
\end{proof}

\section{The cohomology of a point for $p$ odd}\label{sec:cohomPointOdd}

In this section we assume that $p$ is odd.

In Definitions~\ref{def:EulerClasses}, \ref{def:iota} and~\ref{def:mu},
we introduced the elements $e_k$, $\iota_k$, and $\mu_{j,k,d}$.

\begin{definition}\label{def:muBetaD}
Let $\alpha = \sum_{k=2}^{(p-1)/2} n_k (\MM_k-\MM_1) \in RO_0(G)$. 
For $1 < k \leq (p-1)/2$, let $a_k\in\Z$ with $a_k \equiv k^{-1} \pmod p$, and let
\[
 a = \bigl(\prod_{k>1, n_k>0} k^{n_k}\bigr) \cdot
             \bigl(\prod_{k>1, n_k<0} a_k^{|n_k|}\bigr)
         \in \Z
\]
Then define
\[
 \mu^{\alpha,a} = \bigl(\prod_{k>1, n_k>0} \mu_{1,k,k}^{n_k}\bigr) \cdot
             \bigl(\prod_{k>1, n_k<0} \mu_{k,1,a_k}^{|n_k|}\bigr)
           \in \Mackey H_G^\alpha(S^0).
\]
\end{definition}

Note that defining the $a_k$, hence also $a$, involves arbitrary choices.
No matter what choices we make, we have $a \in \nu(\alpha) \in (\Z/p)^\times$.
In Corollary~\ref{cor:muBetaD} we shall extend the notation so that $\mu^{\alpha,a}$ is well-defined for
every integer $a$ in the congruence class $\nu(\alpha)$,
but for the moment it suffices to know that we can define it for at least one such $a$.

Recall also Definitions~\ref{def:xipower} and~\ref{def:ealpha}.

\begin{proposition}\label{prop:multByMu}
Let $\alpha\in RO_0(G)$.
In $\Mackey H_G^\bullet(S^0)$, we have
\[
 \rho(\mu^{\alpha,a}) = a \iota^\alpha.
\]
As elements acting on $\Mackey H_G^\bullet(\tE G)$, we have
\[
 \mu^{\alpha,a} \text{ acts as } e^{\alpha}.
\]
Under the map $\Mackey H_G^\bullet(S^0) \to \Mackey H_G^\bullet(EG_+)$, we have
\[
 \mu^{\alpha,a} \mapsto a \xi^\alpha.
\]
\end{proposition}

\begin{proof}
We can write
\[
 \alpha = \sum_{k=2}^{(p-1)/2} n_k(\MM_k - \MM_1),
\]
so that
\[
 \iota^\alpha = \prod_{k=2}^{(p-1)/2} (\iota_1^{-1}\iota_k)^{n_k},
\]
and similarly for $e^\alpha$ and $\xi^\alpha$.
The statements of the proposition follow from,
and are elaborations of,
Propositions~\ref{prop:rhoMu}, \ref{prop:muTimesE} and~\ref{prop:muTimesXi}.
\end{proof}

\begin{proposition}
Let $\alpha\in RO_0(G)$ and $a\in\nu(\alpha)^{-1}$.
The map $\delta\colon \Mackey H_G^\bullet(EG_+)\to \Mackey H_G^{\bullet+1}(\tE G)$ is given by
\[
 \delta(e_1^m\xi^\alpha\xi_1^n) = 
  \begin{cases}
   a e^\alpha e_1^m\delta\xi_1^n & \text{if $n < 0$} \\
   0 & \text{otherwise.}
  \end{cases}
\]
Here, in the first case, the equality takes place in a group isomorphic to $\Z/p$,
so the result does not
depend on the choice of $a\in\nu(\alpha)^{-1}$.
\end{proposition}

\begin{proof}
That $\delta(e_1^m\xi^\alpha\xi_1^n) = 0$ if $n\geq 0$ follows from the fact that
the group it lives in is 0.

While calculating $\Mackey H_G^\bullet(\tE G)$, we saw that
$\delta(e_1^n\xi_1^{-n}) = e_1^n\delta\xi_1^{-n}$ for $n\geq 1$;
these are the cases in integer grading.
Assume that $b$ is chosen so that $\mu^{\alpha,b}$ is defined,
so $b\in\nu(\alpha)$.
Because $\delta$ respects the action of $\mu^{\alpha,b}$, we have
\[
 b\delta(e_1^n\xi^\alpha\xi_1^{-n}) = 
 \delta(e_1^n\mu^{\alpha,b}\xi_1^{-n}) = \mu^{\alpha,b} e_1^n\delta\xi_1^{-n}
 = e^\alpha e_1^n\delta\xi^{-n}.
\]
Because this equality takes place in a group isomorphic to $\Z/p$, we can write
\[
 \delta(e_1^n\xi^\alpha\xi_1^{-n}) = b^{-1} e^\alpha e_1^n\delta\xi^{-n}
  = a e^\alpha e_1^n\delta\xi^{-n}.
\]
Because $\delta$ also respects multiplication by $e_1$, which acts on $\Mackey H_G^\bullet(\tE G)$
by isomorphisms, we get the statement of the proposition.
\end{proof}

The following result verifies the additive part of Theorem~\ref{thm:oddCohomPoint}.

\begin{theorem}\label{thm:oddAdditiveCohomoPoint}
Let $p$ be odd. Then
\[
 \Mackey H_G^{\alpha}(S^0) \iso
  \begin{cases}
   \Mackey A[\nu(\alpha)] & \text{if $|\alpha| = \alpha^G = 0$} \\
   \Mackey R\Z & \text{if $|\alpha| = 0$ and $\alpha^G < 0$ is even} \\
   \Mackey L\Z & \text{if $|\alpha| = 0$ and $\alpha^G > 0$ is even} \\
   \conc\Z & \text{if $|\alpha| \neq 0$ and $\alpha^G = 0$} \\
   \conc{\Z/p} & \text{if $|\alpha| > 0$ and $\alpha^G < 0$ is even} \\
   \conc{\Z/p} & \text{if $|\alpha| < 0$ and $\alpha^G \geq 3$ is odd} \\
   0 & \text{otherwise.}
  \end{cases}
\]
When $\alpha^G<0$, the map $\Mackey H_G^\alpha(S^0) \to \Mackey H_G^\alpha(EG_+)$
is an isomorphism.
When $|\alpha|<0$, $\Mackey H_G^\alpha(\tE G)\to \Mackey H_G^\alpha(S^0)$ is an isomorphism.
When $\alpha^G = 0$ and $|\alpha|= 0$, we have a short exact sequence
\[
 \xymatrix@R=2.5ex{
 0 \ar[r] & \Mackey H_G^\alpha(\tE G) \ar@{=}[d] \ar[r]
  & \Mackey H_G^\alpha(S^0) \ar@{=}[d] \ar[r]
  & \Mackey H_G^\alpha(EG_+) \ar@{=}[d] \ar[r] & 0 \\
 & \conc\Z \ar[r] 
  & \Mackey A[\nu(\alpha)] \ar[r]
  & \Mackey R\Z
 }
\]
while, if $\alpha^G = 0$ and $|\alpha| > 0$, we have a short exact sequence
\[
 \xymatrix@R=2.5ex{
 0 \ar[r] & \Mackey H_G^\alpha(\tE G) \ar@{=}[d] \ar[r]
  & \Mackey H_G^\alpha(S^0) \ar@{=}[d] \ar[r]
  & \Mackey H_G^\alpha(EG_+) \ar@{=}[d] \ar[r] & 0 \\
 & \conc\Z \ar[r]_p
  & \conc\Z \ar[r]_{\nu(\beta)}
  & \conc{\Z/p}
 }
\]
where $\beta = \alpha - (|\alpha|/2)\MM_1$.
\end{theorem}

\begin{proof}
Consider first the case $|\alpha| = \alpha^G = 0$, so $\alpha\in RO_0(G)$.
We know that $\Mackey H_G^{\alpha-1}(EG_+) = 0$ and
$\Mackey H_G^{\alpha+1}(\tE G) = 0$, which gives us the following short exact sequence:
\[
 \xymatrix@R=2.5ex{
 0 \ar[r] & \Mackey H_G^\alpha(\tE G) \ar@{=}[d] \ar[r]
  & \Mackey H_G^\alpha(S^0) \ar[r]
  & \Mackey H_G^\alpha(EG_+) \ar@{=}[d] \ar[r] & 0 \\
 & \conc\Z
  & 
  & \Mackey R\Z
 }
\]
We know from Proposition~\ref{prop:intTEGCohomology} that, when $\alpha=0$, the
sequence is
\[
 \xymatrix@R=2.5ex{
 0 \ar[r] & \Mackey H_G^0(\tE G) \ar@{=}[d] \ar[r]
  & \Mackey H_G^0(S^0) \ar@{=}[d] \ar[r]
  & \Mackey H_G^0(EG_+) \ar@{=}[d] \ar[r] & 0 \\
 & \conc\Z \ar[r] 
  & \Mackey A_{G/G} \ar[r]
  & \Mackey R\Z
 }
\]
Consider the effect of multiplication by
$\mu^{\alpha,a}\in \Mackey H_G^\alpha(S^0)$.
By Proposition~\ref{prop:multByMu},
mutliplication by $\mu^{\alpha,a}$ on $\Mackey H_G^\bullet(\tE G)$ is the same as
multiplication by $e^\alpha$, so takes the generator $\kappa\in\Mackey H_G^0(\tE G)$
to the generator $e^\alpha\kappa\in\Mackey H_G^\alpha(\tE G)$.
By the same proposition, multiplication by $\mu^{\alpha,a}$ on $\Mackey H_G^\bullet(EG_+)$
is the same as multiplication by $a \xi^\alpha$, hence takes
the generator $1\in \Mackey H_G^0(EG_+)$ to $a$ times the generator
$\xi^\alpha\in \Mackey H_G^\alpha(EG_+)$.
Therefore, multiplication by $\mu^{\alpha,a}$ gives the following map of short exact sequences:
\[
 \xymatrix{
 0 \ar[r] & \conc\Z \ar[d]_1 \ar[r] & \Mackey A_{G/G} \ar[d]^{\mu^{\alpha,a}} \ar[r]
  & \Mackey R\Z \ar[d]^{a} \ar[r] & 0 \\
 0 \ar[r] & \conc\Z \ar[r] & \Mackey H_G^\alpha(S^0) \ar[r]
  & \Mackey R\Z \ar[r] & 0
 }
\]
It follows from Proposition~\ref{prop:extensions} that 
$\Mackey H_G^\alpha(S^0) \iso \Mackey A[a] = \Mackey A[\nu(\alpha)]$,
generated by $\mu^{\alpha,a}$ and $\iota^\alpha$.
Note again that $\mu^{\alpha,a}$ maps to $a\xi^\alpha$ in
$\Mackey H_G^\alpha(EG_+)$; we also have that
$\tau(\iota^\alpha)$ maps to $p\xi^\alpha$.

Similarly, if $\alpha^G = 0$ and $|\alpha|>0$, we have
a short exact sequence
\[
 \xymatrix@R=2.5ex{
 0 \ar[r] & \Mackey H_G^\alpha(\tE G) \ar@{=}[d] \ar[r]
  & \Mackey H_G^\alpha(S^0) \ar[r]
  & \Mackey H_G^\alpha(EG_+) \ar@{=}[d] \ar[r] & 0 \\
 & \conc\Z
  & 
  & \conc{\Z/p}
 }
\]
Because $\alpha^G = 0$, we know that $|\alpha|$ is even; let $n = |\alpha|/2$.
Let $\beta = \alpha-n\MM_1$, then $\beta\in RO_0(G)$.
Consider the effect of multiplication by $e_1^n$ on the Mackey functors in grading $\beta$.
It takes the generator $e^\beta\kappa\in\Mackey H_G^\beta(\tE G)$ to the generator
$e^\beta e_1^n\kappa\in\Mackey H_G^{\alpha}(\tE G)$;
it takes the generator $\xi^\beta\in\Mackey H_G^\beta(EG_+)$ to the generator
$e_1^n\xi^\beta\in\Mackey H_G^{\alpha}(EG_+)$.
Therefore, we have the following map of short exact sequences:
\[
 \xymatrix{
 0 \ar[r] & \conc\Z \ar[d]_1 \ar[r] & \Mackey A[\nu(\beta)] \ar[d]^{e_1^n} \ar[r]
  & \Mackey R\Z \ar[d]^{\pi} \ar[r] & 0 \\
 0 \ar[r] & \conc\Z \ar[r] & \Mackey H_G^\alpha(S^0) \ar[r]
  & \conc{\Z/p} \ar[r] & 0
 }
\]
Proposition~\ref{prop:extensions2} then tells us that
$\Mackey H_G^\alpha(S^0) \iso \conc\Z$ and the diagram takes the form
\[
 \xymatrix{
 0 \ar[r] & \conc\Z \ar[d]_1 \ar[r] & \Mackey A[\nu(\beta)] \ar[d] \ar[r]
  & \Mackey R\Z \ar[d]^{\pi} \ar[r] & 0 \\
 0 \ar[r] & \conc\Z \ar[r]_p & \conc\Z \ar[r]_{\nu(\beta)}
  & \conc{\Z/p} \ar[r] & 0
 }
\]
where $e_1^n\mu^{\beta,b}$ generates the bottom middle group, $e_1^n\tau(\iota^\beta) = 0$,
and $e_1^n\mu^{\beta,b}$ maps to $b e_1^n\xi^\beta \in \Mackey H_G^\alpha(EG_+)$.

If $\alpha^G < 0$ is even, both $\Mackey H_G^\alpha(\tE G) = 0$ and
$\Mackey H_G^{\alpha+1}(\tE G) = 0$, so the long exact sequence gives the isomorphism
$\Mackey H_G^\alpha(S^0) \iso \Mackey H_G^\alpha(EG_+)$.
If $\alpha^G<0$ is odd, then $\Mackey H_G^\alpha(\tE G) = 0$ and
$\Mackey H_G^\alpha(EG_+) = 0$, so we get the same isomorphism (with both being 0).
From the known structure of $\Mackey H_G^\bullet(EG_+)$, this gives us the Mackey functors claimed
in the theorem with $\alpha^G<0$.
To identify the generators, 
first let $\xi_1\in \Mackey H_G^{\MM_1-2}(S^0)\iso \Mackey H_G^{\MM_1-2}(EG_+)$ 
be the element corresponding to the element of the same name under this isomorphism.
For a general $\alpha$ with $\alpha^G<0$, let $\alpha = a + b\MM_1 + \beta$
where $\beta\in RO_0(G)$.
Let $d\in\nu(\beta)$ such that $\mu^{\beta,d}$ is defined and let $d^{-1}$ be an integer
such that $dd^{-1}\equiv 1 \pmod p$.
Then let
\[
 \lambda^{\beta,d^{-1}} = d^{-1}\mu^{\beta,d} + \frac{1-dd^{-1}}{p}\tau(\iota^\beta)
  \in \Mackey H_G^\beta(S^0).
\]
From our calculations in the first case above,
$\lambda^{\beta,d^{-1}}$ maps to $\xi^\beta$ in $\Mackey H_G^\beta(EG_+)$.
From this it follows,
for $m\geq 0$ and $n \geq 1$, that $e_1^m\lambda^{\beta,d^{-1}}\xi_1^n \in\Mackey H_G^\bullet(S^0)$
maps to the generator $e_1^m\xi^\beta\xi_1^n\in \Mackey H_G^\bullet(EG_+)$,
hence generates the Mackey functor in which it lives.

If $|\alpha|<0$, then $\Mackey H_G^\alpha(EG_+) = 0$, which implies that
$\Mackey H_G^\alpha(\tE G)\to \Mackey H_G^\alpha(S^0)$ is an isomorphism for all such $\alpha$.
This gives us the Mackey functors claimed in the theorem with $|\alpha|<0$.
Because $\mu^{\alpha,a}$ and $e^\alpha$ have identical action on $\Mackey H_G^\bullet(\tE G)$,
we may use the same names for the generators of these groups
as in Theorem~\ref{thm:tEGCohomology}, replacing
$e^\alpha$ with $\mu^{\alpha,a}$.

If $\alpha^G>0$, the preceding proposition and the known structures
of $\Mackey H_G^\bullet(EG_+)$ and $\Mackey H_G^\bullet(\tE G)$ show that
$\delta\colon\Mackey H_G^\alpha(EG_+) \to \Mackey H_G^{\alpha+1}(\tE G)$
is an isomorphism if $|\alpha|>0$ and an epimorphism if $|\alpha|=0$. 
This implies that $\Mackey H_G^\alpha(S^0) = 0$ 
if $|\alpha|>0$ and $\alpha^G>0$.

With $|\alpha| = 0$ and $\alpha^G>0$, combining the fact that $\delta$
is onto with the fact that $\Mackey H_G^\alpha(\tE G) = 0$,
we get the following short exact sequence:
\[
 0 \to \Mackey H_G^\alpha(S^0) \to \Mackey R\Z \xrightarrow{\pi} \conc{\Z/p} \to 0,
\]
which implies that $\Mackey H_G^\alpha(S^0) \iso \Mackey L\Z$, the kernel of
the projection $\pi$.
Because $\Mackey L\Z$ is generated as a Mackey functor by a generator of its group at level $G/e$,
and an appropriate $\iota^\beta\iota_1^{-n}$ generates at that level,
we have shown the claims of the theorem.
\end{proof}

The following allows us to extend the notation $\mu^{\alpha,a}$ to any $a\in\nu(\alpha)$,
and also shows that these elements are well defined.

\begin{corollary}\label{cor:muBetaD}
Let $\alpha \in RO_0(G)$. 
For every integer
$a\in\nu(\alpha)$,
there is a unique element $\mu^{\alpha,a}\in \Mackey H_G^\alpha(S^0)$ such that
$\mu^{\alpha,a}$ and $\iota^\alpha$ generate $\Mackey H_G^\alpha(S^0)$ and
$\rho(\mu^{\alpha,a}) = a\iota^\alpha$.
\end{corollary}

\begin{proof}
When discussing the generators and relations of $\Mackey A[a]$
in Section~\ref{sec:MackeyFunctors}, we pointed out that, if given generators
$\mu$ and $\iota$ with $\rho(\mu) = a\iota$, then the other possible generators
are given by the elements $\pm\mu + q\tau(\iota)$, with
$\rho(\pm\mu+q\tau(\iota)) = \pm a + qp$.
The corollary follows.
\end{proof}

We can check that Proposition~\ref{prop:multByMu}
holds as written for all of the elements $\mu^{\alpha,a}$.

As pointed out when discussing generators for $\Mackey A[a]$, the element
\[
 \kappa^\alpha = \kappa\mu^{\alpha,a} = p\mu^{\alpha,a} - a\tau(\iota^\alpha)
\]
is independent of the choice of $a$ in its congruence class modulo $p$. 
Extend the definition of $\lambda^{\alpha,a}$, $a\in\nu(\alpha)^{-1}$, by setting
\[
 \lambda^{\alpha,a} = a\mu^{\alpha,a^{-1}} + \frac{1-aa^{-1}}{p}\tau(\iota^{\alpha}),
\]
for any integer $a^{-1}\in\nu(\alpha)$.
Then $\lambda^{\alpha,a}$ depends only on $a$, 
that is, it is independent of the choice of $a^{-1}$ in its congruence class,
and is characterized by the fact
that $\{\lambda^{\alpha,a},\kappa^{\alpha}\}$ is another set of generators of $\Mackey H_G^\alpha(S^0)$,
$\rho(\lambda^{\alpha,a}) = \iota^\alpha$, and
$\kappa\lambda^{\alpha,a} = a\kappa^{\alpha}$.

We now verify the multiplicative part of Theorem~\ref{thm:oddCohomPoint}.

\begin{theorem}\label{thm:oddMultiplicativeCohomoPoint}
For $p$ odd, $\Mackey H_G^\bullet(S^0)$ is a strictly commutative unital $RO(G)$-graded ring,
generated multiplicatively by elements
\begin{align*}
 \iota_k &\in \Mackey H_G^{\MM_k-2}(S^0)(G/e) & & 1\leq k \leq (p-1)/2 \\
 \iota_k^{-1} &\in \Mackey H_G^{2-\MM_k}(S^0)(G/e) & & 1\leq k \leq (p-1)/2 \\
 \xi_1 &\in \Mackey H_G^{\MM_1-2}(S^0)(G/G) \\
 e_1 &\in \Mackey H_G^{\MM_1}(S^0)(G/G) \\
 \mu^{\alpha,a} &\in \Mackey H_G^\alpha(S^0)(G/G) & & \alpha\in RO_0(G),\ a\in\nu(\alpha) \\
 e_1^{-m}\kappa &\in \Mackey H_G^{-m\MM_1}(S^0)(G/G) & & m\geq 1 \\
 e_1^{-m}\delta\xi_1^{-n} &\in \Mackey H_G^{1-m\MM_1 - n(\MM_1-2)}(S^0)(G/G)
  & & m, n \geq 1
\end{align*}
These generators satisfy the following {\em structural} relations:
\begin{align*}
 t\iota_k^{-1} &= \iota_k^{-1} & & \text{ for all $k$} \\
 \kappa\xi_1 &= 0 \\
 \rho(\xi_1) &= \iota_1 \\
 \rho(e_1) &= 0 \\
 \rho(\mu^{\alpha,a}) &= a \iota^\alpha 
     & & \text{ for all $\alpha\in RO_0(G)$ and $a\in\nu(\alpha)$} \\
 \rho(e_1^{-m}\kappa) &= 0 & & \text{ for $m\geq 1$} \\
 \rho(e_1^{-m}\delta\xi_1^{-n}) & = 0 & & \text{ for $m, n \geq 1$} \\
 p e_1^{-m}\delta\xi_1^{-n} &= 0 & & \text{ for $m, n \geq 1$,}
\intertext{the {\em redundancy} relations}
 \mu^{0,1} &= 1 \\
 \mu^{\alpha, a + p} &= \mu^{\alpha,a} + \tau(\iota^\alpha)
  & & \text{for all $\alpha\in RO_0(G)$ and $a\in\nu(\alpha)$}
\intertext{and the following {\em multiplicative} relations:}
 \iota_k \cdot \iota_k^{-1} &= \rho(1) & & \text{for all $k$} \\
 \mu^{\alpha,a}\cdot \mu^{\beta,b} &= \mu^{\alpha+\beta,ab}
      & & \text{for $\alpha, \beta \in RO_0(G)$} \\
 e_1\cdot e_1^{-m}\kappa &= e_1^{-m+1}\kappa & & \text{for $m\geq 1$} \\
 \xi_1\cdot e_1^{-m}\kappa &= 0 & & \text{for $m\geq 1$} \\
 e_1 \cdot e_1^{-m}\delta\xi_1^{-n} &= e_1^{-m+1}\delta\xi_1^{-n} & & \text{for $m\geq 2$ and $n\geq 1$} \\
 e_1 \cdot e_1^{-1}\delta\xi_1^{-n} &= 0 & & \text{for $n\geq 1$} \\
 \xi_1\cdot e_1^{-m}\delta\xi_1^{-n} &= e_1^{-m}\delta\xi_1^{-n+1} 
     & & \text{for $m\geq 1$ and $n\geq 2$} \\
 \xi_1\cdot e_1^{-m}\delta\xi_1^{-1} &= 0 & & \text{for $m\geq 1$} \\
 e_1^{-m}\kappa \cdot e_1^{-n}\kappa &= p e_1^{-m-n}\kappa & & \text{for $m, n \geq 1$} \\
 e_1^{-\ell}\kappa\cdot e_1^{-m}\delta\xi_1^{-n} &= 0 & & \text{for $\ell, m \geq 1$} \\
 e_1^{-k}\delta\xi_1^{-\ell} \cdot e_1^{-m}\delta\xi_1^{-n} &= 0
     & & \text{for $k, \ell, m, n \geq 1$}
\end{align*}
For $\alpha\in RO_0(G)$ and $a\in \nu(\alpha)^{-1}$, let $a^{-1}\in\nu(\alpha)$ so
that $aa^{-1} \equiv 1 \pmod p$ and let
\[
 \lambda^{\alpha,a} = a\mu^{\alpha,a^{-1}} + \frac{1-aa^{-1}}{p}\tau(\iota^\alpha);
\]
this is independent of the choice of $a^{-1}$.
The following relations are implied by the preceding ones:
\begin{align*}
 \xi_1 \tau(\iota^\alpha) &= \tau(\iota_1\iota^\alpha)
  && \text{for $\alpha^G = 0$} \\
 \mu^{\alpha,a}\tau(\iota^\beta) &= a\tau(\iota^{\alpha+\beta}) 
  && \text{for $\alpha\in RO_0(G)$ and $\beta^G = 0$} \\
 \tau(\iota^\alpha)\tau(\iota^\beta) &= p\tau(\iota^{\alpha+\beta}) 
  && \text{for $\alpha^G = 0 = \beta^G$} \\
 e_1\tau(\iota^\alpha) &= 0 
  && \text{for $\alpha^G = 0$}\\
 \kappa e_1 &= pe_1 \\
 pe_1\xi_1 &= 0 \\
 \mu^{\alpha,a} &= \mathrlap{a\lambda^{\alpha,a^{-1}} + \frac{1-aa^{-1}}{p}\kappa\mu^{\alpha,a}} \\
  & && \text{for $\alpha\in RO_0(G)$, $a\in\nu(\alpha)$, $a^{-1}\in\nu(\alpha)^{-1}$} \\
 \lambda^{\alpha, a + p} &= \lambda^{\alpha,a} + \kappa\mu^{\alpha,a^{-1}} 
  && \text{for $\alpha\in RO_0(G)$, $a\in\nu(\alpha)^{-1}$, $a^{-1}\in\nu(\alpha)$} \\
 \lambda^{\alpha, a + p}\xi_1 &= \lambda^{\alpha,a}\xi_1 
  && \text{for $\alpha\in RO_0(G)$} \\
 e_1\lambda^{\alpha,a} &= a e_1\mu^{\alpha,a^{-1}}
  && \text{for $\alpha\in RO_0(G)$, $a\in\nu(\alpha)^{-1}$, $a^{-1}\in\nu(\alpha)$} \\
 \mu^{\alpha,a}\xi_1 &= a\lambda^{\alpha,a^{-1}}\xi_1
  && \text{for $\alpha\in RO_0(G)$, $a\in\nu(\alpha)$, $a^{-1}\in\nu(\alpha)^{-1}$} \\
 \rho(\lambda^{\alpha,a}) &= \iota^\alpha  
  && \text{for $\alpha\in RO_0(G)$} \\
 \lambda^{\alpha,a}  \lambda^{\beta,b} &= \lambda^{\alpha+\beta, ab} 
  && \text{for $\alpha, \beta \in RO_0(G)$}
\end{align*}
\end{theorem}

\begin{proof}
The additive calculation in Theorem~\ref{thm:oddAdditiveCohomoPoint}
shows that the elements listed generate $\Mackey H_G^\bullet(S^0)$ multiplicatively.
We need to verify that the relations listed hold, and that 
the listed structural, redundancy, and multiplicative relations suffice
to determine all relations.

The structural relations in the theorem follow from the additive calculation
and the discussion of generators and relations for Mackey functors
in \S\ref{sec:MackeyFunctors}.
The redundancy relations come from the proof of
Corollary~\ref{cor:muBetaD}.

We have already noted that $\iota_k$ is invertible, in the sense that we have
an element $\iota_k^{-1}$ such that $\iota_k\cdot \iota_k^{-1} = \rho(1)$, so that
relation has already been verified.

The relation $e_1\cdot \tau(\iota^\beta) = 0$ mentioned in the last group of relations
is a consequence of the Frobenius relations:
\[
 e_1\cdot \tau(\iota^\beta) = \tau(\rho(e_1)\iota^\beta) = 0.
\]
We mention it now because we will use it several times in what follows.

To calculate $\mu^{\alpha,a}\cdot\mu^{\beta,b}$ we do the following.
Because $\rho$ is multiplicative,
\[
 \rho(\mu^{\alpha,a}\mu^{\beta,b}) = ab\iota^{\alpha+\beta}.
\]
As in the proof of Theorem~\ref{thm:oddAdditiveCohomoPoint}, we note
that the effect of multiplication by $\mu^{\alpha,a}\mu^{\beta,b}$ on
$\Mackey H_G^\bullet(\tE G)$ is the same as multiplication by $e^\alpha e^\beta = e^{\alpha+\beta}$,
so takes the generator $\kappa\in \Mackey H_G^0(\tE G)$ to the generator
$e^{\alpha+\beta}\kappa\in \Mackey H_G^{\alpha+\beta}(\tE G)$.
On the other hand, multiplication by $\mu^{\alpha,a}\mu^{\beta,b}$ on
$\Mackey H_G^\bullet(EG_+)$ is the same as multiplication by 
$(a\xi^\alpha)(b\xi^\beta) = ab\xi^{\alpha+\beta}$, so takes
$1\in \Mackey H_G^0(EG_+)$ to $ab$ times the generator 
$\xi^{\alpha+\beta}\in \Mackey H_G^{\alpha+\beta}(EG_+)$. Thus, we have
the following map of short exact sequences:
\[
 \xymatrix{
 & & \Mackey H_G^0(S^0) \ar@{=}[d] \\
 0 \ar[r] & \conc\Z \ar[d]_1 \ar[r] & \Mackey A_{G/G} \ar[d]_{\mu^{\alpha,a}\mu^{\beta,b}} \ar[r]
  & \Mackey R\Z \ar[d]^{ab} \ar[r] & 0 \\
 0 \ar[r] & \conc\Z \ar[r] & \Mackey A[ab] \ar[r]
  & \Mackey R\Z \ar[r] & 0 \\
 & & \Mackey H_G^{\alpha+\beta}(S^0) \ar@{=}[u]
 }
\]
This implies that $\mu^{\alpha,a}\mu^{\beta,b}$ and $\iota^{\alpha+\beta}$ generate
$\Mackey H_G^{\alpha+\beta}(S^0)$ with $\rho(\mu^{\alpha,a}\mu^{\beta,b}) = ab\iota^{\alpha+\beta}$.
But, this characterizes the element $\mu^{\alpha+\beta,ab}$, so we must have
$\mu^{\alpha,a}\mu^{\beta,b} = \mu^{\alpha+\beta,ab}$ as claimed.

The relation $e_1\cdot e_1^{-m}\kappa = e_1^{-m+1}\kappa$ ($m\geq 1$) follows from
the same identity in the cohomology of $\tE G$.
The identity $\xi_1\cdot e_1^{-m}\kappa = 0$ ($m\geq 1$) follows because the group
in which the product would live is 0.

The relation $e_1\cdot e_1^{-m}\delta\xi_1^{-n} = e_1^{-m+1}\delta\xi_1^{-n}$ ($m\geq 2$ and $n\geq 1$)
follows from the same identity in the cohomology of $\tE G$.

For $n\geq 2$, we have 
$\xi_1\cdot e_1^{-m}\delta\xi_1^{-n} = e_1^{-m}\delta(\xi_1\cdot\xi_1^{-n}) = e_1^{-m}\delta\xi_1^{-n+1}$.
On the other hand, $\xi_1\cdot e_1^{-m}\delta\xi_1^{-1} = 0$ for $m\geq 2$ because the product
lives in a 0 group.

To determine the product $e_1^{-m}\kappa \cdot e_1^{-n}\kappa$, multiply by $e_1^{m+n}$:
\[
 e_1^{m+n}(e_1^{-m}\kappa \cdot e_1^{-n}\kappa) = \kappa^2 = p\kappa,
\]
hence $e_1^{-m}\kappa \cdot e_1^{-n}\kappa = pe_1^{-m-n}\kappa$. Similarly,
\[
 e_1^{\ell}(e_1^{-\ell}\kappa \cdot e_1^{-m}\delta\xi_1^{-n})
  = \kappa e_1^{-m}\delta\xi_1^{-n} = 0
\]
because both $pe_1^{-m}\delta\xi_1^{-n} = 0$ and $ge_1^{-m}\delta\xi_1^{-n} = 0$;
hence $e_1^{-\ell}\kappa \cdot e_1^{-m}\delta\xi_1^{-n} = 0$.

Finally for the basic relations, 
$e_1^{-k}\delta\xi_1^{-\ell} \cdot e_1^{-m}\delta\xi_1^{-n} = 0$
because the product would live in a 0 group.

We now verify that the other relations listed in the theorem are implied by those
already shown.

The first several follow from the Frobenius relations (the fourth of these we noted above):
\begin{align*}
  \xi_1\tau(\iota^\beta) &= \tau(\rho(\xi_1)\iota^\beta) = \tau(\iota_1\iota^\beta) \\
  \mu^{\alpha,a}\tau(\iota^\beta) &= \tau(\rho(\mu^{\alpha,a})\iota^\beta)
     = \tau(a\iota^\alpha\iota^\beta) = a\tau(\iota^{\alpha+\beta}) \\
  \tau(\iota^\alpha)\tau(\iota^\beta) &= \tau(\rho\tau(\iota^\alpha)\iota^\beta)
     = \tau(p\iota^\alpha\iota^\beta) = p\tau(\iota^{\alpha+\beta}) \\
  e_1\tau(\iota^\beta) &= \tau(\rho(e_1)\iota^\beta) = 0.
\end{align*}

The calculation $\kappa e_1 = pe_1$ follows from the fact that
$ge_1 = \tau\rho(e_1) = 0$. We then have
\[
 pe_1\xi_1 = \kappa e_1\xi_1 = e_1(\kappa \xi_1) = 0.
\]

The formula $\mu^{\alpha,a} = a\lambda^{\alpha,a^{-1}} + ((1-aa^{-1})/p)\kappa\mu^{\alpha,a}$
is part of the change of basis from $\{\mu^{\alpha,a},\tau(\iota^\alpha)\}$
to $\{\lambda^{\alpha,a^{-1}}, \kappa\mu^{\alpha,a}\}$;
see the discussion of generators and relations in \S\ref{sec:MackeyFunctors}.
Explicitly:
\begin{align*}
 a\lambda^{\alpha,a^{-1}} &{}+ \frac{1-aa^{-1}}{p}\kappa\mu^{\alpha,a} \\
  &= a\left(a^{-1}\mu^{\alpha,a} + \frac{1-aa^{-1}}{p}\tau(\iota^\alpha)\right)
     + \frac{1-aa^{-1}}{p}(p\mu^{\alpha,a} - a\tau(\iota^\alpha)) \\
  &= \mu^{\alpha,a}.
\end{align*}

For the relation $\lambda^{\alpha, a + p} = \lambda^{\alpha,a} + \kappa\mu^{\alpha,a^{-1}}$,
we have
\begin{align*}
 \lambda^{\alpha, a + p}
  &= (a+p)\mu^{\alpha,a^{-1}} + \frac{1-(a+p)a^{-1}}{p} \tau(\iota^\alpha) \\
  &= \left(a\mu^{\alpha,a^{-1}} + \frac{1-aa^{-1}}{p}\tau(\iota^\alpha)\right)
      + (p\mu^{\alpha,a^{-1}} - a^{-1}\tau(\iota^\alpha)) \\
  &= \lambda^{\alpha,a} + \kappa\mu^{\alpha,a^{-1}}.
\end{align*}
The relation $\lambda^{\alpha,a + p}\xi_1 = \lambda^{\alpha,a}\xi_1$ follows from this relation
and the structural relation $\kappa\xi_1 = 0$.

The relation $e_1\lambda^{\alpha,a} = ae_1\mu^{\alpha,a^{-1}}$ follows from the definition
of $\lambda^{\alpha,a}$ and the vanishing of $e_1\tau(\iota^\alpha)$.
That $\mu^{\alpha,a}\xi_1 = a\lambda^{\alpha,a^{-1}}\xi_1$ follows similarly from the formula
for $\mu^{\alpha,a}$ in terms of $\lambda^{\alpha,a^{-1}}$ and $\kappa\mu^{\alpha,a}$,
and the vanishing of $\kappa\xi_1$.

Because $\rho(\mu^{\alpha,a}) = a\iota^\alpha$ and
$\rho\tau(\iota^\alpha) = p\iota^\alpha$, we get
\begin{align*}
 \rho(\lambda^{\alpha,a}) 
   &= \rho\left(a\mu^{\alpha,a^{-1}} + \frac{1-aa^{-1}}{p}\tau(\iota^\alpha)\right) \\
   &= aa^{-1}\iota^\alpha + (1-aa^{-1})\iota^\alpha \\
   &= \iota^\alpha.
\end{align*}

The formula for the product $\lambda^{\alpha,a}\lambda^{\beta,b}$ follows directly from
the definition of $\lambda^{\alpha,a}$ and previous product formulas we've verified.
 
We can now check that all products of the generators have been computed,
so the relations listed suffice to determine the multiplicative structure
of $\Mackey H_G^\bullet(S^0)$.
Moreover, the only possibility for anticommutativity to introduce a sign
among products of generating elements is in the product
$e_1^{-k}\delta\xi_1^{-\ell} \cdot e_1^{-m}\delta\xi_1^{-n}$, but this product is 0,
hence $\Mackey H_G^\bullet(S^0)$ is strictly commutative.
\end{proof}

What about the other Euler classes?

\begin{proposition}
We have
\[
 e_k = \mu^{\MM_k-\MM_1,a} e_1, \quad 2\leq k \leq (p-1)/2,
\]
for any $a\in \nu(\MM_k-\MM_1) = [k]$. We also have
\[
 e_k = -\mu^{\MM_{p-k}-\MM_1,a} e_1, \quad (p+1)/2 \leq k \leq p-1,
\]
for $a\in [p-k] = -[k]$.
\end{proposition}

\begin{proof}
For $2\leq k\leq (p-1)/2$, this is a restatement of Proposition~\ref{prop:muTimesE},
combined with the fact that $e_1\tau(\iota^\alpha) = 0$ to show that
the choice of $a$ in its equivalence class is not important.

If $(p+1)/2\leq k \leq p-2$ and $a\equiv p-k \pmod p$, then
\[
 e_k = -e_{p-k} = -\mu^{\MM_{p-k}-\MM_1,a}e_1. \qedhere
\]
\end{proof}

Recall that we defined $\xi_1\in \Mackey H_G^{\MM_1-2}(S^0)$ as the element
corresponding to $\xi_1\in\Mackey H_G^{\MM_1-2}(EG_+)$, because these groups are isomorphic
in this grading. Similarly, we can define
\[
 \xi_k\in \Mackey H_G^{\MM_k-2}(S^0) \quad\text{for $1\leq k\leq p-1$}
\]
to correspond to the element of the same name in
$\Mackey H_G^{\MM_k-2}(EG_+)$.

\begin{proposition}
For $2 \leq k \leq p-1$,
\[
 \xi_k = \lambda^{\MM_k - \MM_1,a}\xi_1
\]
for any $a\in\nu(\MM_k-\MM_1)^{-1}$.
\end{proposition}

\begin{proof}
First consider the case $1\leq k \leq (p-1)/2$.
In the proof of Theorem~\ref{thm:oddAdditiveCohomoPoint} we pointed out that
$\lambda^{\alpha,a}$ maps to $\xi^\alpha$ in $\Mackey H_G^\alpha(EG_+)$.
Thus, $\lambda^{\MM_k-\MM_1,a}\xi_1$ maps to $\xi^{\MM_k-\MM_1}\xi_1 = \xi_k$.
This gives $\lambda^{\MM_k-\MM_1,a}\xi_1 = \xi_k$ by the definition of $\xi_k$.

For $(p+1)/2\leq k\leq p-1$, we have
\[
 \xi_k = \xi_{p-k} = \lambda^{\MM_{p-k}-\MM_1,a}\xi_1
  = \lambda^{\MM_k-\MM_1,a}\xi_1. \qedhere
\]
\end{proof}

\begin{remark}[Stable homotopy and the Picard group]
In \cite[2.1]{Le:Hurewicz}, Lewis showed that the Hurewicz map 
$\pi^{G,s}_*(S^0) \to H_*^G(S^0)$ is an isomorphism
in $RO_0(G)$ grading, where $\pi^{G,s}_*$ is stable homotopy. Thus, as pointed out in 
\cite{Le:Hurewicz} and \cite{Le:projectivespaces},
our computation of the ordinary (co)homology of a point
in this grading coincides with the computation of the stable homotopy groups
in the same grading given by tom Dieck and Petrie in \cite{tDP:geomodules}.
tom Dieck and Petrie showed that, for any compact Lie group,
these homotopy groups are projective $\pi^{G,s}_0(S^0) = A(G)$-modules
of rank 1, hence can be considered as elements of the Picard group
$\Pic(A(G))$.
Further, the pairings
$\pi^{G,s}_\alpha(S^0)\tensor_{A(G)} \pi^{G,s}_\beta(S^0) \to \pi^{G,s}_{\alpha+\beta}(S^0)$
are isomorphisms, so the assignment of $\pi^{G,s}_\alpha(S^0)$ to $\alpha$
defines a homomorphism $RO_0(G)\to \Pic(A(G))$.
tom Dieck and Petrie go on to give a method for computing this Picard group,
given in its most explicit form for finite abelian groups
by a formula at the bottom of p.\ 162 of \cite{tDP:homotopyrepns}.

Applying this formula to $G = \Z/p$, we see that
\[
 \Pic(A(\Z/p)) \iso (\Z/p)^\times/\{\pm 1\}.
\]
We can see from our computations, then, that the elements of $\Pic(A(\Z/p))$ are the $A(G)$-modules
$A[a]$, $[a]\in (\Z/p)^\times/\{\pm 1\}$, that occur at the $G/G$ level in our Mackey
functors $\Mackey A[a]$. Further, the homomorphism
$RO_0(\Z/p)\to \Pic(A(\Z/p))$ is an epimorphism and is given explicitly by
\[
 \alpha \mapsto [\nu(\alpha)]\in (\Z/p)^\times/\{\pm 1\}.
\]
\end{remark}

\section{The cohomology of a point for other coefficient systems}\label{sec:cohomPointOtherCoeffs}

So far, we have been using the coefficient system $\Mackey A_{G/G}$ for all our cohomology.
There are other systems of interest, particularly $\Mackey R\Z$, which is often referred to
as ``constant $\Z$ coefficients,'' because both $\Mackey R\Z(G/G)$
and $\Mackey R\Z(G/e)$ are $\Z$ and the restriction map
is the identity. So we describe here the cohomology of a point with $\Mackey R\Z$ coefficients,
as well as several others.

First, let $C$ be any abelian group and consider $\conc{C}$, the system with
$\conc{C}(G/G) = C$ and $\conc{C}(G/e) = 0$.

\begin{proposition}\label{prop:concCcohomology}
There is a natural isomorphism
$\tilde H_G^\alpha(X;\conc{C}) \iso \tilde H^{\alpha^G}(X^G;C)$, for any $\alpha\in RO(G)$.
\end{proposition}

\begin{proof}
This is a special case of \cite[1.13.22]{CW:ordinaryhomology}, or can be seen directly as follows:
$\tilde H^{\alpha^G}(X^G;C)$ is an $RO(G)$-graded cohomology theory in based $G$-spaces $X$.
It obeys a dimension axiom in integer grading, with
$\tilde H^n((G/G)^G_+;C) \iso C$ if $n=0$ but equal to $0$ if $n\neq 0$, and
$\tilde H^n((G/e)^G_+;C) = 0$ for all $n$.
This is precisely the dimension axiom satisfied by $\tilde H_G^\bullet(X;\conc{C})$,
so the two theories must be naturally isomorphic by the uniqueness of equivariant
ordinary cohomology.
\end{proof}

In particular, consider $\conc\Z$ and the short exact sequence
\[
 0 \to \conc\Z \xrightarrow{\kappa} \Mackey A_{G/G} \to \Mackey R\Z \to 0,
\]
where the map $\kappa$ takes $1$ to $\kappa\in A(G)$.
We use this to think of $\conc\Z$ as a submodule of $\Mackey A_{G/G}$, identifying
$\Z$ as the multiples of $\kappa$.

\begin{proposition}\label{prop:concZcohomologypoint}
Additively,
\[
 \Mackey H_G^\alpha(S^0;\conc\Z) \iso
  \begin{cases}
    \conc\Z & \text{if $\alpha^G = 0$} \\
    0 & \text{otherwise.}
  \end{cases}
\]
The map $\Mackey H_G^\bullet(S^0;\conc\Z) \to \Mackey H_G^\bullet(S^0;\Mackey A_{G/G})$
is injective and identifies $\Mackey H_G^\bullet(S^0;\conc\Z)$ as the
ideal
\[
 \Mackey H_G^\bullet(S^0;\conc\Z) \iso 
  \langle e^\alpha\kappa \mid \alpha^G=0 \rangle 
  \subset \Mackey H_G^0(S^0;\Mackey A_{G/G}).
\]
The map factors through
$\Mackey H_G^\bullet(\tE G;\Mackey A_{G/G})$, where it identifies
$\Mackey H_G^\bullet(S^0;\conc\Z)$ as the submodule of elements in gradings
$\alpha$ with $\alpha^G = 0$.
\end{proposition}

\begin{proof}
The additive calculation is immediate from Proposition~\ref{prop:concCcohomology}.
We also know that the map $\Mackey H_G^0(S^0;\conc\Z)\to \Mackey H_G^0(S^0;\Mackey A_{G/G})$
is the inclusion $\conc\Z\to \Mackey A_{G/G}$.
Write $\kappa\in \Mackey H_G^0(S^0;\conc\Z)$ for the generator that maps
to $\kappa\in \Mackey H_G^0(S^0;\Mackey A_{G/G})$.

If $p$ is odd,
consider the inclusion $S^0\to S^{\MM_i}$. Again by Proposition~\ref{prop:concCcohomology},
we have that the induced map
$\Mackey H_G^\bullet(S^{\MM_i};\conc\Z)\to \Mackey H_G^\bullet(S^0;\conc\Z)$ is an
isomorphism. But this map is multiplication by $e_i$, so multiplication by $e_i$
is an isomorphism on $\Mackey H_G^\bullet(S^0;\conc\Z)$. This implies that,
for $\alpha$ such that $\alpha^G = 0$, $\Mackey H_G^\alpha(S^0;\conc\Z)$
is generated by $e^\alpha\kappa$, an element that maps to
$e^\alpha\kappa\in H_G^\alpha(S^0;\Mackey A_{G/G})$.

Similarly, for $p=2$, considering the inclusion $S^0\to S^\Lambda$,
we see that multiplication by $e$ is an isomorphism and that,
for $\alpha^G = 0$, $\Mackey H_G^\alpha(S^0;\conc\Z)$ is generated by
an element $e^\alpha\kappa = e^{|\alpha|}\kappa$.

We can now see that, for any $p$,
the map $\Mackey H_G^\bullet(S^0;\conc\Z) \to \Mackey H_G^\bullet(S^0;\Mackey A_{G/G})$
is injective with image the ideal $\langle e^\alpha\kappa \mid \alpha^G=0 \rangle$.

For the last statement, Proposition~\ref{prop:concCcohomology} implies
that
\[
 \Mackey H_G^\bullet(S^0;\conc\Z) \iso \Mackey H_G^\bullet(\tE G;\conc\Z)
\]
because $\tE G^G \hmtpc S^0$. Thus, we get a factorization
\[
 \Mackey H_G^\bullet(S^0;\conc\Z) \iso \Mackey H_G^\bullet(\tE G;\conc\Z)
  \to \Mackey H_G^\bullet(\tE G;\Mackey A_{G/G})
  \to \Mackey H_G^\bullet(S^0;\Mackey A_{G/G}).
\]
The image in $\Mackey H_G^\bullet(\tE G;\Mackey A_{G/G})$ is clear from our computations.
\end{proof}

It's useful to notice that the image of
$\Mackey H_G^\bullet(S^0;\conc\Z) \to \Mackey H_G^\bullet(\tE G;\Mackey A_{G/G})$
is a direct summand. The other summand is given by the groups in gradings $\alpha$
with $\alpha^G\neq 0$; it's straightforward to check that this is a submodule.

\begin{theorem}\label{thm:pointEvenRZcoeffs}
Let $p=2$.
Additively,
\[
 \Mackey H_G^\alpha(S^0;\Mackey R\Z) \iso
  \begin{cases}
    \Mackey R\Z & \text{if $|\alpha| = 0$ and $\alpha^G \leq 0$ is even} \\
    \Mackey R\Z_- & \text{if $|\alpha| = 0$ and $\alpha^G \leq 1$ is odd} \\
    \Mackey L\Z & \text{if $|\alpha| = 0$ and $\alpha^G > 0$ is even} \\
    \Mackey L\Z_- & \text{if $|\alpha| = 0$ and $\alpha^G \geq 3$ is odd} \\
    \conc{\Z/2} & \text{if $|\alpha| > 0$ and $\alpha^G \leq 0$ is even} \\
    \conc{\Z/2} & \text{if $|\alpha| < 0$ and $\alpha^G \geq 3$ is odd} \\
    0 & \text{otherwise}.
  \end{cases}
\]
$\Mackey H_G^\bullet(S^0;\Mackey R\Z)$ is a strictly commutative $RO(G)$-graded 
algebra over $\Mackey R\Z$,
generated multiplicatively by elements
\begin{align*}
 \iota &\in \Mackey H_G^{\LL-1}(S^0;\Mackey R\Z)(G/e) \\
 \iota^{-1} &\in \Mackey H_G^{1-\LL}(S^0;\Mackey R\Z)(G/e) \\
 \xi &\in \Mackey H_G^{2(\LL-1)}(S^0;\Mackey R\Z)(G/G) \\
 e &\in \Mackey H_G^{\LL}(S^0;\Mackey R\Z)(G/G) \\
 e^{-m}\delta\xi^{-n} &\in \Mackey H_G^{1 - m\LL - 2n(\LL-1)}(S^0;\Mackey R\Z)(G/G)
   & & m, n \geq 1.
\end{align*}
These generators satisfy the following {\em structural} relations:
\begin{align*}
 \tau(\iota^{-1}) &= 0 \\
 \rho(\xi) &= \iota^2 \\
 \rho(e) &= 0 \\
 e^{-1}\delta\xi^{-n} &= \tau(\iota^{-2n-1}) & & \text{for $n\geq 1$} \\
 \rho(e^{-m}\delta\xi^{-n}) &= 0 & & \text{for $m\geq 2$ and $n\geq 1$}\\
\intertext{and the following {\em multiplicative} relations:}
  \iota\cdot \iota^{-1} &= \rho(1) \\
 e\cdot e^{-m}\delta\xi^{-n} &= e^{-m+1}\delta\xi^{-n} & &\text{for $m\geq 2$ and $n\geq 1$} \\
 \xi\cdot e^{-m}\delta\xi^{-n} &= e^{-m}\delta\xi^{-n+1} & & 
   \text{for $m\geq 1$ and $n\geq 2$} \\
 \xi\cdot e^{-m}\delta\xi^{-1} &= 0 & & \text{for $m\geq 2$} \\
 e^{-m}\delta\xi^{-k}\cdot e^{-n}\delta\xi^{-\ell} &= 0 & & 
   \text{if $m, n, k, \ell\geq 1$}
\end{align*}
The following relations are implied by the preceding ones:
\begin{align*}
 2e^m\xi^n &= 0 & & \text{if $m > 0$ and $n \geq 0$} \\
 2e^{-m}\delta\xi^{-n} &= 0 & & \text{if $m\geq 2$ and $n\geq 1$} \\
 t\iota^k &= (-1)^k\iota^k & & \text{for all $k$} \\
 \xi\cdot\tau(\iota^k) &= \tau(\iota^{k+2}) & &\text{for all $k$} \\
 e\cdot\tau(\iota^k) &= 0 & &\text{for all $k$} \\
 e^{-m}\delta\xi^{-n}\cdot \tau(\iota^k) &= 0 & & \text{for all $m, n\geq 1$ and $k$} \\
 \tau(\iota^k)\cdot\tau(\iota^\ell) &= 0 & &\text{if $k$ or $\ell$ is odd} \\
 \tau(\iota^{2k})\cdot\tau(\iota^{2\ell}) &= 2\tau(\iota^{2(k+\ell)}) & &\text{for all $k$ and $\ell$}\\
 \tau(\iota^{2k+1}) &= 0 & & \text{if $k\geq 0$} \\
 e\cdot e^{-1}\delta\xi^{-n} &= 0 & & \text{if $n\geq 1$}
\end{align*}
\end{theorem}

\begin{proof}
By Proposition~\ref{prop:concZcohomologypoint}, we have a short exact sequence
\[
 0 \to \Mackey H_G^\bullet(S^0;\conc\Z) \to \Mackey H_G^\bullet(S^0;\Mackey A_{G/G})
  \to \Mackey H_G^\bullet(S^0;\Mackey R\Z) \to 0
\]
exhibiting $\Mackey H_G^\bullet(S^0;\Mackey R\Z)$ as a quotient ring of
$\Mackey H_G^\bullet(S^0;\Mackey A_{G/G})$.
The additive calculation follows by noticing what is killed off in the quotient, which
are all of the $e^{-m}\kappa$ for $m\geq 0$ and the elements $2e^m$ for $m\geq 1$.
The rest is just seeing what relations are still necessary from those in
Theorem~\ref{thm:evenCohomPointProved}.
We also simplify a bit by noticing that, for an $\Mackey R\Z$-module
generated by an element $x$ at level $G/G$, $\rho(x) = 0$ implies that $2x = 0$.
\end{proof}

\begin{figure}
\[\def\objectstyle{\scriptstyle}
 \xymatrix@!0@R=4ex@C=2.5em{
  & & & & & & & & & & \\
  & \conc{\Z/2} &\cdot& \conc{\Z/2} &\cdot& \conc{\Z/2} &\cdot& \cdot &\cdot& \cdot & \cdot \\
  & \conc{\Z/2} &\cdot& \conc{\Z/2} &\cdot& \conc{\Z/2} &\cdot& \cdot &\cdot& \cdot & \cdot \\
  & \conc{\Z/2} &\cdot& \conc{\Z/2} &\cdot& \conc{\Z/2} &\cdot& \cdot &\cdot& \cdot & \cdot \\
  & \conc{\Z/2} &\cdot& \conc{\Z/2} &\cdot& \conc{\Z/2} &\cdot& \cdot &\cdot& \cdot & \cdot \\
  \ar@{-}'[r]'[rr]'[rrr]'[rrrr]'[rrrrr]'[rrrrrr]'[rrrrrrr]'[rrrrrrrr]'[rrrrrrrrr]'[rrrrrrrrrr][rrrrrrrrrrr]
   & \Mackey R\Z & \Mackey R\Z_{\mathrlap{-}} & \Mackey R\Z & \Mackey R\Z_{\mathrlap{-}} & \Mackey R\Z & \Mackey R\Z_{\mathrlap{-}} & \Mackey L\Z & \Mackey L\Z_{\mathrlap{-}} & \Mackey L\Z & \Mackey L\Z_{\mathrlap{-}} &\\
  & \cdot &\cdot& \cdot &\cdot& \cdot &\cdot& \cdot & \conc{\Z/2} & \cdot &  \conc{\Z/2}  \\
  & \cdot &\cdot& \cdot &\cdot& \cdot &\cdot& \cdot & \conc{\Z/2} & \cdot &  \conc{\Z/2}  \\
  & \cdot &\cdot& \cdot &\cdot& \cdot &\cdot& \cdot & \conc{\Z/2} & \cdot &  \conc{\Z/2}  \\
  & \cdot &\cdot& \cdot &\cdot& \cdot &\cdot& \cdot & \conc{\Z/2} & \cdot &  \conc{\Z/2}  \\
  & & & & & \ar@{-}'[u]'[uu]'[uuu]'[uuuu]'[uuuuu]'[uuuuuu]'[uuuuuuu]'[uuuuuuuu]'[uuuuuuuuu][uuuuuuuuuu]
 }
\]
\[\def\objectstyle{\scriptstyle}
 \xymatrix@!0@R=4ex@C=2.5em{
  & & & & & & & & & & \\
  & e^4\xi^2 &\cdot& e^4\xi &\cdot& e^4 &\cdot& \cdot &\cdot& \cdot & \cdot \\
  & e^3\xi^2 &\cdot& e^3\xi &\cdot& e^3 &\cdot& \cdot &\cdot& \cdot & \cdot \\
  & e^2\xi^2 &\cdot& e^2\xi &\cdot& e^2 &\cdot& \cdot &\cdot& \cdot & \cdot \\
  & e\xi^2 &\cdot& e\xi &\cdot& e &\cdot& \cdot &\cdot& \cdot & \cdot \\
  \ar@{-}'[r]'[rr]'[rrr]'[rrrr]'[rrrrr]'[rrrrrr]'[rrrrrrr]'[rrrrrrrr]'[rrrrrrrrr]'[rrrrrrrrrr][rrrrrrrrrrr]
   & \xi^2 & (\iota^3) & \xi & (\iota) & 1 & (\iota^{-1}) & (\iota^{-2}) & (\iota^{-3}) & (\iota^{-4}) & (\iota^{-5}) &\\
  & \cdot &\cdot& \cdot &\cdot& \cdot &\cdot& \cdot & e^{-2}\delta\xi^{-1} & \cdot &  e^{-2}\delta\xi^{-2}  \\
  & \cdot &\cdot& \cdot &\cdot& \cdot &\cdot& \cdot & e^{-3}\delta\xi^{-1} & \cdot &  e^{-3}\delta\xi^{-2}  \\
  & \cdot &\cdot& \cdot &\cdot& \cdot &\cdot& \cdot & e^{-4}\delta\xi^{-1} & \cdot &  e^{-4}\delta\xi^{-2}  \\
  & \cdot &\cdot& \cdot &\cdot& \cdot &\cdot& \cdot & e^{-5}\delta\xi^{-1} & \cdot &  e^{-5}\delta\xi^{-2}  \\
  & & & & & \ar@{-}'[u]'[uu]'[uuu]'[uuuu]'[uuuuu]'[uuuuuu]'[uuuuuuu]'[uuuuuuuu]'[uuuuuuuuu][uuuuuuuuuu]
 }
\]
\caption{$\protect\Mackey H_G^\bullet(S^0;\protect\Mackey R\Z)$ and its generators, $p=2$}\label{fig:EvenCohomPointRZ}
\end{figure}
Figure~\ref{fig:EvenCohomPointRZ} shows $\Mackey H_G^\bullet(S^0;\Mackey R\Z)$ and its
generators.
You can see something interesting in this figure.
Recall that the integer-graded part lies along the diagonal through the origin, from
southwest to northeast. As it must, because of the dimension axiom,
it contains only one nonzero group, which is $\Mackey R\Z$.
Now look at the diagonal one to the right. It also contains only one nonzero group,
which is $\Mackey R\Z_-$. So, if we were to shift the $\Mackey R\Z$ cohomology one to the left,
uniqueness shows that we get cohomology with coefficients in $\Mackey R\Z_-$.
This continues for two more diagonals and gives us the following result.

\begin{corollary}\label{cor:manyIsoCoeffs}
There are natural isomorphisms
\begin{multline*}
 \Mackey H_G^\alpha(X;\Mackey R\Z)
  \iso \Mackey H_G^{\alpha+(\Lambda-1)}(X;\Mackey R\Z_-) \\
  \iso \Mackey H_G^{\alpha+2(\Lambda-1)}(X;\Mackey L\Z)
  \iso \Mackey H_G^{\alpha+3(\Lambda-1)}(X;\Mackey L\Z_-).
\end{multline*}
Moreover, these are isomorphisms of modules over
$\Mackey H_G^\bullet(S^0;\Mackey R\Z)$.
\end{corollary}

\begin{proof}
The proof of the isomorphisms was given above. That all these theories are modules
over $\Mackey H_G^\bullet(S^0;\Mackey R\Z)$ follows from the fact that
$\Mackey R\Z_-$, $\Mackey L\Z$, and $\Mackey L\Z_-$ are all modules over $\Mackey R\Z$,
as noted in \S\ref{sec:MackeyFunctors}, or as exhibited in the fact that all
appear as groups within $\Mackey H_G^\bullet(S^0;\Mackey R\Z)$.
\end{proof}

Here are the similar results for $p$ odd.

\begin{theorem}\label{thm:pointOddRZcoeffs}
Let $p$ be odd.
Additively,
\[
 \Mackey H_G^{\alpha}(S^0;\Mackey R\Z) \iso
  \begin{cases}
   \Mackey R\Z & \text{if $|\alpha| = 0$ and $\alpha^G \leq 0$ is even} \\
   \Mackey L\Z & \text{if $|\alpha| = 0$ and $\alpha^G > 0$ is even} \\
   \conc{\Z/p} & \text{if $|\alpha| > 0$ and $\alpha^G \leq 0$ is even} \\
   \conc{\Z/p} & \text{if $|\alpha| < 0$ and $\alpha^G \geq 3$ is odd} \\
   0 & \text{otherwise.}
  \end{cases}
\]
$\Mackey H_G^\bullet(S^0;\Mackey R\Z)$ is a strictly commutative unital $RO(G)$-graded algebra over $\Mackey R\Z$,
generated by elements
\begin{align*}
 \iota_k &\in \Mackey H_G^{\MM_k-2}(S^0;\Mackey R\Z)(G/e) & & 1\leq k \leq (p-1)/2 \\
 \iota_k^{-1} &\in \Mackey H_G^{2-\MM_k}(S^0;\Mackey R\Z)(G/e) & & 1\leq k \leq (p-1)/2 \\
 \xi_1 &\in \Mackey H_G^{\MM_1-2}(S^0;\Mackey R\Z)(G/G) \\
 e_1 &\in \Mackey H_G^{\MM_1}(S^0;\Mackey R\Z)(G/G) \\
 \lambda^{\alpha} &\in \Mackey H_G^\alpha(S^0;\Mackey R\Z)(G/G) & & \alpha\in RO_0(G) \\
 e_1^{-m}\delta\xi_1^{-n} &\in \Mackey H_G^{1-m\MM_1 - n(\MM_1-2)}(S^0;\Mackey R\Z)(G/G)
  & & m, n \geq 1
\end{align*}
These generators satisfy the following {\em structural} relations:
\begin{align*}
 t\iota_k^{-1} &= \iota_k^{-1} & & \text{for all $k$} \\
 \rho(\xi_1) &= \iota_1 \\
 \rho(e_1) &= 0 \\
 \rho(\lambda^{\alpha}) &= \iota^\alpha 
     & & \text{for all $\alpha\in RO_0(G)$} \\
 \rho(e_1^{-m}\delta\xi_1^{-n}) & = 0 & & \text{for $m, n \geq 1$} \\
\intertext{and the following {\em multiplicative} relations:}
 \iota_k \cdot \iota_k^{-1} &= \rho(1) & & \text{for all $k$} \\
 \lambda^0 &= 1 \\
 \lambda^\alpha\cdot\lambda^\beta &= \lambda^{\alpha+\beta}
      & & \text{for $\alpha, \beta \in RO_0(G)$} \\
 e_1 \cdot e_1^{-m}\delta\xi_1^{-n} &= e_1^{-m+1}\delta\xi_1^{-n} & & \text{for $m\geq 2$ and $n\geq 1$} \\
 e_1 \cdot e_1^{-1}\delta\xi_1^{-n} &= 0 & & \text{for $n\geq 1$} \\
 \xi_1\cdot e_1^{-m}\delta\xi_1^{-n} &= e_1^{-m}\delta\xi_1^{-n+1} 
     & & \text{for $m\geq 1$ and $n\geq 2$} \\
 \xi_1\cdot e_1^{-m}\delta\xi_1^{-1} &= 0 & & \text{for $m\geq 1$} \\
 e_1^{-k}\delta\xi_1^{-\ell} \cdot e_1^{-m}\delta\xi_1^{-n} &= 0
     & & \text{for $k, \ell, m, n \geq 1$}
\end{align*}
The following relations are implied by the preceding ones:
\begin{align*}
 \xi_1 \tau(\iota^\alpha) &= \tau(\iota_1\iota^\alpha) && \text{for $\alpha^G = 0$}\\
 \lambda^\alpha\tau(\iota^\beta) &= \tau(\iota^{\alpha+\beta}) 
   && \text{for $\alpha\in RO_0(G)$ and $\beta^G = 0$}\\
 \tau(\iota^\alpha)\tau(\iota^\beta) &= p\tau(\iota^{\alpha+\beta}) 
   && \text{for $\alpha^G = 0 = \beta^G$} \\
 e_1\tau(\iota^\alpha) &= 0 \\
 pe_1 &= 0 \\
 p e_1^{-m}\delta\xi_1^{-n} &= 0 & & \text{for $m, n \geq 1$}
\end{align*}
\qed
\end{theorem}

\begin{proof}
We again consider the short exact sequence
\[
 0 \to \Mackey H_G^\bullet(S^0;\conc\Z) \to \Mackey H_G^\bullet(S^0;\Mackey A_{G/G})
  \to \Mackey H_G^\bullet(S^0;\Mackey R\Z) \to 0.
\]
Again, the effect is to kill all the elements $e^\alpha\kappa$.
Because $\lambda^{\alpha,a} - \lambda^{\alpha,a'}$ is a multiple of $\kappa$, this also
identifies $\lambda^{\alpha,a}$ with $\lambda^{\alpha,a'}$ when $a\in\nu(\alpha)^{-1}$ and
$a'\in\nu(\alpha)^{-1}$. Thus, we write the common image as $\lambda^\alpha$.
Note that the image of the $\mu^{\alpha,a}$ is $a\lambda^\alpha$, $a\in\nu(\alpha)$, so
it is $\lambda^\alpha$ that is much more useful in this context.
The rest of the argument is as in the case $p=2$.
\end{proof}

\begin{figure}
\[\def\objectstyle{\scriptstyle}
 \xymatrix@!0@R=4ex@C=2.5em{
  & & & & & & & & & & \\
  & \conc{\Z/p} & & \conc{\Z/p} & & \conc{\Z/p} & & \cdot & & \cdot &  \\
  &  &\cdot&  &\cdot&  &\cdot&  &\cdot&  & \cdot \\
  & \conc{\Z/p} & & \conc{\Z/p} & & \conc{\Z/p} & & \cdot & & \cdot &  \\
  &  &\cdot&  &\cdot&  &\cdot&  &\cdot&  & \cdot \\
  \ar@{-}'[r]'[rrr]'[rrrrr]'[rrrrrrr]'[rrrrrrrrr][rrrrrrrrrrr]
   & \Mackey R\Z &  & \Mackey R\Z &  & \Mackey R\Z &  & \Mackey L\Z &  & \Mackey L\Z &  &\\
  &  &\cdot&  &\cdot&  &\cdot&  & \conc{\Z/p} &  &  \conc{\Z/p}  \\
  & \cdot & & \cdot & & \cdot & & \cdot &  & \cdot &     \\
  &  &\cdot&  &\cdot&  &\cdot&  & \conc{\Z/p} &  &  \conc{\Z/p}  \\
  & \cdot & & \cdot & & \cdot & & \cdot &  & \cdot &     \\
  & & & & & \ar@{-}'[u]'[uuu]'[uuuuu]'[uuuuuuu]'[uuuuuuuuu][uuuuuuuuuu]
 }
\]
\[\def\objectstyle{\scriptstyle}
 \xymatrix@!0@R=4ex@C=2.5em{
  & & & & & & & & & & \\
  & \lambda^{\alpha}e_1^2\xi_1^2 & & \lambda^{\alpha}e_1^2\xi_1 & & \lambda^{\alpha}e_1^2 & & \cdot & & \cdot &  \\
  &  &\cdot&  &\cdot&  &\cdot&  &\cdot&  & \cdot \\
  & \lambda^{\alpha}e_1\xi_1^2 & & \lambda^{\alpha}e_1\xi_1 & & \lambda^{\alpha}e_1 & & \cdot & & \cdot &  \\
  &  &\cdot&  &\cdot&  &\cdot&  &\cdot&  & \cdot \\
  \ar@{-}'[r]'[rrr]'[rrrrr]'[rrrrrrr]'[rrrrrrrrr][rrrrrrrrrrr]
   & \lambda^{\alpha}\xi_1^2 &  & \lambda^{\alpha}\xi_1 &  & \lambda^{\alpha} &  & (\iota^\alpha\iota_1^{-1}) &  & (\iota^\alpha\iota_1^{-2}) &  &\\
  &  &\cdot&  &\cdot&  &\cdot&  & \lambda^{\alpha}e_1^{-1} \delta\xi_1^{-1} &  &  \lambda^{\alpha}e_1^{-1} \delta\xi_1^{-2}  \\
  & \cdot & & \cdot & & \cdot & & \cdot &  & \cdot &     \\
  &  &\cdot&  &\cdot&  &\cdot&  & \lambda^{\alpha}e_1^{-2} \delta\xi_1^{-1} &  &  \lambda^{\alpha}e_1^{-2} \delta\xi_1^{-2}  \\
  & \cdot & & \cdot & & \cdot & & \cdot &  & \cdot &     \\
  & & & & & \ar@{-}'[u]'[uuu]'[uuuuu]'[uuuuuuu]'[uuuuuuuuu][uuuuuuuuuu]
 }
\]
\caption{$\protect\Mackey H_G^\bullet(S^0;\protect\Mackey R\Z)$ and its generators, $p$ odd}\label{fig:OddCohomPointRZ}
\end{figure}
Figure~\ref{fig:OddCohomPointRZ} shows $\Mackey H_G^\bullet(S^0;\Mackey R\Z)$ and its generators.
Comparing this figure to Figure~\ref{fig:OddCohomPoint}, we have replaced some $\mu$'s with
$\lambda$'s. At the origin, this reflects the fact that $\lambda^\alpha$ generates the
copy of $\Mackey R\Z$ in that location. Elsewhere, the factor of $\mu^{\alpha,a}$ should become
$a\lambda^\alpha$, but the group in each case is $\Z/p$, where $a$ is invertible, hence we can
divide by $a$ to get a simpler generator.

Notice that the elements $\lambda^\alpha$, $\alpha\in RO_0(G)$, are invertible
in $\Mackey H_G^\bullet(S^0;\Mackey R\Z)$.
Hence $\Mackey H_G^\alpha(X;\Mackey R\Z) \iso \Mackey H_G^\beta(X;\Mackey R\Z)$ if
$|\alpha| = |\beta|$ and $\alpha^G = \beta^G$, the isomorphism given by multiplication
by $\lambda^{\beta-\alpha}$.
When using $\Mackey A_{G/G}$ coefficients, this is true only when
$\nu(\alpha-\beta) = \pm[1]$.

From the figure we can deduce the following analogue of
Corollary~\ref{cor:manyIsoCoeffs}.

\begin{corollary}\label{cor:manyOddIsoCoeffs}
There is a natural isomorphism
\[
 \Mackey H_G^\alpha(X;\Mackey R\Z) \iso \Mackey H_G^{\alpha+(\MM_1-2)}(X;\Mackey L\Z).
\]
Moreover, this is an isomorphism of modules over $\Mackey H_G^\bullet(S^0;\Mackey R\Z)$.
\qed
\end{corollary}

\clearpage	

\bibliographystyle{amsplain}
\bibliography{Topology}

\end{document}